\documentclass[11pt]{article}

\usepackage{amssymb,amsmath,amsfonts}
\usepackage{graphicx,color,enumitem}
\usepackage{amsthm} 
\usepackage{bm}
\usepackage[round]{natbib}

\usepackage{geometry}
\usepackage{stmaryrd}
\usepackage{subfig}
\usepackage{pifont}

 \usepackage{mathtools}
 \mathtoolsset{showonlyrefs}
	
\usepackage{bbm}
\usepackage{soul}

\usepackage[colorinlistoftodos, textwidth=4cm, shadow]{todonotes}
\RequirePackage[colorlinks,citecolor=blue,urlcolor=blue]{hyperref}

\newcommand{\E}{{\mathbb E}}
\newcommand{\F}{{\mathbb F}}

\renewcommand{\P}{{\mathbb P}}
\newcommand{\Q}{{\mathbb Q}}

\newcommand{\R}{{\mathbb R}}
\renewcommand{\S}{{\mathbb S}}
\newcommand{\N}{{\mathbb N}}

\newcommand{\Acal}{{\mathcal A}}

\newcommand{\Mid}{{\ \Big|\ }}

\newcommand{\vertiii}[1]{{\left\vert\kern-0.25ex\left\vert\kern-0.25ex\left\vert #1 
    \right\vert\kern-0.25ex\right\vert\kern-0.25ex\right\vert}}
\newcommand{\V}{\mathrm{Var}}

\newcommand{\Fc}{{\mathcal F}}

\newcommand{\id}{{\rm Id}}

\DeclareMathOperator{\tr}{tr}

\newtheorem{theorem}{Theorem}

\newtheorem{corollary}[theorem]{Corollary}
\newtheorem{definition}[theorem]{Definition}

\newtheorem{lemma}[theorem]{Lemma}

\newtheorem{remark}[theorem]{Remark}

\theoremstyle{definition}
\newtheorem{example}[theorem]{Example}

\numberwithin{equation}{section}
\numberwithin{theorem}{section}

\definecolor{darkgreen}{rgb}{0,0.7,0}

\DeclareMathOperator{\diag}{diag}

\newcommand{\iii}{{\vert\kern-0.25ex\vert\kern-0.25ex\vert}}

\newcommand{\bec}[1]{\begin{equation} \begin{cases} #1\end{cases} \end{equation}}
\newcommand{\bes}[1]{\begin{equation} \begin{split} #1\end{split} \end{equation}}
\renewcommand{\c}{\alpha}
\newcommand{\T}{\top}

\newcommand{\cmark}{\ding{51}}%
\newcommand{\xmark}{\ding{55}}%
\definecolor{blue0}{RGB}{0,77,153} 
\definecolor{red0}{RGB}{179,0,77} 
\definecolor{green0}{RGB}{134,219,76} 
\definecolor{gray0}{RGB}{84,97,110}

\begin{document}

\title{Markowitz portfolio selection for multivariate affine and quadratic Volterra models}
\author{Eduardo ABI JABER\footnote{Universit\'e Paris 1 Panth\'eon-Sorbonne, Centre d'Economie de la Sorbonne, 106, Boulevard de l'H\^opital, 75013 Paris, \sf  eduardo.abi-jaber at univ-paris1.fr.} \and Enzo MILLER\footnote{Universit\'e de Paris and Sorbonne Universit\'e, Laboratoire de Probabilit\'es, Statistique et Mod\'elisation (LPSM, UMR CNRS 8001), 
Building Sophie Germain, Avenue de France, 75013 Paris,  \sf  enzo.miller at polytechnique.org}  
\and Huy\^en PHAM\footnote{Universit\'e de Paris and Sorbonne Universit\'e, Laboratoire de Probabilit\'es, Statistique et Mod\'elisation (LPSM, UMR CNRS 8001),  
Building Sophie Germain, Avenue de France, 75013 Paris, \sf pham at lpsm.paris 
}
}

\maketitle

\begin{abstract}
This paper concerns portfolio selection with  multiple assets  under rough co\-variance matrix. 
We investigate  the continuous-time Markowitz mean-variance problem for a multivariate class of  affine and quadratic Volterra models.  
In this incomplete non-Markovian and non-semimartingale market framework with unbounded random coe\-fficients, 
the optimal portfolio strategy is expressed by means of  a Riccati backward stochastic differential equation (BSDE).  In the case of affine Volterra models, we derive explicit solutions to this BSDE in terms of  multi-dimensional Riccati-Volterra equations. This framework includes multivariate rough Hes\-ton models and extends the results of \cite{han2019mean}. In the quadratic case, we obtain new   analytic formulae for the the Riccati BSDE and we establish their  link with  infinite dimensional Riccati equations. This covers rough Stein-Stein and Wishart type  co\-variance models.  
Numerical results on a two dimensional rough Stein-Stein model  illustrate the impact of rough volatilities and stochastic correlations on the optimal Markowitz strategy. In particular for positively correlated assets, we find that the optimal strategy in our model is a   `buy rough sell smooth' one. 
\end{abstract}

\vspace{5mm}


\vspace{5mm}


\newpage


\section{Introduction}

The \cite{mar52} mean-variance portfolio selection problem  is the cornerstone of mo\-dern portfolio allocation theory. Investment decisions rules are made according to a trade-off  between  return and risk, and the use of Markowitz efficient portfolio strategies in the financial industry has become quite popular mainly due to its natural and intuitive formulation. A vast volume of research has been devoted over the last decades to extend Markowitz problem from static to continuous-time setting, first in Black-Scholes and complete markets (\cite{zhouli00}), and then to consider more general frameworks with random coefficients and multiple assets, see e.g. \cite{lim2004quadratic}, \cite{chiu2014mean}, or more recently 
\cite{ismpha19}  for taking into account model uncertainty on the assets correlation.

In the direction of more realistic modeling of asset prices,  it is now well-established  that volatility is rough \citep{gatetal18}, modeled by fractional Brownian motion with small Hurst parameter, which captures  empirical facts of times series of realized volatility and key features of  implied volatility surface, see \citet{alos2007short, fukasawa2011asymptotic}.  Subsequently, an important literature has focused on option pricing and asymptotics  in rough volatility models.  In comparison, the research on portfolio optimization in fractional and rough models is still little developed but has gained an increasing  attention with the recent papers of \cite{fouhu18,bau18,han2019merton}, which consider  fractional Ornstein-Uhlenbeck and Heston stochastic volatility models  for power utility function criterion, and the work by 
\cite{han2019mean}  where the authors study  the Markowitz problem  in a Volterra Heston model, which covers the  rough Heston model of \citet{EER:07}.

Most of the developments in rough volatility literature for asset modeling, option pricing or portfolio selection have been carried out in the mono-asset case. However, investment in multi-assets by taking into account the correlation risk is an importance feature in portfolio choice in financial markets, see \cite{buretal10}. 
Inspired by the recent papers \citet{jaber2019laplace,AJLP17,cuchiero2019markovian,rostho19}  that consider  multivariate versions of rough Volterra volatility models, the basic goal of this paper is to enrich the literature on portfolio selection:
\begin{itemize}
\item[(i)]  by introducing a  class of multivariate Volterra models, which captures stylized facts of financial assets, namely various rough volatility patterns across assets, (possibly random) correlation between stocks, and leverage effects, i.e., correlation between a stock and its volatility. 
\item[(ii)] by keeping  the model tractable for explicit computations of the optimal Markowitz portfolio strategy, which can be a quite challenging task in multivariate non-Markovian settings. 
\end{itemize}

\noindent {\bf Main contributions.} In this paper, we study the continuous-time Markowitz  problem in a multivariate setting with a focus on two classes: (i) affine Volterra models as in 
\cite{AJLP17} that include multivariate rough Heston models, (ii)  quadratic Volterra models, which are new class of Volterra models, and embrace multivariate rough Stein-Stein models, and 
rough Wishart type covariance matrix  models, in the spirit of  \citet{jaber2019laplace,cuchiero2019markovian}. 
We provide:
\begin{itemize}
\item  \textbf{A generic verification result} for the corresponding mean-variance problem, which is formulated in an incomplete non-Markovian and non-semimartingale framework with unbounded random coefficients of the volatility and market price of risk, and under general filtration. This result expresses the solution to the Markowitz problem in terms of a Riccati backward stochastic differential equation (BSDE) by checking in particular the admissibility condition  of the optimal control. We stress that related existing verification results   in the literature  
(see \cite{lim2004quadratic},  \cite{jeamansch12},  \cite{chiu2014mean}, \cite{shen2015mean}) 
cannot be applied directly  to our setting, and we shall 
discuss more in detail this point in Section \ref{secverif}.  
\item \textbf{Explicit solutions}  to the Riccati BSDE in two concrete specifications of multivariate Volterra models exploiting the representation of the solution in terms of a Laplace transform:
\begin{enumerate}
	\item 
\textbf{the affine case}: the optimal Markowitz strategy is expressed in terms of multivariate Riccati-Volterra equations which naturally extends the one obtained in \cite{han2019mean}.   
We point out that the martingale {distortion} arguments used in \cite{han2019mean} for  the univariate Volterra Heston model, do not  apply  in higher dimensions, unless the correlation structure is highly degenerate.
\item  \textbf{the quadratic case:} our major result is to derive analytic expressions for the optimal investment stra\-tegy  by  explicitly solving operator Riccati equations. This gives new explicit formulae for rough Stein-Stein and Wishart type co\-variance models. These analytic expressions can be efficiently implemented: the integral operators can be approximated by  closed form expressions involving finite dimensional matrices  and the underlying processes can be simulated by the celebrated Cholesky decomposition algorithm.
\end{enumerate}
\item \textbf{Numerical simulations} of the optimal Markowitz strategy in a two-asset rough Stein-Stein model to illustrate our results.\footnote{The code of our implementation  can be found at  the following \href{https://colab.research.google.com/drive/1P_SYE3WgFgwUKpOo8uCBDdIC04XyxE2a?usp=sharing}{link}.} {We depict the impact of some parameters onto the optimal investment when one asset is rough, and the other smooth (in the sense of the Hurst index of their volatility),  and show in particular that for positively correlated assets, the optimal strategy is to 
``buy rough, sell smooth", which is consistent with the empirical backtesting in \citet{glasserman2020buy}. }
\end{itemize} 
 
 \noindent {\bf Outline of the paper.} The rest of the paper is organized as follows: Section \ref{secform} formulates the financial market model and the mean-variance problem in a multivariate setting with random covariance matrix and market price of risk, and defines the general correlation structure.  We state in Section \ref{secverif} our generic verification result, which can be seen 
as unifying framework for previous results obtained in related literature. Section \ref{S:affine} is devoted to affine Volterra models {where} we derive {an explicit expression} for the optimal Markowitz strategy. In Section \ref{S:quadratic}, we consider the class of quadratic Volterra models, and we show how to solve the infinite-dimensional Riccati equations that appear in the closed-form  expressions of the optimal portfolio. Numerical illustrations on the behavior of the optimal investment in a two-asset rough Stein-Stein model are given in Section \ref{secnum}.  Finally, the proof of the verification result and other technical lemmas  are postponed to the Appendices.

 \vspace{2mm}

\noindent \textbf{Notations.}
Given a probability space $(\Omega,\Fc,\P)$ and a  filtration $\F$ $=$ $(\Fc_t)_{t \geq 0}$ satisfying the usual conditions, we denote by 
\bes{
     L^{\infty}_{\F}([0,T], \R^d) &= \left\{ Y:\Omega \times [0,T]\mapsto \R^d, \; \F-\text{prog.~measurable and bounded a.s.} \right\} \\
     L^p_{\F}([0,T], \R^d) &= \left\{ Y:\Omega \times [0,T]\mapsto \R^d, \; \F-\text{prog.~measurable s.t.~} \E\Big[ \int_0^T |Y_s|^p ds \Big] < \infty   \right\} \\
     {\S^{\infty}_{\F}([0,T], \R^d)} &= \left\{ Y:\Omega \times [0,T]\mapsto \R^d, \; \F-\text{prog.~measurable s.t.~} \sup_{t\leq T} |Y_t(w)|< \infty \mbox{ a.s.} \right\}. }
Here $|\cdot|$ denotes the Euclidian norm on $\R^d$.  Classically, for $p \in \llbracket1, \infty  \rrbracket$, we define $L^{p, loc}_{\F}([0,T], \R^d)$ as the set of progressive 
processes $Y$ for which there exists a sequence of increasing stopping times 
$\tau_n \uparrow \infty$ such that the stopped processes $Y^{\tau_n}$ are in $L^{p}_{\F}([0,T], \R^d)$ for every $n \geq 1$, and we recall that it consists of all progressive processes $Y$ s.t. 
$ \int_0^T |Y_t|^p dt$ $<$ $\infty$, a.s. To unclutter notation, we write $L^{p, loc}_{\F}([0,T])$  instead of $L^{p, loc}_{\F}([0,T], \R^d)$ when 
the context is clear.

\section{Formulation of the problem} \label{secform} 

Fix $T > 0$, $d,{N}\in \N$. We  consider a financial market on $[0,T]$  on some filtered probability space $(\Omega,\Fc,\F:=(\Fc_t)_{t \geq 0},\P)$ with a non--risky asset  $S^0$ 
\begin{align*}
dS^0_t = S^0_t r(t) dt,
\end{align*}
with a deterministic short rate $r:\R_+ \to \R$, and  $d$ risky assets with dynamics
\begin{align}
    \label{eq:stocks}
    dS_t = \diag(S_t) \big[ \big( r(t) {\bold{1}_d} + \sigma_t \lambda_t  \big)dt + \sigma_t dB_t \big],
\end{align}
driven by a $d$-dimensional Brownian motion $B$, with a $d\times d$-matrix valued  stochastic volatility process $\sigma$ and a $\R^d$-valued continuous stochastic process $\lambda$, 
called {\it market price of risk}.  Here ${\bold{1}_d}$ denotes the vector in $\R^d$ with all components equal to $1$. 
The market is typically incomplete, in the sense that the dynamics of the continuous volatility process $\sigma$ is driven by an $N$-dimensional process $W$ $=$ $(W^1,\ldots,W^N)^\T$ 
defined by:
\begin{align}\label{eq:correlstructuregen}
W^k_t =  C_k^\top  B_t + \sqrt{1-C_k^\T C_k} B^{\perp,k}_t, \quad k=1,\ldots,N,
\end{align}
where $C_k \in \R^{d}$ s.t. $C_k^\T C_k$ $\leq$ $1$,  and $B^{\perp}$ $=$ $(B^{\perp,1},\ldots,B^{\perp,N})^\T$ is an $N$--dimensional Brownian motion independent of $B$. 
Note that $d\langle W^k \rangle_t = dt$  but $W^k$ and $W^j$ can be correlated, hence $W$ is not necessarily a Brownian motion. Observe that 
processes $\lambda$ and $\sigma$ are $\F$-adapted, possibly unbounded,  but not necessarily adapted to the filtration generated by $W$.  We point out that $\F$ may be strictly larger than the augmented filtration generated by $B$ and $B^{\perp}$ as we shall deal with weak solutions to stochastic Volterra equations.

\begin{remark} \label{remlambdasigma} 
In our applications, we will be chiefly interested in the case where $\lambda_t$ is linear in $\sigma_t$, and where the dynamics of the matrix-valued process    
$\sigma$ is governed by a Volterra equation of the form 
\begin{align}
    \label{eq:SIGMA}
    \sigma_t = g_0(t)+ \int_0^t \mu(t,s,\omega) ds +  \int_0^t \chi(t,s,\omega) dW_s. 
\end{align}
The class of models that we shall develop in Sections \ref{S:affine} and \ref{S:quadratic} includes in particular the case of  Volterra Heston model when $d$ $=$ $1$ with 
$\lambda_t$ $=$ $\theta\sigma_t$, for some constant $\theta$, as studied in  \cite{han2019mean},  and the case of Wishart process for the covariance matrix process $V_t$ $=$ 
$\sigma_t\sigma_t^\top$, as studied in \cite{chiu2014mean}. 
The class of models that we will develop in Sections \ref{S:affine} and \ref{S:quadratic} includes  in particular the case of 
\begin{enumerate}
    \item 
    multivariate Volterra Heston models based on Volterra square-root processes, see \citet[Section 6]{AJLP17}, we refer to \citet{rostho19} for a microstuctural foundation. When $d$ $=$ $1$, we recover the results of  \cite{han2019mean}, which cover the case of the rough Heston model of \citet{EER:06}.
    \item
    multivariate Volterra Stein-Stein and Wishart  type in the sense of \citet{jaber2019laplace}, where the {instantaneous} covariance is given by squares of Gaussians. Under the Markovian setting, we recover a similar structure as in the results of  \cite{chiu2014mean}.
\end{enumerate}  
\end{remark}

\vspace{1mm}

\noindent \textbf{Mean-variance optimization problem.} 
Let $\pi_t$ denote the vector of the amounts invested in the risky assets $S$ at time $t$ in a self--financing strategy and set $\alpha = \sigma^\top \pi$. Then, the dynamics of the wealth $X^{\alpha}$ of the portfolio we seek to optimize is given by 
\begin{align}
\label{eq:wealth}
dX^{\alpha}_t &= \big( r(t) X^{\alpha}_t  + \alpha_t^\T \lambda_t \big) dt + \alpha_t^\T dB_t, \quad t \geq 0, \quad X_0^\alpha = x_0 \in \R. 
\end{align} 
By a solution to \eqref{eq:wealth}, we mean an $\F$-adapted continuous process $X^{\alpha}$ satisfying \eqref{eq:wealth} on $[0,T]$ $\P$-a.s. and such that 
\begin{align}
\label{eq:estimateX}
 \E\big[\sup_{t\leq T} |X^{\alpha}_t|^2 \big] &< \;  \infty.
\end{align}
The set of admissible investment strategies is naturally defined by
$$\mathcal A  = \{ \alpha \in {L^{2,loc}_{\F}([0,T], \R^d)} \mbox{ such that \eqref{eq:wealth} has a  solution satisfying } \eqref{eq:estimateX} \}.$$


The Markowitz portfolio selection problem in continuous-time consists in solving the following constrained problem
\begin{align} \label{optimization_problem}
     V(m) &  := \; \inf_{\substack{\c \in \Acal}}\big\{ \V (X_T):   \text{s.t. } \E[X_T] = m \big\}. 
\end{align}
given some expected return value $m$ $\in$ $\R$, where $ \V (X_T)$ $=$ $\E\big[\big( X_T-\E[ X_T ]\big)^2\big]$ stands for the variance. 

\section{A generic verification result} \label{secverif} 

In this section, we establish a generic verification result for the optimization problem \eqref{optimization_problem} given the solution of a certain Riccati BSDE. We stress that 
our mean-variance problem deals with  incomplete markets with unbounded random coefficients $\sigma$ and $\lambda$, so that existing results cannot be applied directly to our setting: 
\citet{lim2004quadratic} presents a general methodology to solve the MV problem for the wealth process \eqref{eq:wealth} in an incomplete market without assuming any particular dynamics on 
 $\sigma$  nor that the excess return is proportional to $\sigma$. However, a nondegeneracy assumption is made on $\sigma\sigma^\T$, see \citet[Assumption (A.1)]{lim2004quadratic}. 
The main verification  result in \citet[Proposition 3.3]{lim2004quadratic}, based on a completion of squares argument, states that if a solution to a certain (nonlinear) Riccati BSDE exists, then the MV is solvable. The difficulty resides in proving the existence of solutions to such nonlinear BSDEs (see also \citet{lim2002mean} for similar results  in  complete markets). \\

Here, we assume that the excess return is proportional to $\sigma$ (instead of the nondegeneracy condition) and state a verification result in terms of solutions of Riccati BSDEs (completion of squares, ie LQ problem with random coefficients).   A verification result depending on the solution of a Riccati BSDE  is also stated in \cite{chiu2014mean}, but the admissibility of the optimal candidate control is not proved. We also mention the paper of \cite{jeamansch12} where the authors adopt a BSDE approach for general semimartingales, but focusing on situations in which the existence of an optimal strategy is assumed.  In our case, the existence of an admissible optimal control is obtained under a suitable exponential integrability assumption involving the market price of risk and the $Z$ components of the BSDE, which extends the condition in \cite{shen2015mean}.

Our main result of this section, Theorem~\ref{T:verif} below, can be seen as unifying framework for the aforementioned results, refer to Table \ref{tablecomparison}. For the sake of presentation, we postpone its proof to Appendix~\ref{A:verifresult}.

{\begin{table}[h!]
	\centering  
	\resizebox{\textwidth}{!}{\begin{tabular}{c c  c c c} 
		\hline
		 &   Random coef. & Unbounded coef. & degenerate $\sigma$ & Incomplete market \\
		\hline \hline                 
\citet{lim2002mean} & \textcolor{blue0}{\cmark} & \textcolor{red0}{\xmark} & \textcolor{red0}{\xmark}   & \textcolor{red0}{\xmark}  \\
	\citet{lim2004quadratic} & \textcolor{blue0}{\cmark} & \textcolor{red0}{\xmark} & \textcolor{red0}{\xmark}   & \textcolor{blue0}{\cmark}  \\
\citet{shen2015mean} 	 & \textcolor{blue0}{\cmark} & \textcolor{blue0}{\cmark} & \textcolor{red0}{\xmark}   & \textcolor{red0}{\xmark}  \\
Theorem~\ref{T:verif} & \textcolor{blue0}{\cmark} & \textcolor{blue0}{\cmark} & \textcolor{blue0}{\cmark}   & \textcolor{blue0}{\cmark}  \\
			\hline 
		\end{tabular}}
	\caption{Comparison to existing verification results for mean-variance problems. }
	\label{tablecomparison} 
\end{table}}

We define $C\in \R^{N\times d}$  by
\bes{C &= \left(C_1, \ldots, C_N \right)^\T, \label{eq:CR}}
where we recall that the vectors $C_i \in \R^d$ come from the correlation structure \eqref{eq:correlstructuregen}.
We will use the matrix norm $|A| = \tr(A^\T A)$  in the subsequent theorem.

\begin{theorem}
\label{T:verif}
Assume that  there exists a solution triplet $ (\Gamma, Z^1, Z^2) \in \S^{\infty}_{\F}([0,T], \R)$   $\times L^{2, loc}_{\F}([0,T], \R^d) \times L^{2, loc}_{\F}([0,T], \R^N) $ to the  Riccati BSDE 
\begin{equation}
    \label{eq:riccati_sto}
 \left\{
 \begin{array}{ccl}   
  d \Gamma_t &=& \Gamma_t \Big[ \big(-2r(t) +   \left|\lambda_t +  Z^1_t + C Z^2_t \right|^2 \big)dt +\left(Z^1_t\right)^\T dB_t + \left(Z^2_t\right)^\T  dW_t \Big], \\
   \Gamma_T &=& 1,
 \end{array}
 \right.
\end{equation}   
such that 
\begin{enumerate}
    \item[(H1)] \label{T:verif:i}
    {$0<\Gamma_0 < e^{2 \int_0^T r(s) ds}$,} and  $\Gamma_t>0$, for all $t\leq T$,
\item[(H2)] \label{T:verif:ii}
  \bes{
    \label{eq:assumption_novikov}
        \E \Big[ \exp\Big( a(p)\int_0^T \big(  |\lambda_s|^2 + \left|Z^1_s\right|^2 + \left|Z^2_s\right|^2 \big)ds \Big) \Big] \;< \;  \infty,
    }
    for some $p > 2$ and a constant $a(p)$ given by 
   {  \bes{a(p)=\max & \Big[p \left(3 + |C| \right),   {3 (8p^2 {- 2p}) \left( 1  + |{C}|^2  \right)}  \Big]. \label{eq:constap}
    }}
      \end{enumerate}
Then, the optimal investment strategy for the Markowitz problem \eqref{optimization_problem} is  given by the admissible control 
\bes{
    \label{eq:optimal_control_final}
   \c^*_t \; = \;  - \big(\lambda_t +  Z^1_t + C Z^2_t\big) \big(X_t^{\alpha^*}  - \xi^* e^{-\int_t^T r(s) ds}\big), 
}
where  
\bes{
\label{eq:eta_star}
\xi^* \; = \;  \frac{ m - \Gamma_0 e^{-\int_0^T r(t) dt} x_0}{1- \Gamma_0 e^{-2 \int_0^T r(t) dt}}. 
} 
Furthermore, the value of \eqref{optimization_problem} for the optimal wealth process $X^*$ $=$ $X^{\alpha^*}$ is 
\bes{
     \label{eq:value_final}
  V(m) \; = \;    \V (X_T^*) \; = \;  \Gamma_0 \frac{\big|x_0 - m e^{-\int_0^T r(t) dt} \big|^2}{1 - \Gamma_0 e^{-2\int_0^T r(t) dt}}.
}
\end{theorem}

\begin{proof}
    We refer to Appendix~\ref{A:verifresult}.
\end{proof}

\vspace{1mm}

\begin{remark}
By setting $\tilde Z_t^i$ $=$ $\Gamma_t Z_t^i$, $i$ $=$ $1,2$, the BSDE \eqref{eq:riccati_sto} agrees with the one in  \citet[Theorem 3.1]{chiu2014mean}: 
\begin{align}
 d \Gamma_t &=\;  \Gamma_t \Big[ \Big(-2r(t) +   \big| \lambda_t +  \frac{\tilde Z^1_t  + C \tilde Z^2_t}{\Gamma_t} \big|^2 \Big) \Big] dt 
 + \big(\tilde Z^1_t\big)^\T dB_t + \big(\tilde Z^2_t\big)^\T  dW_t,
\end{align}
and justifies the terminology Riccati BSDE. 
\end{remark}

\vspace{1mm}

In the sequel, we will provide concrete specifications of multivariate stochastic Volterra models for which  the solution to the non-linear Riccati BSDE \eqref{eq:riccati_sto} can be computed in closed and semi-closed forms,  while satisfying conditions (H1) and (H2). 
The  key idea is to observe that, first,  if such solution exists, then, it  admits the following representation as a Laplace transform:
\bes{
    \Gamma_t \; = \;  \E \Big[ \exp\Big(\int_t^T \big(2r(s) -  \left|\lambda_s +  {Z}^1_s + C {Z}^2_s \right|^2 \big)ds \Big) \Mid \mathcal F_t\Big], \quad 0 \leq t \leq T. 
}
In the special case where $\lambda$ is deterministic, then the solution to \eqref{eq:riccati_sto} trivially exists with $Z^1$ $=$ $Z^2$ $=$ $0$, and condition (H1) and (H2) are obviously satisfied when $\lambda$ is nonzero and bounded. In the general case  where $\lambda$ is an (unbounded) stochastic process, the admissibility of the optimal control is obtained under finiteness of a certain exponential moment of the solution triplet $(\Gamma,Z^1,Z^2)$ and the risk premium $\lambda$ as precised in (H2). Such estimate is crucial to deal with the unbounded random coefficients in \eqref{eq:wealth}, see for instance \citet{han2019mean,shen2014mean, shen2015mean} where similar conditions appear.  If the coefficients are bounded, such condition is not needed, see \citet[Lemma 3.1]{lim2004quadratic}.

Our main interest  is to find specific dynamics for  the volatility $\sigma$ and for the market price of risk  $\lambda$ such that the Laplace transform can be computed in (semi)-explicit form.  
We shall consider models as mentioned in Remark \ref{remlambdasigma}, where all the randomness in $\lambda$ comes from the 
process $W$ driving $\sigma$, and for which we naturally expect that $Z^1$ $=$ $0$.  We  solve more specifically  this problem for two classes of models:
\begin{enumerate}
\item
Multivariate affine  Volterra  models of Heston type in Section \ref{S:affine}. This  extends the results of \citet{han2019mean} to the multi dimensional case and provides semi-closed formulas.
    \item
Multivariate quadratic Volterra  models of Stein-Stein and Wishart type in Section \ref{S:quadratic} for which we derive new closed-form solutions.
\end{enumerate}

\section{Multivariate affine Volterra models}\label{S:affine}

 We let $K=\diag(K_1,\ldots,K_d)$ be diagonal with scalar kernels $K_i\in L^{2}([0,T],\R)$ on the diagonal, $\nu=\diag(\nu_1,\ldots,\nu_d)$ and $D \in \R^{d\times d}$ such that 
\begin{equation} \label{sqrt2}
 D_{ij}\geq 0, \quad i\neq j.
\end{equation}

We assume that  $\sigma$ in \eqref{eq:SIGMA} is given by $\sigma = \sqrt{\diag(V)}$, where $V=(V^1,\ldots, V^d)^\top$  is the following $\R^d_+$--valued Volterra square--root process

\begin{equation} 
\label{VolSqrt}
\begin{aligned}
V_t = g_0(t) + \int_0^t K(t-s) D V_s ds  + \int_0^t K(t-s) \nu \sqrt{\diag(V_s)}dW_s.
\end{aligned}
\end{equation}
Here $g_0:\R_+\to \R^{d}_+$, $W$ is a $d$-dimensional Brownian motion and the correlation structure with $B$ is given by
{\begin{align}\label{eq:correstructureheston}
W^i = \rho_i B^i + \sqrt{1-\rho_i^2} B^{\perp,i}, \quad i=1,\ldots,d,
\end{align}
for some $(\rho_1,\ldots,\rho_d)\in[-1,1]^d$.}  This corresponds to a particular case of the correlation structure in \eqref{eq:correlstructuregen} with $N$ $=$ $d$, and $C_i$ $=$ 
$(0,\ldots,\rho_i,\ldots,0)^\T$. 
Furthermore,  the risk premium is assumed to be in the form
{$\lambda$ $=$ $\big(\theta_1\sqrt{V^1},\ldots,\theta_d \sqrt{V^d}\big)^\top$}, for some {$\theta_i \geq 0$},  so that the dynamics  for the stock prices \eqref{eq:stocks} reads
\begin{align}
    \label{eq:hestonS}
    dS^i_t = S^i_t \left( r(t)  + \theta_i V^i_t   \right) dt + S^i_t \sqrt{V^i_t} dB^i_t, \quad i=1,\ldots, d.
\end{align}

We assume that there exists a continuous $\R^{2d}_+$-valued weak solution $(V,S)$ to \eqref{VolSqrt}-\eqref{eq:hestonS} on some filtered probability space $(\Omega,\mathcal F, (\mathcal F)_{t\geq 0}, \mathbb P)$ such that 
\begin{align}\label{eq:moments V1}
    \sup_{t\leq T} \E\left[ |V_t|^p \right] < \infty, \quad p \geq 1.
\end{align}
For instance, weak existence of $V$ such that \eqref{eq:moments V1} holds  is established  under suitable assump\-tions on the kernel $K$ and specifications $g_0$ as shown in the following remark. The existence of $S$ readily follows from that of $V$.

\begin{remark}\label{R:existenceV}
Assume that, for each $i=1,\ldots,d$, $K_i$ is completely monotone on $(0,\infty)$\footnote{A function $f$ is completely monotone on   $(0,\infty)$ if it is infinitely differentiable on $(0,\infty)$ such that $(-1)^n f^n(t)\geq 0$, for all $n\geq 1$ and $t>0$.},and that there exists $\gamma_i\in(0,2]$ and $k_i>0$ such that
 \begin{align}\label{eq:kernelhcond}
      \int_0^h K_i^2(t)dt +  \int_0^T (K_i(t+h)-K_i(t))^2 dt  & \leq \;  k_i   h^{\gamma_i}, \quad h>0.
 \end{align}
 This covers, for instance,  constant non-negative kernels,  fractional kernels  of the form $t^{H-1/2}/\Gamma(H+1/2)$ with $H \in(0,\frac12]$, and exponentially decaying kernels ${\rm e}^{-\beta t}$ with $\beta>0$. Moreover, sums and products of completely monotone functions are completely monotone, refer to \citet{AJLP17} for more details.
\begin{itemize}
 \item
 If $g_0(t)=V_0 + \int_0^t K(t-s)b^0ds$, for some $V_0,b^0 \in \R^d_+$, then \citet[Theorem 6.1]{AJLP17} ensures the existence of $V$ such that \eqref{eq:moments V1} holds,
 \item
 In \citet{AJEE18b}, the existence is obtained  for more general input curves $g_0$ for the case $d=1$, the extension to the multi-dimensional setting is straightforward. 
   \end{itemize}
\end{remark}

 \vspace{1mm}

Exploiting the affine structure of \eqref{VolSqrt}-\eqref{eq:hestonS}, see \citet{AJLP17}, we provide an explicit solution to the Riccati BSDE \eqref{eq:riccati_sto} 
in terms of the Riccati-Volterra equation 
\begin{align}
\psi^i(t)&= \int_0^t K_i(t-s) F_i(\psi(s)) ds,  \label{eq:Riccatipsi1} \\
F_i(\psi) &= -\theta_i^2  - 2 \theta_i \rho_i \nu_i \psi^i + (D^\top \psi)_i + \frac {\nu_i^2} 2  ( 1- 2\rho_i^2)(\psi^i)^2, \quad i=1,\ldots,d, \label{eq:Riccatipsi2}
\end{align} 
and the $\R^d$-valued process 
\begin{align}\label{eq:processg}
 g_t(s)= g_0(s) + \int_0^t K(s-u) D V_u du + \int_0^t K(s-u)\nu \sqrt{\diag(V_u)}dW_u, \quad s\geq t.
\end{align}
One notes that for each, $s\leq T$, $(g_t(s))_{t\leq s}$ is the adjusted forward process 
\begin{align}
g_t(s) &= \; \mathbb E\Big[  V_s - \int_t^s K(s-u)DV_udu \Mid \Fc_t\Big].
\end{align}
 
\vspace{1mm} 

\begin{lemma}
\label{lemma:existence_riccati_sto}
Assume that there exists  a solution $\psi \in C([0,T],\R^d)$ to the Riccati-Volterra equation \eqref{eq:Riccatipsi1}-\eqref{eq:Riccatipsi2}. Let $\left(\Gamma, Z^1, Z^2\right)$ be defined as 
\begin{equation} \label{eq:GammaHeston}
\left\{
\begin{array}{ccl}
    \Gamma_t &=& \exp\Big( 2\int_t^T r(s) ds +  \sum_{i=1}^d\int_t^T  F_i(\psi(T-s)) g^i_t(s) ds \Big), \\
    Z^1_t &=& 0, \\
  Z^{2,i}_t &=&  \psi^i(T-t) \nu_i \sqrt{V^i_t}, \quad i=1,\ldots,d, \quad 0 \leq t \leq T, 
\end{array}  
\right.
\end{equation}
where $g$ $=$ $(g^1,\ldots,g^d)^\T$ is given by \eqref{eq:processg}.  Then, $\left(\Gamma, Z^1, Z^2\right)$  is a   $\S^{\infty}_{\F}([0,T], \R) \times L^2_{\F}([0,T], \R^d) \times L^2_{\F}([0,T], {\R^d})$-valued solution to  \eqref{eq:riccati_sto}.
\end{lemma}
\begin{proof} 
We first observe that the correlation structure \eqref{eq:correstructureheston} implies that $C$ in \eqref{eq:CR} is given by  $C=\diag(\rho_1,\ldots,\rho_d)$. 
Set 
$$G_t =  2\int_t^T r(s) ds +  \sum_{i=1}^d\int_t^T  F_i(\psi(T-s)) g^i_t(s) ds, \quad t \leq T.$$
Then, $\Gamma = \exp(G)$ and 
\begin{align}\label{eq:Gammaitoexp}
d\Gamma_t = \Gamma_t \Big( d G_t + \frac 1 2 d\langle G \rangle_t \Big).
\end{align}
Using  \eqref{eq:processg}, and by stochastic Fubini's theorem, see \citet[Theorem 2.2]{V:12}, the dynamics of $G$ reads as  
\begin{align}
dG_t &= \;  \Big(-2r(t)  - \sum_{i=1}^d F_i(\psi(T-t)) V^i_t + \sum_{j=1}^d \int_t^T F_j(\psi(T-s)){K_j}(s-t)ds \sum_{i=1}^d D_{ji} V^i_t\Big)dt\\
& \quad \quad + \;  \sum_{i=1}^d \int_t^T F_i(\psi(T-s)){K_i}(s-t)ds \nu_i \sqrt{V^i_t}dW^i_t \\
& = \;  \Big(-2r(t) - \sum_{i=1}^d F_i(\psi(T-t)) V^i_t + \sum_{j=1}^d \psi^j(T-t) \sum_{i=1}^d D_{ji} V^i_t\Big)dt\\
& \quad \quad + \;  \sum_{i=1}^d \psi^i(T-t)\nu_i \sqrt{V^i_t}dW^i_t,
\end{align}
where we changed variables and  used the Riccati--Volterra equation \eqref{eq:Riccatipsi1} for $\psi$ for the last equality. This yields that 
the dynamics of $\Gamma$ in \eqref{eq:Gammaitoexp} is  given by 
\begin{align}
    d\Gamma_t &= \;  \Gamma_t \Big(-2r(t) + \sum_{i=1}^d V^i_t \big( - F_i(\psi(T-t)) + \sum_{j=1}^d  D_{ji} \psi^j(T-t) + \frac {\nu_i^2} 2 (\psi^i(T-t))^2 \big)\Big)dt \\ 
    & \quad \quad  + \;  \Gamma_t \sum_{i=1}^d \psi^i(T-t)\nu_i \sqrt{V^i_t}dW^i_t \\
    & =  \;  \Gamma_t \Big[ \big(-2r(t)  + \sum_{i=1}^d V^i_t (\theta_i + \rho_i \nu_i \psi^i(T-t))^2\big)dt +(Z^{2}_t)^\top dW_t \Big],  \label{eq:gamma_heston}
\end{align}
where we used  \eqref{eq:Riccatipsi2} for the last identity. Finally, observing that  
\begin{align}
\left| \lambda_t + Z^1_t + C Z^2_t \right|^2  &= \;  \sum_{i=1}^d \left(\theta_i + \rho_i \nu_i\psi^i(T-t )\right)^2 V^i_t,
\end{align}
together with $\Gamma_T=1$, we get that $(\Gamma,Z^1,Z^2)$ as defined in \eqref{eq:GammaHeston} solves the BSDE \eqref{eq:riccati_sto}.

{It remains to show that $\left(\Gamma, Z^1, Z^2\right) \in \S^{\infty}_{\F}([0,T],\R) \times L^2_{\F}([0,T], \R^d) \times L^2_{\F}([0,T], \R^d)$}.  For this, define the process 
\begin{align}
M_t &= \;  \Gamma_t \exp\Big(\int_t^T\big( - 2r(s) + \sum_{i=1}^d V^i_s (\theta_i + \rho_i \nu_i \psi^i(T-s))^2\big) ds\Big), \quad t \leq T.
\end{align}
An application of It\^o's formula combined with the dynamics \eqref{eq:gamma_heston} shows that $dM_t = M_t  (Z^2_t)^\top dW_t$, and so $M$ is a local martingale of the form
\begin{align} 
M_t &= \; \mathcal E\Big( \int_t^T \sum_{i=1}^d \psi^i(T-s)\nu_i \sqrt{V^i_s}dW^i_s \Big). 
\end{align}
Since $\psi$ is continuous, it is bounded so that a straightforward adaptation of \citet[Lemma 7.3]{AJLP17} to the multi-dimensional setting, recall \eqref{eq:moments V1},  
yields that $M$ is a true martingale. Since $M_T=1$, writing $\E[M_T|\mathcal F_t]=M_t$, we obtain 
\bes{
    \label{eq:ito_gamma}
    \Gamma_t = \E \Big[ \exp\Big(\int_t^T \big(2r(s) - \sum_{i=1}^d V^i_s (\theta_i + \rho_i \nu_i \psi^i(T-s))^2\big) ds\Big) \mid \mathcal F_t\Big], \quad t\leq T,
}
which ensures that $0$ $<$ $\Gamma_t$ $\leq$ $e^{2 \int_t^T r(s) ds}$, $\P-a.s.$, since $V\in \mathbb R^d_+$. As for $Z^2$, it is clear that it belongs to $L^2_{\F}([0,T], \R^d)$ since $\Gamma$ and $\psi$ are bounded and $\E \Big[\int_0^T  \sum_{i=1}^d V^i_s ds \Big] <  \infty$ by \eqref{eq:moments V1}.
\end{proof}

The following remark makes precise the existence of a continuous solution to the Riccati-Volterra equation \eqref{eq:Riccatipsi1}-\eqref{eq:Riccatipsi2}.
 
\begin{remark}
Assume that $K$ satisfies the assumptions  of Remark \ref{R:existenceV}.
\begin{itemize}
    \item 
    If $1-2 \rho_i^2 \geq 0$, then \citet[Lemma 6.3]{AJLP17} provides the existence of a unique solution $\psi \in L^2([0,T],\R^d_-)$. Continuity of such solution can then be easily established, since as opposed to  \citet[Lemma 6.3]{AJLP17}, \eqref{eq:Riccatipsi1} starts from $0$.
    \item
    If $d=1$ and $1-2\rho_1^2 < 0$, \citet[Lemma A.4]{han2019mean} establishes the existence of a continuous solution $\psi$.
\end{itemize}
\end{remark}

\vspace{1mm}

Using Theorem~\ref{T:verif}, we can now explicitly solve the Markowitz problem \eqref{optimization_problem} in the multivariate Volterra Heston model \eqref{VolSqrt}-\eqref{eq:correstructureheston}-\eqref{eq:hestonS}. The next theorem extends \cite[Theorem 4.2]{han2019mean} to the multivariate case.  Notice that the martingale {distortion} argument in this cited paper is specific to the dimension $d$ $=$ $1$, and here, instead, we rely on the generic verification result in Theorem \ref{T:verif}. 
 
\begin{theorem} 
Assume that there exists  a solution $\psi \in C([0,T],\R^d)$ to the Riccati-Volterra equation \eqref{eq:Riccatipsi1}-\eqref{eq:Riccatipsi2} such that 
\bes{\label{eq:condtheta}
    \max_{1 \leq i \leq d} \max_{t \in [0,T]} \left( \theta_i^2 + \nu_i^2 \psi^i(t)^2 \right) \leq \frac{a}{a(p)}, \quad \text{ for some } p>2,
}
 where $a(p)$ is given by \eqref{eq:constap} and the constant $a>0$ is such that $\E\left[\exp\big(a\int_0^T \sum_{i=1}^d V^i_s ds\big)\right] < \infty$. Assume that $g^i_0(0)>0$ for some $i\leq d$. Then, the optimal investment strategy for the maximization problem \eqref{optimization_problem} in the   multivariate Volterra Heston model \eqref{VolSqrt}-\eqref{eq:correstructureheston}-\eqref{eq:hestonS} is  given by the admissible control
        \bes{
        \label{eq:optimal_control_Heston}
        \c^{*i}_t = - \left(\theta_i + \rho_i\nu_i \psi^i(T-t) \right) \sqrt{V_t^i} \left(X^{\c^*}_t - \xi^* e^{-\int_t^T r(s) ds}\right), \quad 1 \leq i \leq d,
    }
 where $\xi^*$ is defined as in \eqref{eq:eta_star}, the wealth process $X^*$ $=$ $X^{\alpha^*}$ by \eqref{eq:wealth} with $\lambda=\big(\theta_1 \sqrt{V^1},$ $\ldots, \theta_d \sqrt{V^d} \big)^\top$, 
 and the optimal value is given by \eqref{eq:value_final} with $\Gamma_0$ as in  \eqref{eq:ito_gamma}.     
\end{theorem}
\begin{proof}
First note that under the specification \eqref{eq:GammaHeston}, the candidate for the optimal feedback control defined in  \eqref{eq:optimal_control_final} takes the form
    \bes{
    \c^*_t &= \;   - \big(\lambda_t + Z^1_t + C Z^2_t \big) \big(X^*_t  - \xi^* e^{-\int_t^T r(s) ds}\big)  \\
    &= \;  \Big(- \big(\theta_i + \rho_i\nu_i \psi^i(T-t) \big) \sqrt{V_t^i} \big(X^*_t - \xi^* e^{-\int_t^T r(s) ds}\big)   \Big)_{1 \leq i \leq d}.
    }
It then  suffices to check that the assumptions of Theorem \ref{T:verif} are verified to ensure that such $\c^*$ is optimal and to get that \eqref{eq:value_final}  is the optimal value. 
    The existence of a solution triplet $ (\Gamma, Z^1, Z^2) \in { \S^{\infty}_{\F}([0,T], \R) \times L^2_{\F}([0,T], \R^d) \times L^2_{\F}([0,T], \R^N)}$ to the stochastic backward Riccati equation \eqref{eq:riccati_sto} is ensured by Lemma \ref{lemma:existence_riccati_sto}. In addition,  \eqref{eq:ito_gamma} implies that $\Gamma_0 < e^{2\int_0^T r(s) ds}$ since $g^i_0(0)>0$ for some $i\leq d$ by assumption and $V^i$ is continuous. Thus  condition (H1)  of Theorem \ref{T:verif} is verified.
    As for condition (H2) of Theorem \ref{T:verif}, note that 
    \bes{
       a(p)\left( |\lambda_s|^2 + \left|Z^1_s\right|^2 + \left|Z^2_s\right|^2 \right) \; = \; a(p) \sum_{i=1}^d V_s^i \left( \theta_i^2 + \nu_i^2 \psi^i(t)^2 \right) \; \leq \;  a \sum_{i=1}^d V_s^i,
    }
which implies that $\E \left[ \exp\left(  a(p) \int_0^T \left( |\lambda_s|^2 + \left|Z^1_s\right|^2 + \left|Z^2_s\right|^2  \right)ds\right) \right]< \infty$ and ends the proof.
\end{proof}
 
\begin{remark} 
Condition \eqref{eq:condtheta} concerns the risk premium constants  $(\theta_1,\ldots, \theta_d)$. For  $a>0$,   a sufficient condition ensuring 
$\E\big[\exp\big(a\int_0^T \sum_{i=1}^d V^i_s ds\big)\big]$ $<$ $\infty$ is the existence of a conti\-nuous solution $\tilde{\psi}$ to  the Riccati--Volterra 
\begin{align} 
\tilde{\psi}^i(t) &= \; \int_0^t K_i(t-s) \Big(a+ \big(D\tilde{\psi}(s)\big)_i + \frac{\nu_i^2}{2}\tilde{\psi}^i(s)  \Big) ds,
\end{align} 
see \citet[Theorem 4.3]{AJLP17}.
In the one dimensional case $d=1$,   such existence  is established in \citet[Lemma A.2]{han2019mean} {for the case where $g_0(t) = V_0 + \kappa \int_0^t K(t-s)\phi ds$, $\phi\geq 0$, $D = - \kappa$ and $a < \frac{\kappa^2}{2 \nu^2} $}. 
\end{remark}

{\begin{remark}
    Note that in the one dimensional case the condition \eqref{eq:condtheta} can be made more explicit by bounding $\psi$ with respect to $\theta$. Indeed since $-\theta^2<0$ we get from \citet[Theorem C.1]{abi2019multifactor} that $\psi$ is non-positive. Furthermore, the fact that $\psi$ is solution to the following linear Volterra equation 
    \bes{
        \chi(t) = \int_0^t K(t-s)  \Big( -\theta^2 + \big((D-2\theta \rho \nu) + \frac{\nu^2}{2}(1-2\rho^2)\psi(s)\big)\chi(s) \Big)ds,
    }
    leads to, see \citet[Corollary C.4]{abi2019multifactor},
    \bes{
        \sup_{t \in [0,T]} |\psi_t| \leq |\theta|^2\int_0^T R_{D}(s)ds,
    }
    where $R_D$ is the resolvent of $K D$. Consequently, a sufficient condition on $\theta$ to ensure \eqref{eq:condtheta} would be 
    \bes{
        \theta^2\left( 1+(\theta\nu)^2\int_0^T R_D(s)ds \right) \leq \frac{a}{a(p)}.
    }
\end{remark}
}

\begin{remark}
In order to numerically implement the optimal strategy \eqref{eq:optimal_control_Heston}, one needs to simulate  the possibly non-Markovian process $V$ and to discretize the Riccati-Volterra equation  for $\psi$.
\citet{abi2019lifting,abi2019multifactor} develop a  taylor-made  approximating procedure for  the stochastic Volterra equation \eqref{VolSqrt} (resp. the Riccati-Volterra equation \eqref{eq:Riccatipsi1}), using  finite-dimensional Markovian semimartingales (resp. finite-dimensional Riccati ODE's). An illustration of such procedure on the mean-variance pro\-blem in the univariate Volterra Heston model for the fractional kernel is given in  \citet[Section 5]{han2019mean}.
\end{remark}

\section{Multivariate quadratic Volterra models}\label{S:quadratic}

Before we introduce the class of multivariate quadratic Volterra models, we need to define and {introduce some notations} on integral operators.

\subsection{Integral operators}

Fix $T>0$.  We denote by $\langle \cdot, \cdot \rangle_{L^2}$ the inner product on  $L^2\left([0,T],\R^N\right)$ that is
\begin{align}
\langle f, g\rangle_{L^2} = \int_0^T f(s)^{\top} g(s) ds, \quad f,g\in L^2\left([0,T],\R^N\right). 
\end{align}

We define $L^2\left([0,T]^2,\R^{N\times N}\right)$ to be the space of  measurable kernels $K:[0,T]^2 \to \R^{N\times N}$ such that 
\begin{align*}
\int_0^T \int_0^T |K(t,s)|^2 dt ds < \infty.
\end{align*}
For any $K,L \in  L^2\left([0,T]^2,\R^{N\times N}\right)$ we define the $\star$-product   by
\begin{align}\label{eq:starproduct}
(K \star L)(s,u) = \int_0^T K(s,z) L(z,u)dz, \quad  (s,u) \in [0,T]^2,
\end{align}
which is well-defined in $L^2\left([0,T]^2,\R^{N\times N}\right)$ due to the Cauchy-Schwarz inequality.  For any  kernel $K \in L^2\left([0,T]^2,\R^{N\times N}\right)$, we denote by {$\boldsymbol K$} the integral operator   induced by the kernel $K$ that is 
\begin{align}
 ({\boldsymbol K} g)(s)=\int_0^T K(s,u) g(u)du,\quad g \in L^2\left([0,T],\R^N\right).
\end{align}
$\boldsymbol K$ is a linear bounded operator from  $L^2\left([0,T],\R^N\right)$ into itself. 
  If $\boldsymbol{K}$ and $\boldsymbol{L}$ are two integral operators induced by the kernels $K$ and $L$  in $L^2\left([0,T]^2,\R^{N\times N}\right)$, then $\boldsymbol{K}\boldsymbol{L}$ is the  integral operator induced by the kernel $K\star L$.

We denote by $K^*$ the adjoint kernel of $K$ for $\langle \cdot, \cdot \rangle_{L^2}$, that is 
\begin{align}
K^*(s,u) &= \; K(u,s)^\top, \quad  (s,u) \in [0,T]^2,
\end{align}
and by $\boldsymbol{K}^*$ the corresponding adjoint integral operator. 

\vspace{1mm}

\begin{definition}\label{D:nonnegative}
A kernel $K \in L^2\left([0,T]^2,\R^{N\times N}\right)$ is symmetric nonnegative if $K=K^*$ and  
$$   \int_0^T \int_0^T f(s)^\top K(s,u) f(u) du ds \geq 0, \quad  \forall f  \in L^2\left([0,T],\R^N\right).  $$
In this case, the integral operator $\boldsymbol{K}$ is said to be symmetric nonnegative and $\boldsymbol{K}=\boldsymbol{K}^*$ and $\langle f,\boldsymbol{K}f\rangle_{L^2}\geq 0$.  $\boldsymbol{K}$ is said to be symmetric nonpositive, if $(-\boldsymbol{K})$ is symmetric nonnegative. 
\end{definition}

\vspace{1mm}

We recall the definition of    Volterra kernels of continuous and  bounded type  in the terminology of \citet[Definitions 9.2.1, 9.5.1 and 9.5.2]{GLS:90}.  

\begin{definition}\label{D:kernelvolterra} A kernel $K:\R_+^2 \to \R^{N\times N}$  is a Volterra kernel of continuous and  bounded type  in $L^2$ if   $K(t,s)=0$ whenever $s>t$ and
\begin{align}
    \label{assumption:K_stein}
    \sup_{t\in [0,T]}\int_0^T |K(t,s)|^2 ds <  \infty, \; \mbox{ and }  \lim_{h \to 0} \int_0^T |K(u+h,s)-K(u,s)|^2 ds=0, \; u \leq T.
\end{align}
  \end{definition}
  
\vspace{1mm}

Any convolution kernel of the form $K(t,s)=k(t-s)\bold 1_{s\leq t}$	with $k\in L^2\left([0,T],\R^{N\times N}\right)$ satisfies \eqref{assumption:K_stein}, we refer to \citet[Example 3.1]{jaber2019laplace} for additional examples. Note that $(s,t)\mapsto K(s,t)$ is not necessarily continuous nor bounded.

  {For completeness, we collect in Appendix~\ref{A:resolvents} below standard results for integral operators and their resolvents.}

 \subsection{The model}
 
 In this section, we assume that the components of the stochastic volatility matrix $\sigma$ in \eqref{eq:stocks} are given by  $\sigma^{ij}= \gamma_{ij}^\top Y $, where  $\gamma_{ij}\in \R^N$ and  $Y=(Y^1,\ldots,Y^N)^\top$ is the following $N$-dimensional Volterra Ornstein--Uhlenbeck process
\begin{align}
    \label{eq:SteinY}
    Y_t = g_0(t) + {\int_0^t K(t,s)DY_sds+} \int_0^t K(t,s) {\eta} dW_s, 
\end{align} 
where  $D,\eta \in \R^{N\times N}$,  $g_0:\R_+ \to \R^{N}$ is locally bounded,
$W$ is a {$N$-dimensional} process as in \eqref{eq:correlstructuregen}, i.e.,  
\begin{align}\label{eq:Steincorel}
W^k_t =  C_k^\T B_t + \sqrt{1-C_k^\T C_k }B^{\perp,k}_t, \quad 
\end{align}
where $C_k \in {\R^d}$, such that $C_k^\top C_k\leq 1$, $k=1,\ldots,N$,
and
$K:[0,T]^2 \to \R^{N\times N}$ is a Volterra kernel of continuous and bounded type   in $L^2$ as in Definition~\ref{D:kernelvolterra}. We stress that the process $W$ is not necessarily a $N$-dimensional Brownian motion due to the possible correlations.

Furthermore,  the risk premium is assumed to be in the form
\begin{align}
\lambda_t =\Theta Y_t, \quad t \leq T,
\end{align}
for some $\Theta \in \R^{d \times N}$,
 so that the dynamics  for the stock prices \eqref{eq:stocks} reads as 
\begin{align}\label{eq:Steinprice}
dS^i_t &= \;  S^i_t \Big( r(t) + \sum_{k,\ell=1}^N \sum_{j=1}^d \gamma_{ij}^\ell  \Theta^{jk}  Y^\ell_t Y^k_t    \Big) dt 
+  S^i_t \sum_{j=1}^d  \gamma_{ij}^\top Y_t dB^j_t, \quad i=1,\ldots, d.
\end{align}

The appellation quadratic reflects the quadratic dependence of the drift and the covariance matrix of $\log S$ in $Y$.  Such models nest as special cases the Volterra extensions of the celebrated  \citet{stein1991stock} or \citet{schobel1999stochastic} model  and certain Wishart models of \citet{bru1991wishart} as shown in the following example. 

\vspace{1mm}

\begin{example}\label{E:quadratic}
(i) \textit{The multivariate Volterra Stein-Stein model:}

\noindent For $N=d$,  $K=\diag(K^1,\ldots,K^d)$ and $\gamma_{ij}=\beta_{ij} e_i$ with $\beta_{ij} \in \R$  such that $\sum_{j=1}^d \beta_{ij}^2=1$  and $(e_1,\ldots,e_d)$ the canonical basis of $\R^d$, we recover the multivariate Volterra Stein-Stein model defined by 
$$
\begin{cases}
dS^i_t &= \; S^i_t \Big( r(t) + \sum_{j,k=1}^d \beta_{ij} \Theta^{jk} Y^i_t Y^k_t    \Big) dt + S^i_t Y^i_t \sum_{j=1}^d \beta_{ij} dB^j_t, \\
\; Y^i_t &=  \;  g_0^i(t) + \int_0^t K^i(t,s) \sum_{j=1}^d D^{ij}Y^j_s ds + \int_0^t K^i(t,s) \eta^i dW_s^i,
\quad i=1,\ldots, d,
\end{cases}
$$
 and $C_{i}=\rho_i(\beta_{i1},\ldots,\beta_{id})^\top$ to take into account the leverage effect.  Recall that $W$ is possibly correlated and is not necessarily a Brownian motion.
\vspace{1mm}

\noindent (ii) \textit{The Volterra Wishart covariance model:}  

\noindent Using the vectorization operator, which stacks the columns of a matrix one underneath another in a vector, see \citet[Section 3.1]{jaber2019laplace}, one can recover the Volterra Wishart covariance model for $N=d^2$:
$$
\begin{cases}
dS_t &= \;  \diag(S_t) \big[ r(t) {\bold{1}_d}  dt +   \tilde Y_t dB_t\big],  \quad S_0 \in \R^d_+, \\
\; \Tilde Y_t &= \;  \tilde g_0(t) + \int_0^t  \tilde K(t,s) D Y_s ds + \int_0^t \tilde K(t,s)\eta dW_s,
\end{cases}
$$
with $\tilde g_0:[0,T]\to \R^{d \times d}$, a suitable measurable kernel  $\tilde K: [0,T]^2 \to \R^{d \times d}$,  
a $d\times d$ Brownian motion $W$ and 
\begin{align}\label{eq:correlatedBM}
W^{ij} = \rho_{ij}^\top B + \sqrt{1- \rho_{ij}^\top \rho_{ij}}\, B^{\perp,ij},\quad i,j=1,\ldots,d,
\end{align}
for some  $ \rho_{ij}\in \R^{d \times d}$ such that $ \rho_{ij}^\top \rho_{ij}\leq 1$, for $i,j=1,\ldots,d$,
where  $B^\perp$ is a  $d\times d$--dimensional Brownian motion  independent of $B$. Here the process $\tilde Y$ is $d\times d$-matrix valued.
\end{example}

\begin{remark}\label{R:correlStein}
Note  that with \eqref{eq:Steincorel},  there are  no restrictions on the correlations between $Y^i$ and the stocks $S^i$ in \eqref{eq:SteinY} and \eqref{eq:Steinprice}, 
in contrast with the correlation structure \eqref{VolSqrt} of the multivariate Volterra Heston model.  Moreover,  the models in Example \ref{E:quadratic} allow us to deal with  correlated stocks 
in contrast with the multivariate Heston model in \eqref{eq:hestonS} where no correlation between the driving Brownian motion of the assets $S^i$ and $S^j$ is allowed  in order to keep the affine structure. 
\end{remark}

\vspace{1mm}

{Since $K$ is a Volterra kernel of continuous and bounded type  in $L^2$, there exists a progressively measurable  $\R^{N}\times \R_+^d$-valued strong solution $(Y,S)$ to \eqref{eq:SteinY} and \eqref{eq:Steinprice}  such that 
\begin{align}\label{eq:moments V}
    \sup_{t\leq T} \E\left[ |Y_t|^p \right] < \infty, \quad p \geq 1.
\end{align}
Indeed, the solution for \eqref{eq:SteinY} is given in the following closed form
\begin{align}\label{eq:Yclosed}
    Y_t =g_0(t)+ \int_0^t R_D(t,s)g_0(s)ds +\int_0^t (K(t,s)+R_D(t,s))\eta dW_s,
\end{align}
where $R_D$ is the resolvent of  $KD$, whose existence is ensured by Lemma \ref{L:resolventvolterra}-\ref{L:resolventvolterrai} below, {we refer to Appendix~\ref{A:resolvents} for more details on the resolvents}. The existence of $S$ readily follows from that of $Y$ and is given as a stochastic exponential.} In the sequel, we will assume  that the solution $Y$ is continuous. Additional conditions  on $K$, in the spirit of \eqref{eq:kernelhcond}, 
 are needed to ensure the existence of continuous modification, by an application of the Kolmogorov-Chentsov continuity criterion,  for instance, as shown in the following remark.

\begin{remark}
For $s\leq t$ and $p\geq 2$,
an application of Jensen and Burkholder-Davis-Gundy's inequalities yield
\begin{align}
    \E\left[ |(Y_t -g_0(t))- (Y_s-g_0(s))|^p \right] \leq  c&\Big( 1 + \sup_{r\leq T}\E\big[|Y_s|^p\big]\Big)\\
    & \times \left( \int_s^t |K(t,r)|^2 dr + \int_0^T |K(t,r)-K(s,r)|^2 dr \right)^{p/2}.
\end{align}
 This shows that $(Y-g_0)$ admits a continuous modification, by the Kolmogorov-Chentsov continuity criterion,  provided that
\begin{align}
     \int_s^t |K(t,r)|^2 dr + \int_0^T |K(t,r)-K(s,r)|^2dr \leq c |t-s|^{\gamma}, 
\end{align}
for some $\gamma>0$.
\end{remark}


\subsection{The explicit solution}

In this section, we provide an explicit solution for the Markowitz problem for quadratic Volterra models, and  our main result is stated in Theorem~\ref{T:quadracticsolution} below. 

Exploiting the quadratic structure of \eqref{eq:SteinY}-\eqref{eq:Steinprice}, see \citet{jaber2019laplace}, we provide an explicit solution to the Riccati BSDE in Lemma \ref{lemma:existence_riccati_sto_quadratic} below, in terms of the following family of linear operators  $( \boldsymbol{\Psi}_t)_{0\leq t\leq T}$ acting on $L^2\left([0,T],\R^N\right)$:
\bes{
    \label{def:riccati_operator}
     \boldsymbol{\Psi}_t = - \Big(\id - \boldsymbol{\hat{K}}\Big)^{-*} {\Theta^\top \Big(\id + 2\Theta \tilde{\boldsymbol{\Sigma}}_t \Theta^\top\Big)^{-1} 
     \Theta}\Big(\id - \boldsymbol{\hat{K}}\Big)^{-1}, \quad 0\leq t\leq T,
}
where $\boldsymbol F^{-*}:=(\boldsymbol{F}^{-1})^{*}$, and $\boldsymbol{\hat{K}}$  is the integral operator induced by the kernel $\hat{K} = K(D-2\eta C^\T \Theta)$  and $\tilde{\boldsymbol{\Sigma}}_t$ the integral operator defined by
\bes{
    \label{def:C_tilde}
\boldsymbol{\tilde{\Sigma}}_t = (\id-\boldsymbol{\hat{K}})^{-1} \boldsymbol{\Sigma}_t (\id-\boldsymbol{\hat{K}})^{-*}, \qquad t \in [0,T],
}
with $\boldsymbol{\Sigma}_t$ defined as the integral operator associated to the kernel
\bes{\label{eq:sigmakernel}
{\Sigma}_t(s,u) = \int_t^{s \wedge u } K(s,z)\eta \big(U -  2 C^\T C \big) \eta^\T K(u,z)^\T  dz, \qquad t \in [0,T],
}
where $U = \frac{d \langle W\rangle_t}{dt} = \left(1_{i=j} + 1_{i\neq j} (C_i)^\T C_j  \right)_{ 1 \leq i,j\leq N}$.

We start by deriving some first properties of  $t\mapsto\boldsymbol{\Psi}_t$, namely that it is well-defined,  strongly differentiable and satisfies an operator Riccati equation under  the following additional assumption on the kernel:
\begin{align}\label{eq:assumptionkerneldiff1}
    \sup_{t\leq T} \int_0^T |K(s,t)|^2 ds < \infty.
\end{align}
We recall that $t\mapsto \boldsymbol{\Psi}_t$ is said to be strongly differentiable at time $t\geq  0$,  if there exists a bounded linear operator $\dot{\boldsymbol{\Psi}}_t$  from $ L^2\left([0,T],\R^N\right)$  into  itself  such that
\begin{align}
    \lim_{h\to 0} \frac{1}{h} \| \boldsymbol{\Psi}_{t+h}-\boldsymbol{\Psi}_{t} -h \dot{\boldsymbol{\Psi}}_{t} \|_{\rm{op}}=0, \quad \text{where }  \|\boldsymbol{G}\|_{\rm {op}}= \sup_{f \in L^2([0,T],\R^N)} \frac{\|\bold G f\|_{L^2}}{\|f\|_{L^2}}.
\end{align}

\begin{lemma}\label{L:Psi}
Fix a kernel $K$ as in Definition~\ref{D:kernelvolterra} satisfying \eqref{eq:assumptionkerneldiff1}. Assume  that $(U-2C^\top C) \in \S^N_+$. Then, for each $t\leq T$, $\boldsymbol{\Psi}_t$ given by \eqref{def:riccati_operator} is well-defined and is a bounded  linear operator from $L^2\left([0,T],\R^N\right)$ into itself. Furthermore, 
\begin{enumerate}
\item \label{L:Psi1}
$(\Theta^\top \Theta \id + \boldsymbol{\Psi}_t)$ is an integral  operator induced by a kernel  $\psi_t(s,u)$ such that 
\bes{
    \label{eq:bound_psi_leb}
    \sup_{t\leq T} \int_{[0,T]^2}|\psi_t(s,u)|^2 ds du<\infty.
}
\item \label{L:Psi2}
For any $f \in L^2\left([0,T],\R^N\right)$,  
\bes{
    \label{eq:Psi_on_boudary}
    (\boldsymbol{\Psi}_t f 1_t)(t) =& (-\Theta^\top \Theta \id + \boldsymbol{\hat{K}}^*\boldsymbol{\Psi}_t )(f 1_t)(t),
}
where $1_t:s\mapsto \bold 1_{t\leq s}$.
    \item \label{L:Psi3}
  $t\mapsto \boldsymbol{\Psi}_t$ is strongly differentiable and  satisfies the operator Riccati equation
 \bes{
    \label{eq:riccati_psiBold}
    \dot{\boldsymbol{\Psi}}_t &= 2\boldsymbol{\Psi}_t \dot{\boldsymbol{{\Sigma}}}_t  \boldsymbol{\Psi}_t, \qquad t \in [0,T] \\
    {\boldsymbol{\Psi}_T}&=- \left(\id - \boldsymbol{\hat{K}}\right)^{-*} { \Theta^\top \Theta}\left(\id - \boldsymbol{\hat{K}}\right)^{-1}\\
}   
where $\dot{\boldsymbol{{\Sigma}}}_t$ is the strong derivative of $t\mapsto\bold{\Sigma}_t$ induced by the kernel
\begin{align}\label{eq:diffkernelsigma}
 \dot{\Sigma}_t(s,u)=-K(s,t)\eta \big( U-2C^\top C\big)\eta^\top K(u,t)^\top, \quad a.e.
\end{align}
\end{enumerate}
\end{lemma}

\begin{proof}
{The proof is given in Appendix~\ref{AppendixL:Psi}.}
\end{proof}



\vspace{3mm}

We are now ready to provide a solution for the Riccati-BSDE \eqref{eq:riccati_sto}. For this, denote by $g$ the process
\begin{align}
    \label{eq:processg_quadratic}
     g_t(s)= \mathbf{1}_{t \leq s} \Big(g_0(s) + \int_0^t K(s,u) D Y_u du + \int_0^t K(s,u)\eta dW_u \Big). 
\end{align}
One notes that for each, $s\leq T$, $(g_t(s))_{t\leq s}$ is the adjusted forward process 
\begin{align}
g_t(s) &= \; \mathbb E\Big[  Y_s - \int_t^s K(s,u)DY_udu \mid \Fc_t\Big], \quad s\geq t.
\end{align}
We also denote the trace of an integral operator $\boldsymbol F$ by  $\mbox{Tr}(\boldsymbol{F})=\int_0^T \tr(F(s,s))ds$, where $\tr$ is the usual trace of a matrix, 
and we define the function $\phi$ by
\begin{equation}    \label{eq:ode_phi_quadratic}
\begin{cases}
     \dot{\phi}_{t} &={\mbox{Tr}\big( \boldsymbol{\Psi}_t \dot{\boldsymbol{{\Lambda}}}_t \big)} - 2r(t) \\
     &= \int_{(t,T]} \tr\left( \Theta^\top \Theta K(s,t)\eta  U\eta^\T K(s,t)^\T \right)ds\\  &\quad -   \int_{(t,T]^2} \tr\left( \psi_{t}(s,u) K({u},t)\eta  U\eta^\T K({s},t)^\T \right)ds du - 2r(t), \\
    \phi_{T} &= 0, 
\end{cases}
\end{equation}
where   $\dot{\boldsymbol{{\Lambda}}}_t$ is the integral operator induced by the kernel given by 
\begin{align}\label{eq:dotlambda}
    \dot{\Lambda}_t(s,u)= -K(s,t){\eta U \eta^\top }K(u,t)^\top, \quad u,s\leq T.
\end{align}

\vspace{1mm}

\begin{lemma} \label{lemma:existence_riccati_sto_quadratic} Fix a kernel $K$ as in Definition~\ref{D:kernelvolterra} satisfying \eqref{eq:assumptionkerneldiff1}. Assume  that $(U-2C^\top C) \in \S^N_+$.
     Let $\boldsymbol{\Psi}$ be the operator defined in \eqref{def:riccati_operator}. Then, the process $\left(\Gamma, Z^1, Z^2\right)$  defined by
    \begin{equation}\label{eq:Gamma_quadratic}
    \begin{cases}
        \Gamma_t &= \;  {\exp\left( \phi_{t}  + \langle g_t, \boldsymbol{\Psi}_{t} g_t \rangle_{L^2} \right)  } , \\
        Z^1_t &= \; 0, \\
      Z^{2}_t &= \; 2 \big( ( \boldsymbol{\Psi}_t \boldsymbol{K} \eta)^* g_t \big) (t),
     \end{cases} 
    \end{equation}
 where $g$ and $\phi$ are respectively given by \eqref{eq:processg_quadratic} and \eqref{eq:ode_phi_quadratic}, is a  $\S^{\infty}_{\F}([0,T], \R) \times L^2_{\F}([0,T], \R^d) \times L^2_{\F}([0,T], {\R^N})$-valued solution to the Riccati-BSDE \eqref{eq:riccati_sto}.
\end{lemma}

\begin{proof}
Set $G_t =  \phi_{t}  + \langle g_t, \boldsymbol{\Psi}_t g_t \rangle_{L^2}$, so that  $\Gamma_t = \exp(G_t)$ and 
\bes{
	\label{eq:Gamma_ito_exp}
	d\Gamma_t =& \;  \Gamma_t \big(d G_t + \frac 1 2 d\langle G \rangle_t \big).
}
To obtain the dynamics of $G$ it suffices to determine the dynamics of the process $t\mapsto \langle g_t, \boldsymbol{\Psi}_t g_t \rangle_{L^2}$. 

\vspace{1mm}

\noindent \textit{Step 1.} In this step we prove that the dynamics of $t\mapsto \langle g_t, \boldsymbol{\Psi}_t g_t \rangle_{L^2}$ is given by
\begin{align}
d\langle g_t, \boldsymbol{\Psi}_t g_t \rangle_{L^2}&= \; \Big(\langle g_t, \dot{\boldsymbol{\Psi}}_t g_t \rangle_{L^2} + \lambda_t^\T \lambda_t +  
2 \lambda_t^\T C Z_t^2    + \mbox{Tr}\big( \boldsymbol{\Psi}_t \dot{\boldsymbol{\Lambda}}_t\big) \Big) dt  +  { (Z_t^2)^\T dW_t}. \label{eq:dynamicsL2inner}
\end{align}
We first note that
\bes{
	\langle g_t, \boldsymbol{\Psi}_t g_t \rangle_{L^2} =& \int_0^T g_t(s)^\T (\boldsymbol{\Psi}_t g_t)(s) ds, 
}
and compute the dynamics of $t\mapsto g_t(s)^\T (\boldsymbol{\Psi}_t g_t)(s)$. For fixed $s\leq T$, it follows from \eqref{eq:processg_quadratic} and the fact that $Y_t=g_t(t)$, that 
\begin{align}
dg_t(s) &= \; - \delta_{t=s}g_t(t) dt+ K(s,t)Dg_t(t) dt + K(s,t)\eta dW_t.
\end{align}
Together with Lemma~\ref{L:Psi}-\ref{L:Psi3}, we deduce  that $t\mapsto (\boldsymbol{\Psi}_t g_t)(s)$ is a semimartingale with the following dynamics 
\begin{align*}
d(\boldsymbol{\Psi}_tg_t)(s) &= (\dot{\boldsymbol{\Psi}}_tg_t)(s)dt +  (\boldsymbol{\Psi}_tdg_t)(s)\\ 
&= (\dot{\boldsymbol{\Psi}}_tg_t)(s)dt -\psi_t(s,t)g_t(t)dt +(\boldsymbol{\Psi}_t K(\cdot,t)D g_t(t))(s) dt   +   (\boldsymbol{\Psi}_t K(\cdot,t)\eta dW_t)(s).   
\end{align*}
Here,  we used the fact that $\id \delta_{t} = 0$: indeed, for every $f \in L^2([0,T], \R^d)$ we have $(\id \delta_{t})(f) = (f(\cdot)\mathbf{1}_{t = \cdot}) =0_{L^2}$.
Moreover, 
\begin{align*}
d\langle g_{\cdot}(s), (\boldsymbol{\Psi}_{\cdot} g_{\cdot})(s) \rangle_t&= \; 
- \tr\big( \Theta^\top \Theta K(s,t)\eta  U\eta^\T K(s,t)^\T \big)dt \\
& \quad\quad  + \;    \int_t^T \tr\big( \psi_{t}(s,u) K({u},t)\eta  U\eta^\T K({s},t)^\T \big) du dt \\
&= \;  \tr\big( \Theta^\top \Theta \dot{\Lambda}_t(s,s) \big)dt - \int_t^T \tr\big( \psi_{t}(s,u) \dot{\Lambda}_t(u,s) \big)du dt \\
&= \;  {- \tr\Big( \big(\boldsymbol{\Psi}_t \dot \Lambda_t (\cdot, s)\big)(s) \Big).} 
\end{align*}
Whence, combining the previous three identities, we get 
\begin{align*}
d\left(g_t(s)^\T (\boldsymbol{\Psi}_t g_t)(s)\right)&=  \; dg_t(s)^\top (\boldsymbol{\Psi}_t g_t)(s)+ g_t(s)^\top d(\boldsymbol{\Psi}_t g_t)(s) + d\langle g_{\cdot}(s), (\boldsymbol{\Psi}_{\cdot} g_{\cdot})(s) \rangle_t  \\
&= \; -\delta_{t=s}g_t(t)^\top (\boldsymbol{\Psi}_t g_t)(s) dt  +  g_t(t)^\top D^\top K(s,t)^\top  (\boldsymbol{\Psi}_t g_t)(s)  dt  \\
& \quad + \;   g_t(s)^\top (\dot{\boldsymbol{\Psi}}_tg_t)(s)dt - g_t(s)^\top\psi_t(s,t)g_t(t)dt + g_t(s)^\top(\boldsymbol{\Psi}_t K(\cdot,t)D g_t(t))(s) dt   \\
&\quad -  \; \tr\Big( \big(\boldsymbol{\Psi}_t \dot \Lambda_t (\cdot, s)\big)(s) \Big) \\
&\quad  + \;  dW_t^\top \eta^\top K(s,t)^\top  (\boldsymbol{\Psi}_t g_t)(s)  +   g_t(s)^\top (\boldsymbol{\Psi}_t K(\cdot,t)\eta dW_t)(s)  \\
&= \;  \Big [ \bold{I}(s)+\bold{II}(s)+ \bold{III}(s)+ \bold{IV}(s)+\bold{V}(s)+ \bold{VI}(s) \Big] dt + \bold{VII}(s) + \bold{VIII}(s). 
\end{align*}
We now integrate  in $s$. First, using Lemma~\ref{L:Psi}-\ref{L:Psi1}  we get that  
\begin{align*}
\int_0^T\big[  \bold{I}(s) +\bold{IV}(s) \big] ds &= \;  -g_t(t)^\top (\boldsymbol{\Psi}_t g_t)(t) - g_t(t)^\top \int_t^T \psi_t(t,u)g_t(u)du \\
&  = \; \lambda_t^\top \lambda_t -2g_t(t)^\top \int_t^T \psi_t(t,u)g_t(u)du \\
&= \;  \lambda_t^\top \lambda_t - 2 g_t(t)^\T((\boldsymbol{\Psi}_t + \Theta^\T \Theta \id)g_t)(t). 
\end{align*}
On the other hand, since,  $\boldsymbol{\Psi}^*=\boldsymbol{\Psi}$, we have 
\begin{align*}
\int_0^T \big[ \bold{II}(s)+ \bold{V}(s) \big]  ds & = \;   2 g_t(t)^\top \Big( \big( (\boldsymbol{K} D)^* \boldsymbol{\Psi}_t \big) g_t \Big)(t).
\end{align*}
Therefore, summing the above, using Lemma~\ref{L:Psi}-\ref{L:Psi2}, and the definition of $\boldsymbol{\hat{K}}$, we get 
\begin{align*}
\int_0^T \big[ \bold{I}(s) + \bold{IV}(s) +\bold{II}(s)+ \bold{V}(s)\big] ds &= \;  
\lambda_t^\top \lambda_t - 2 g_t(t)^\T\Big( \big(\boldsymbol{\Psi}_t + \Theta^\T \Theta \id - (\left( \boldsymbol{K} D\right)^* )\boldsymbol{\Psi}_t \big)g_t\Big)(t)\\
&= \; \lambda_t^\top \lambda_t + 4 g_t(t)^\T((\boldsymbol{K}\eta C^\top \Theta )^*\boldsymbol{\Psi}_t) g_t)(t)\\
&= \; \lambda_t^\top \lambda_t + 2 \lambda_t^\T C Z_t^2. 
\end{align*}
Finally, observing that 
\begin{align}
\int_0^T \bold{III}(s)ds &= \;  \langle g_t, \dot{\boldsymbol{\Psi}}_t g_t\rangle_{L^2}, \quad  \int_0^T \bold{VI}(s) ds \; = \; \mbox{Tr}\Big( \boldsymbol{\Psi}_t \dot{\boldsymbol{\Lambda}}_t \Big), \\
\int_0^T \big[ \bold{VII}(s)+\bold{VIII}(s) \big]ds &= \; \big(Z^2_t\big)^\top dW_t,
\end{align}
we obtain  the claimed dynamics \eqref{eq:dynamicsL2inner}.

\vspace{1mm}

\noindent \textit{Step 2.} Plugging the dynamics \eqref{eq:dynamicsL2inner} in   \eqref{eq:Gamma_ito_exp}  yields 
\bes{
	\frac{d\Gamma_t}{\Gamma_t}  =&  \Big[ \underbrace{\dot{\phi}_{t,T} -\mbox{Tr}\left( \boldsymbol{\Psi}_t \dot{\boldsymbol{\Lambda}}_t\right)}_{\bold 1}
	+ \underbrace{ \langle g_t, \dot{\bold \Psi}_t  g_t \rangle_{L^2}+ \frac{(Z_t^2)^\T U Z_t^2}{2}}_{\bold 2} + \underbrace{\lambda_t^\T \lambda_t 
		+ 2\lambda_t^\T C Z_t^2}_{\bold 3} \Big] dt \\
	&\quad \quad + \; \left(Z^2_t\right)^\T  dW_t .\\
}
By   \eqref{eq:ode_phi_quadratic}, we have:  $\bold{1}$ $=$ $-2r(t)$.  From the definition of $Z^2$, we have
\bes{
	\frac{(Z_t^2)^\T U Z_t^2}{2} &= \; 2 \Big[ \Big( \big( \boldsymbol{\Psi}_t \boldsymbol{K} \eta \big)^* g_t \Big) (t) \Big]^\T U  
	\Big( \big( \boldsymbol{\Psi}_t \boldsymbol{K} \eta \big)^* g_t \Big) (t) \\
	&= \;   - 2\langle g_t , ( \boldsymbol{\Psi}_t \dot {\boldsymbol \Lambda}_t\boldsymbol{\Psi}_t ) g_t\rangle_{L^2}.
}
Thus, using the Riccati relation \eqref{eq:riccati_psiBold}, we get 
\bes{
	\boldsymbol{2} &= \;   \langle g_t, (\dot{\bold \Psi}_t - \boldsymbol{\Psi}_t \dot {\boldsymbol \Lambda}_t\boldsymbol{\Psi}_t)g_t \rangle_{L^2} 
	\; = \; 4\Big[ \Big( \big( \boldsymbol{\Psi}_t \boldsymbol{K} \eta \big)^* g_t \Big) (t) \Big]^\T C^\top C \Big( \big( \boldsymbol{\Psi}_t \boldsymbol{K} \eta \big)^* g_t \Big) (t) \\
	&= \;  (Z^2_t)^\top CC^\top Z^2_t. 
}
Combining $\bold{1}, \bold{2}$ and $\bold{3}$ yields
\bes{
	\label{eq:GammaStein}
	\frac{d \Gamma_t}{\Gamma_t}&=  \;  \Big(-2r(t)  + \left|\lambda_t +  Z^1_t + C Z^2_t \right|^2 \Big) dt + (Z^2_t)^\T  dW_t . 
}
This shows that $(\Gamma,Z^1,Z^2)$ solves \eqref{eq:riccati_sto}. \\

\vspace{1mm}

\noindent \textit{Step 3.} It remains to check  that $\left(\Gamma, Z^1, Z^2\right) \in \S^{\infty}_{\F}([0,T],\R) \times L^2_{\F}([0,T], \R^d) \times L^2_{\F}([0,T], \R^N)$.  For this, 
observe that since $\boldsymbol{\Psi}$ is a nonpositive operator over $[0,T]$,  we have the bound $0<\Gamma_t \leq e^{\int_t^T 2 r(s) ds}$. 
Finally, to show that $Z^2 \in L^2_{\F}([0,T], \R^d)$, it is enough to show that 
\bes{
	\E \Big[ \int_0^T \Big| \int_t^T K(s,t)^\top g_t(s) ds \Big|^2  dt \Big] & < \;  \infty, \\  
	\mbox{ and }   \;   \E \Big[ \int_0^T  \Big| \int_{(t,T]^2} K(v,t)^\top \psi_t(v,s) g_t(s) dv ds \Big|^2  dt \Big] &< \;  \infty.
}
This follows from the fact that $K$ and $\psi$  satisfy \eqref{assumption:K_stein}-\eqref{eq:bound_psi_leb} respectively,  and 
\bes{
	\sup_{0\leq t \leq s\leq T} \E \left[|g_t(s)|^2 \right] & \leq \;  \sup_{s\leq T} |g_0(s)|^2 \Big( 1 + \sup_{s \leq T} \int_0^T |R_D(s,u)|^2 du \Big) \; < \; \infty,
}
where $R_D$ is the resolvent of $KD$.
\end{proof}

\vspace{2mm}

From Theorem~\ref{T:verif}, we can now explicitly solve the Markowitz  problem \eqref{optimization_problem} in the quadratic Volterra  model \eqref{eq:SteinY}, \eqref{eq:Steincorel} and \eqref{eq:Steinprice}, see Theorem~\ref{T:quadracticsolution} below. In order to verify  condition (H2) of Theorem \ref{T:verif}, we will first need the following lemma  whose proof is postponed to Appendix~\ref{seclemma1}.

\begin{lemma}\label{L:lemmacondtheta1}
Let the assumptions of Lemma~\ref{lemma:existence_riccati_sto_quadratic} be in force. {Assume $|D-2\eta C^\T \Theta| \times \|K\|_{L^2([0,T]^2)}^2 < 1$}, then
    \bes{
        \label{eq:bound_lambda_plus_Z}
         |\lambda_s|^2 + {\left|Z^1_s\right|^2 + \left|Z^2_s\right|^2}  \leq \kappa(\Theta)  \left( |g_s(s)|^2 +\int_0^T |g_s(u)|^2 du \right), \quad s\leq T, \quad \Theta \in \R^{d\times N},
    }
    where {$\kappa(\Theta) =  c|\Theta|^2(1 + |\Theta|^4 \hat \kappa(\Theta))$ with $c>0$ independent of $\Theta$ and
    \bes{
         \hat \kappa(\Theta) =  \left(\frac{|f(\Theta)| \times \|K\|_{L^2([0,T]^2)}^2}{ 1 -|f(\Theta)| \times \|K\|_{L^2([0,T]^2)}^2} \right)^4.
    }}
\end{lemma}
\begin{proof}
	See Appendix~\ref{seclemma1}.
\end{proof}

We now arrive to the  main result of this section. 

\begin{theorem}\label{T:quadracticsolution}
Fix a kernel $K$ as in Definition~\ref{D:kernelvolterra} satisfying \eqref{eq:assumptionkerneldiff1} and assume  that $(U-2C^\top C) \in \S^N_+$. Let $a(p)$ be as in \eqref{eq:constap}  and $\kappa$ the function defined in Lemma \ref{L:lemmacondtheta1}.   Assume that there exists $\Theta \in \R^{d\times N}$  such that 
\begin{align}\label{eq:notsofriendlytheta}
\E\Big[ \exp \Big(  a(p){\kappa(\Theta)}  \int_0^T \big(|g_s(s)|^2 + \int_0^T |g_s(u)|^2 du\big)ds \Big) \Big] \; < \; \infty,
\end{align}
 for some $p>2$.
Assume that $g^i_0(0)>0$ for some $i\leq d$.
    Then, the optimal investment strategy for the Markowitz  problem \eqref{optimization_problem} is  given by the admissible control
    \bes{
        \label{eq:optimal_control_2}
        \c^{*}_t &= \;  -    \Big( \big({ \Theta } + 2  C  \left[\boldsymbol{\Psi}_{t} \boldsymbol{K}\eta \right]^*  \big) g_t \Big)(t) \Big(X^{\c^*}_t - \xi^* e^{-\int_t^T r(s) ds} \Big), 
    }
where $\xi^*$ is defined in \eqref{eq:eta_star},  and the optimal value is given by \eqref{eq:value_final} with $\Gamma_0$ as in \eqref{eq:Gamma_quadratic}. 
\end{theorem}

\begin{proof}
First note that  under the specification \eqref{eq:Gamma_quadratic}, and $\lambda_t = \Theta Y_t = \Theta g_t(t)$, the candidate for the optimal feedback control defined in 
\eqref{eq:optimal_control_final}  takes the form
    \bes{
    \c^*_t &= \;   -  \left(\lambda_t + { Z^1_t + C Z^2_t}{}\right) \big(X^{\alpha^*}_t  - \xi e^{-\int_t^T r(s)  ds}\big)  \\
    &= \;  \Big( \big( {\Theta } + 2  C \left[\boldsymbol{\Psi}_{t} \boldsymbol{K}\eta \right]^* \big) g_t \Big)(t) \big(X^{\alpha^*}_t  - \xi e^{-\int_t^T r(s) ds}\big).
    }
It thus suffices to check that the assumptions of Theorem \ref{T:verif} are verified to ensure that $\c^*(\xi^*)$ is optimal and to get that  \eqref{eq:value_final} is the optimal value.
    The existence of a solution triplet $ (\Gamma, Z^1, Z^2) \in { \S^{\infty}_{\F}([0,T], \R) \times L^2_{\F}([0,T], \R^d) \times L^2_{\F}([0,T], \R^N)}$ to the stochastic backward Riccati equation \eqref{eq:riccati_sto} is ensured by Lemma \ref{lemma:existence_riccati_sto_quadratic}. 
    In addition, we have 
   \bes{
        \Gamma_0 &= \; \E\Big[ e^{\int_0^T \big( 2 r(s)  - \left|\lambda_s +  Z^1_s + C Z^2_s \right|^2 \big) ds} \Big]  
        \; = \; \E\Big[ e^{ \int_0^T \big[ 2 r(s) - \big|  \big( \big( {\Theta } + 2  C \left[\boldsymbol{\Psi}_{s} \boldsymbol{K}\eta \right]^* \big) g_s \big)(s) \big|^2 \big]  ds } \Big],
    }
 which implies that $\Gamma_0 < e^{2\int_0^T r(s) ds}$ { since $g^i_0(0)>0$ for some $i\leq d$ by assumption.} Thus  condition (H1)  of Theorem \ref{T:verif} is verified. Condition (H2) follows directly from Lemma~\ref{L:lemmacondtheta1} and \eqref{eq:notsofriendlytheta}. The proof is complete.
\end{proof}

\vspace{1mm}

The following lemma  provides a general sufficient condition for the existence of $\Theta$ satisfying \eqref{eq:notsofriendlytheta}. Without loss of generality, we  assume that $D=0$ in \eqref{eq:SteinY}.\footnote{If $D \neq 0$, then making use of the resolvent kernel $R_D$ of $KD$, we reduce to the case $D=0$ as illustrated on \eqref{eq:Yclosed} by working on the kernel $(K+R_D)$ instead of $K$.} 
{Define
     $Z(s,u) = (\frac 1 T g_s(s), g_s(u))^\top$ for any $s,u \in [0,T]$, which we view as a random variable in $L^2([0,T]^2, \R^{2N})$. Its mean is given by $\mu(s,u)=\E[Z(s,u)]=(\frac 1 T g_0(s),g_0(u))^\top$ and its covariance kernel  by
       \begin{align}
            \bar \Sigma((s,u),(t,r)) = \E\left[ \Big(Z(s,u) - \E(Z(s,u))\Big)    \Big(Z(t,r) - \E(Z(t,r))\Big) ^\T\right], \qquad s,u,t,r \in [0,T],
       \end{align}
    which is symmetric and nonnegative. {It follows from assumption \eqref{assumption:K_stein} that $\bar \Sigma$ is continuous  on $[0,T]^4$} so that an application of Mercer's theorem,  see \citet[Theorem 1 p.208]{shorack2009empirical},
yields the existence of a countable orthonormal basis $(e^n)_{n \geq 1}$ in $L^2([0,T]^2, \R^{2N})$ and a non increasing sequence of nonnegative numbers $(\lambda^n)_{n\geq 1}$, with $\lambda^n \to 0$, as $n \to \infty$, such that 
 \bes{\label{eq:decompbarsigma}
        \bar \Sigma((s,u),(t,r)) = \sum_{n\geq 1} \lambda^n e^n(s,u) e^n(t,r)^\T. 
    }  
In addition, we observe by virtue of \eqref{assumption:K_stein} that 
\begin{align}\label{eq:sumlambda_}
\sum_{n \geq 1} \lambda^n=\tr (\bold{\bar \Sigma}) =& \frac{1}{T} \int_0^T \left( \int_0^s \tr\left( K(s,z)\eta U \eta^T K(s,z)^\T  \right) dz\right)ds  \\
& + \int_0^T \left( \int_0^T \left( \int_0^u \tr\left( K(s,z)\eta U \eta^T K(s,z)^\T  \right) dz\right)ds \right) du < \infty. 
\end{align}
}


\begin{lemma}\label{L:lemmacondtheta2} Set $D=0$. Let $a>0$ be such that $2a < \frac 1 {\lambda^1}$.
Then, 
        \begin{align*}
        \E\Big[\exp\Big( {a} \int_0^T\big( |g_s(s)|^2 + \int_0^T |g_s(u)|^2 du \big)ds\Big)  \Big] \; < \; \infty.
        \end{align*}
        In particular, \eqref{eq:notsofriendlytheta} holds if {$2a(p)\kappa(\Theta)<\frac 1 {\lambda^1}$} for some $p>2$.
\end{lemma}
\begin{proof}
We refer to Appendix~\ref{seclemma2}. 
\end{proof}

\begin{remark}
\label{rk:cond_theta}
In practice, {as $\lambda_1 \leq \tr(\bold{\bar \Sigma})$}, it follows from Lemma~\ref{L:lemmacondtheta2}  and \eqref{eq:sumlambda_}, that a sufficient condition for the existence of 
$\Theta$ satisfying \eqref{eq:notsofriendlytheta}   would be $$2a(p)\kappa(\Theta)<\frac 1 {{\tr(\bold{\bar \Sigma})}  }.$$
{For instance, for the fractional convolution kernel $K(t,s)$ $=$ $\bm 1_{s\leq t}(t-s)^{H-1/2}$, we have $\int_0^T \int_0^T |K(t,s)|^2 ds dt$ $=$ $T^{2H+1}$. Consequently $\tr(\bold{\bar{\Sigma}}) \geq \eta^2(T^{2H} + T^{2(H+1)})$ and the condition on $\Theta$ reads
\bes{
    \kappa(\Theta) \leq (2a(p)\eta^{2}(T^{2H} + T^{2(H+1)}))^{-1}. 
}
}
\end{remark}

\vspace{1mm}

The following corollary treats the standard Markovian and semimartingale case for $K=I_N$ and shows how to recover the well-known formulae in the spirit of  \cite{chiu2014mean}.

\begin{corollary}\label{C:chui}
 Set $K(t,s) = I_N 1_{s \leq t}$ and {$g_0(t)\equiv Y_0$ for some $Y_0 \in \R^N$}. Then,  the solution to the Riccati BSDE can be re-written in the form 
    \bes{
        \label{eq:classic_Z2}
        \Gamma_t  = \exp\big( \phi_{t}  + Y_t^\T P_t Y_t \big),  \quad \mbox{and} \quad {Z^2_t = 2 \eta^\T  P_t Y_t  ,}
    }
    where  $P:[0,T]\mapsto \R^{N\times N}$ and $\phi$ solve the conventional system of $N\times N$-matrix Riccati equations
\bec{
        \dot{P}_t &= \;  { \Theta^\T \Theta} + P_t (2\eta C^\T \Theta-D) +  (2\eta C^\T \Theta-D)^\T P_t +  2 P_t(\eta (U - 2 C^\T C)\eta^\T)P_t, \\
        P_T &= \; 0,\\
          \dot{\phi}_t &= \;  - 2r(t) - \tr(P_t \eta U \eta^\T), \quad t \in [0,T], \\
        \phi_T &= \; 0.
    }
    {Furthermore, the optimal control reads 
    \bes{
        \label{eq:classic_alpha}
        \c^{*}_t &= \;  -    \Big( { \Theta } + 2  C  (D\eta)^\T   P_t Y_t     \Big) \Big(X^{\c^*}_t - \xi^* e^{-\int_t^T r(s) ds} \Big).
    }
    }
\end{corollary}

\begin{proof}
    For $K(t,s) = I_N 1_{s \leq t}$,
    $$ Y_s=Y_t + \int_t^s DY_u du + \int_t^s \eta dW_u, \quad s\geq t,$$
    so that
    the adjusted forward process reads  
    $$g_t(s) = \E\Big[ Y_s -\int_t^s DY_u du \mid \mathcal F_t \Big] =\mathbf{1}_{t \leq s} Y_t,$$ 
    and the solution to the Riccati BSDE can be re-written in the form
    \bes{
        \Gamma_t =& \exp\big( \phi_{t}  + \langle g_t, \boldsymbol{\Psi}_{t} g_t \rangle_{L^2} \big) = \exp\big( \phi_{t}  + Y_t^\T P_t Y_t \big),
    }
    where $P_t = \int_t^T (\boldsymbol{\Psi}_{t} \bold{1}_t )(s)ds$ with the $\R^N$-valued indicator function $\bold{1}_t:(s)\mapsto(1_{t\leq s},\ldots,1_{t\leq s})^\top$. We now derive the equations satisfied by $P$ and $\phi$.  First we have $K_T =0$ and 
    \bes{
        \dot{P}_t =& \;  -(\boldsymbol{\Psi}_{t} 1_t)(t) + \int_t^T \frac{d (\boldsymbol{\Psi}_{t} 1_t)(s)}{dt} ds \\
        =& \; -(\boldsymbol{\Psi}_{t} 1_t)(t) + \int_t^T (\dot{\boldsymbol{\Psi}}_t 1_t)(s) ds - \int_t^T \psi_t(s,t)ds\\
        =& \; \boldsymbol{1} + \boldsymbol{2} + \boldsymbol{3}.
    }
    Using Lemma \ref{L:Psi}--\ref{eq:Psi_on_boudary} and the expression $\hat{K}(s,u) = 1_{u \leq s}(D-2\eta C^\T \Theta)$ we get
    \bes{
    \boldsymbol{1} &= (-\Theta^\top \Theta \id + \boldsymbol{\hat{K}}^*\boldsymbol{\Psi}_t )(1)(t)= -\Theta^\top \Theta + (D-2\eta C^\T \Theta)^\T P_t.
    }
    Furthermore, Lemma \ref{L:Psi}--\ref{L:Psi3} and $\dot{\Sigma}_t(s,u) = 1_{t \leq s \wedge u} \eta (U - 2 C^\T C)\eta^\T$  yield 
    \bes{
         \boldsymbol{2} &= \int_t^T (\dot{\boldsymbol{\Psi}}_t 1_t)(s) ds    \; = \;  \int_t^T (\boldsymbol{\Psi}_t \dot{\boldsymbol{\Sigma}}_t \boldsymbol{\Psi}_t 1_t)(s)ds \\
         &= \left(\int_t^T (\boldsymbol{\Psi}_t 1_t)(s)ds \right) \eta (U - 2 C^\T C)\eta^\T \left(\int_t^T (\boldsymbol{\Psi}_t 1_t)(s)ds \right)  \\
         &= P_t (\eta (U - 2 C^\T C)\eta^\T) P_t.
    }
 Moreover, by using Lemma \ref{L:Psi}--\ref{L:Psi1}-\ref{L:Psi2},  we obtain
    \bes{
         \boldsymbol{3} &= - \int_t^T \psi_t(s,t)ds = - (\boldsymbol{\Psi}_t + \Theta^\T \Theta id)^*(1_t) =- (\hat{\boldsymbol{K}}^* \boldsymbol{\Psi}_t)^*(1) (t) =- P_t (D-2\eta C^\T \Theta).
 }This proves the equation for $P$, and  that of $\phi$  is immediate. 
 Finally to prove the formula of $Z^2$ in \eqref{eq:classic_Z2} and $\c^*$ in \eqref{eq:classic_alpha} it suffices to observe the following identity 
 \bes{
      \big( ( \boldsymbol{\Psi}_t \boldsymbol{K} \eta)^* g_t \big) (t) = \eta^\T  P_t Y_t.
 }
\end{proof}

\section{Numerical experiment: rough Stein-Stein for two assets} \label{secnum} 

We illustrate the results of Section~\ref{S:quadratic} on a special case of the two dimensional rough Stein-Stein model as described in Example~\ref{E:quadratic}.  We consider a four dimensional Brownian motion 
$(B^1,B^2,B^{1,\perp},B^{2,\perp})$,  and define 
\begin{align*}
\tilde B^1 = B^1, \quad \tilde B^2 = \rho B^1 + \sqrt{1-\rho^2} B^2, \quad 
W^i=c_i \tilde B^i + \sqrt{1-c_i^2} \tilde B^{\perp,i},
\end{align*}
for some $\rho$ $\in$ $[-1,1]$, and $c_i \in [-1,1]$, $i=1,2$.

For simplicity we set $r\equiv 0$, and consider two stocks of price process  $S^1$ and $S^2$ with the following  dynamics\footnote{This corresponds to Example \ref{E:quadratic}-(i) with $(\beta_{11},\beta_{12},\beta_{21},\beta_{22}) = (1,0, \rho, \sqrt{1-\rho^2})$ and  $\Theta = \beta^{-1} \diag{(\theta_1,\theta_2)}$.}
\begin{equation}
\begin{cases}
dS^i_t &= \; S^i_t \theta_i (Y^i_t)^2     dt + S^i_t Y^i_t d\tilde B^i_t, \\
\; Y^i_t &=  \;  Y^i_0 +  \frac{1}{\Gamma(H_i+1/2)}  \int_0^t {(t-s)^{H_i-1/2}} \eta_i dW_s^i,
\quad i=1,2,
\end{cases}
\end{equation}
with $H^i>0$, $\eta_i,\theta_i \geq 0$ and $Y^i_0\in \R$. 

Although the framework of Section~\ref{S:quadratic} allows for a more general correlation structure for the Brownian motion, the model is already rich enough to  capture the following stylized facts:
\begin{itemize}
\item the two stocks $S^i$, $i$ $=$ $1,2$,  are correlated through $\rho$, 
\item  each stock $S^i$ has a stochastic rough volatility $|Y^i|$ with possibly different  Hurst indices $H_i$,
\item each stock $S^i$  is correlated with its own volatility process through $c_i$ to take into account the leverage effect.
\end{itemize}

Our main motivation for considering the multivariate rough Stein-Stein model is to study   the `buy rough sell smooth' strategy of \citet{glasserman2020buy} that was backtested empirically:  this strategy consisting in buying the roughest assets while shorting on the smoothest ones was shown to be profitable. We point out that the numerical simulations for the one dimensional rough Heston model carried  in  \citet{han2019mean}  by varying the Hurst index $H$ could not provide much insight on such strategy, apart from suggesting that the vol-of-vol has a possible impact on the `buy rough sell smooth strategy'. Our quadratic multivariate framework allows for  more flexible simulations, with a richer correlation structure compared to multivariate extensions of the rough Heston model, recall Remark~\ref{R:correlStein}. Our results below provide new insights on the strategy by showing that the correlation between stocks plays a key role.

  Our present goal is to illustrate the influence of some parameters, namely the horizon $T$, the vol-of-vol $\eta$ and the correlation  $\rho$ between the stocks, onto the optimal investment strategy when two assets, one rough and one smooth with $H_1 < H_2$, are at stakes. To ease comparison, we set $c_1=c_2=-0.7$ for the leverage effects,  $Y^1_0=Y^2_0$ and we normalize the vol-of-vols by setting $\eta_1=\eta_2$. We consider  the evolution of  optimal   vector of amount invested  into each stock, i.e.,  $t \mapsto \pi_t^*$  (recall that $ \alpha_t^*=\sigma_t^\top \pi_t^*$ {with $\sigma=\diag(Y^1,Y^2) \beta$} and $\alpha^*$ is given by \eqref{eq:optimal_control_2}). 
 $\pi$ being a stochastic process,  we also consider the deterministic function $t  \mapsto ( ({ \Theta } + 2  C  \left[\boldsymbol{\Psi}_{t} \boldsymbol{K}\eta \right]^*  ) Y_0 )(t) (\xi^*)$, {where $\xi^*$ is defined in \eqref{eq:eta_star}}, to help us in our analysis.\\ 
 
For our implementation of $\alpha^*$ given by \eqref{eq:optimal_control_2}, we discretize in time the operators acting on $L^2$, so that the kernel of the operator  $\Psi$ in \eqref{def:riccati_operator} is  approximated by a  finite dimensional  matrix (see for instance \citet[Section 2.3]{jaber2019laplace} for a similar procedure) and the Gaussian process $(g_t(s))_{t\leq s\leq T}$ defined in \eqref{eq:processg_quadratic} is simulated by Cholesky's decomposition algorithm. We refer to  the following \href{https://colab.research.google.com/drive/1P_SYE3WgFgwUKpOo8uCBDdIC04XyxE2a?usp=sharing}{url} for the full code and  additional simulations. \\

 Our observations from the simulations are the following.\\

 \textbf{1. Horizon $T$:} With the goal of understanding the effect of the horizon $T$ on the investment strategy, we fix all parameters but $T$ with $\rho=0$. The results are illustrated on Figures \ref{fig:horizon_T_1}-\ref{fig:horizon_T_2}-\ref{fig:horizon_T_3} and \ref{efficient_frontier_small_T}-\ref{efficient_frontier_medium_T}-\ref{efficient_frontier_big_T}. We can distinguish 3 regimes:
    \begin{itemize}
        \item $T \ll 1$ : When the investment horizon is close  to the end, the rough asset is overweighted over the smooth one.
        \item $T \approx 1$ : A transition appears, as the smooth asset is first overweighted  and then the rough asset becomes overweighted as we approach the final horizon.
        \item $T \gg 1$:  The smooth asset is overweighted all along the experiment, letting its first position only when the maturity is close, suggesting that the transition point becomes closer to $T$ as $T$ grows.
    \end{itemize}
    One possible interpretation of this transition is the following. Rough processes are more volatile than smooth processes in the short term but less volatile in the long term, since their variances evolve approximately as $t^{2H}$. Thus, when there is not much time left, it seems natural to look for rough processes to obtain some performance. {Conversely, the more time we have, the more we favor the smooth asset.} \\

\begin{figure}[h!]
\centering
\subfloat[$T = 0.5$]{
\label{fig:horizon_T_1}
    \includegraphics[width=70mm]{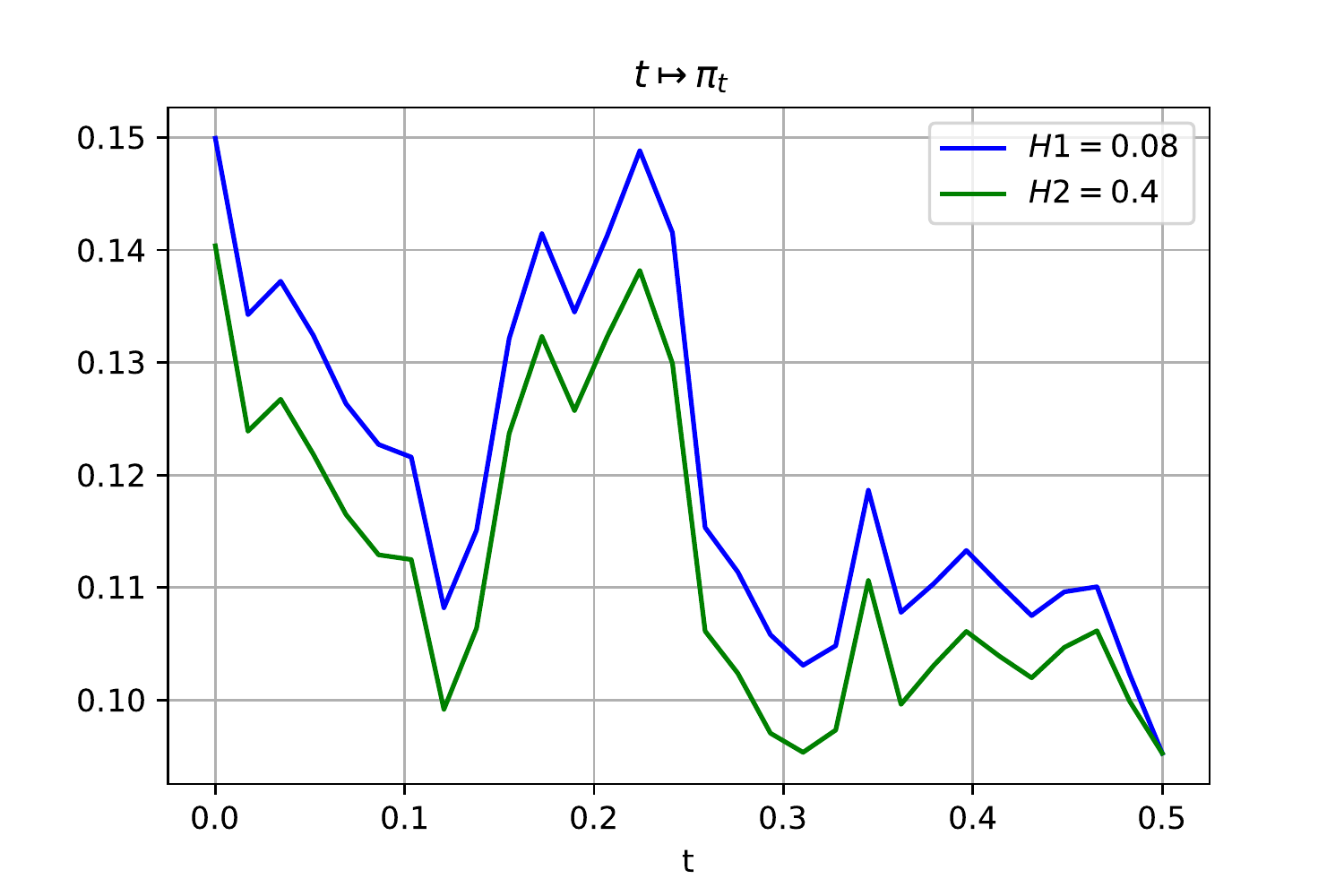}
    \includegraphics[width=70mm]{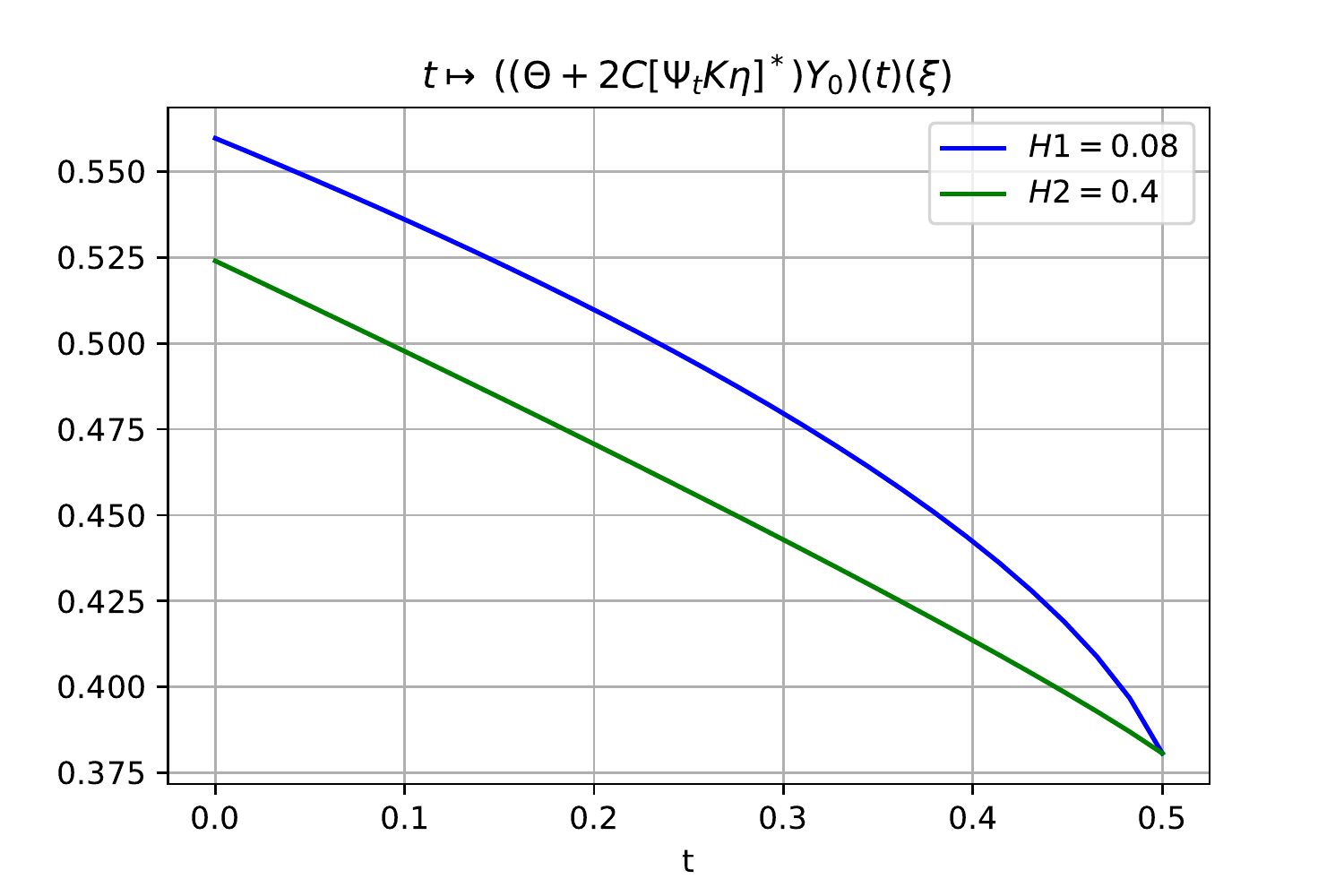}
    }
\vspace{-0.3cm}
\subfloat[$T = 1.5$]{
\label{fig:horizon_T_2}
    \includegraphics[width=70mm]{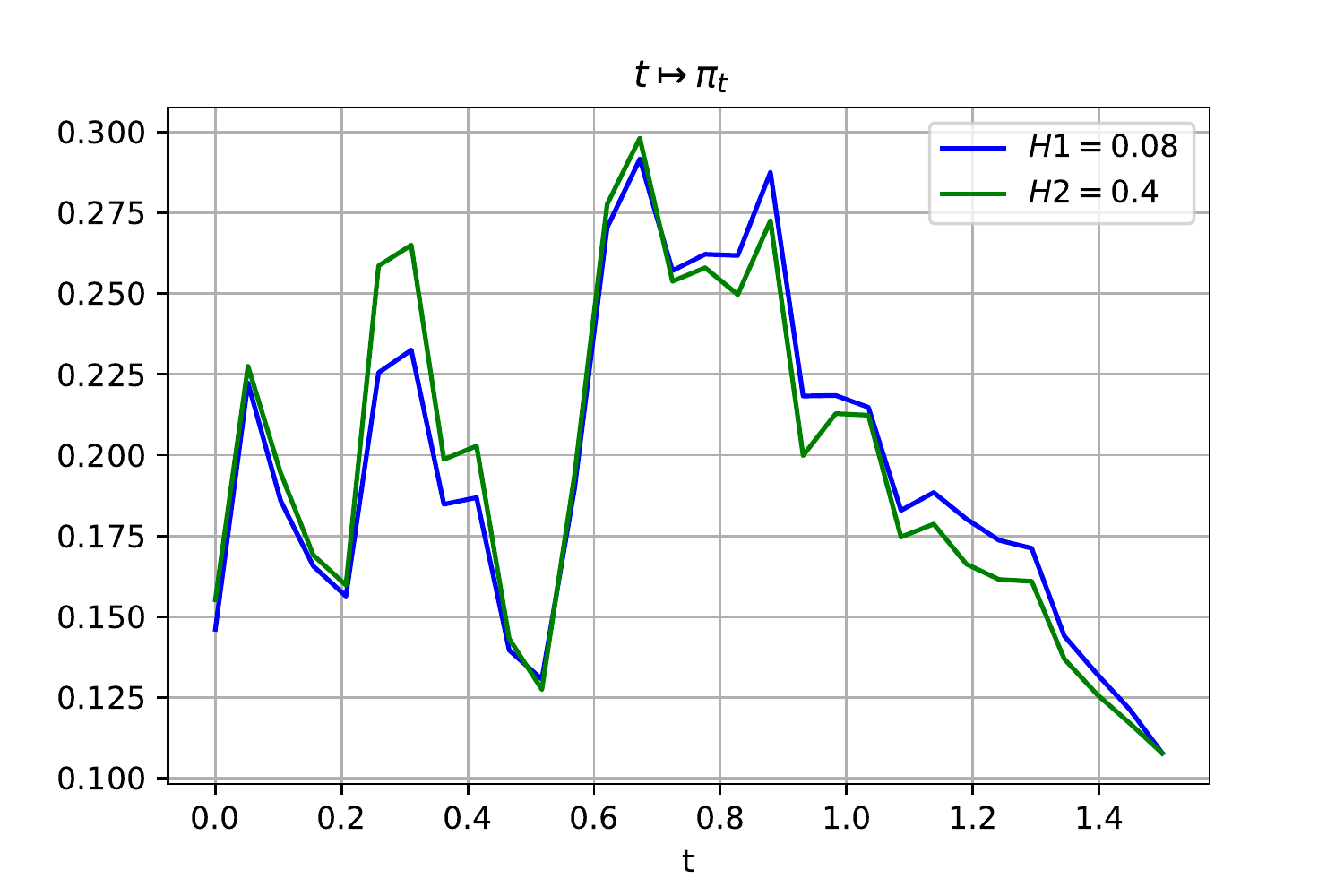}  
    \includegraphics[width=70mm]{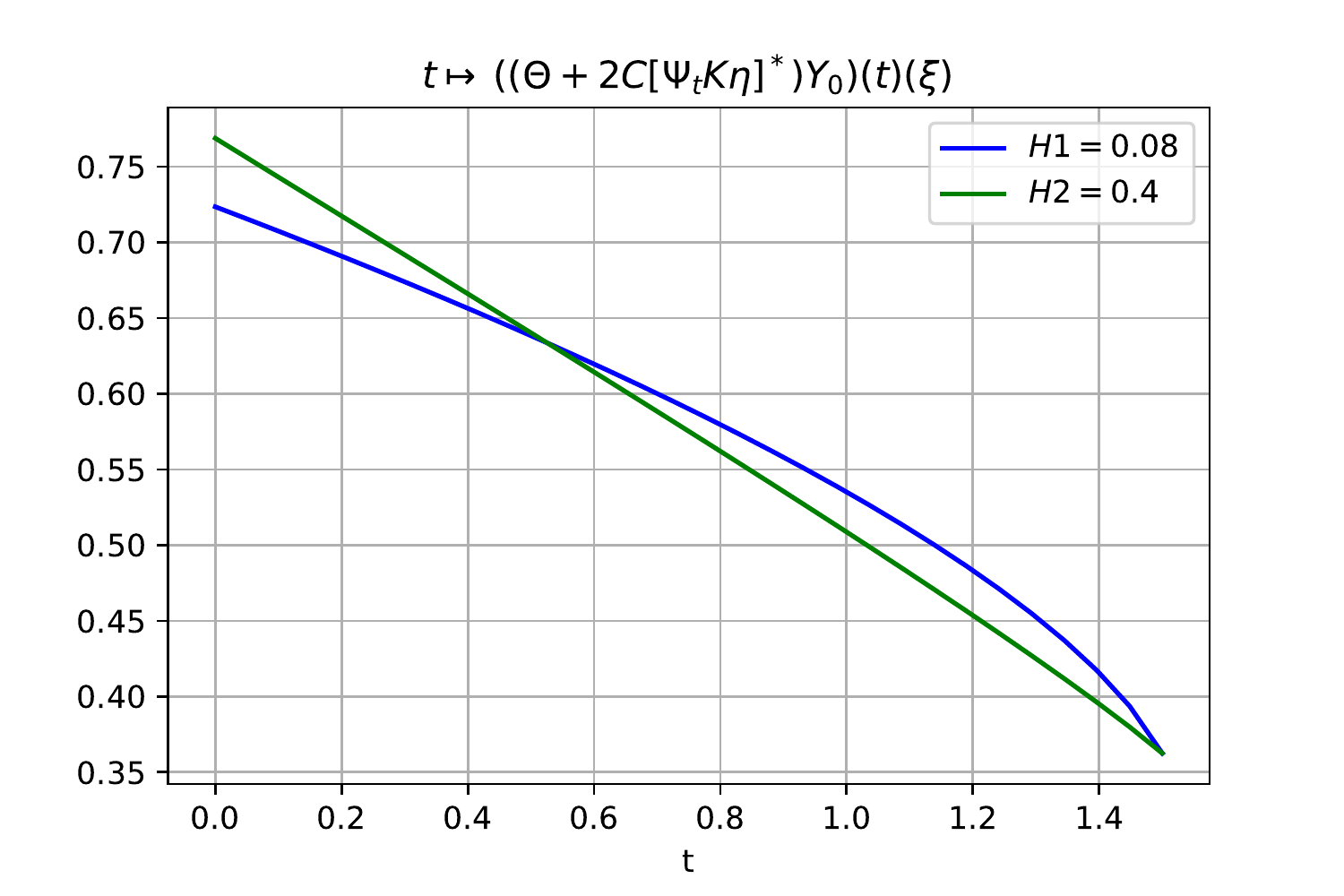}
    }
\vspace{-0.3cm}
\subfloat[$T = 2.4$]{
\label{fig:horizon_T_3}
    \includegraphics[width=70mm]{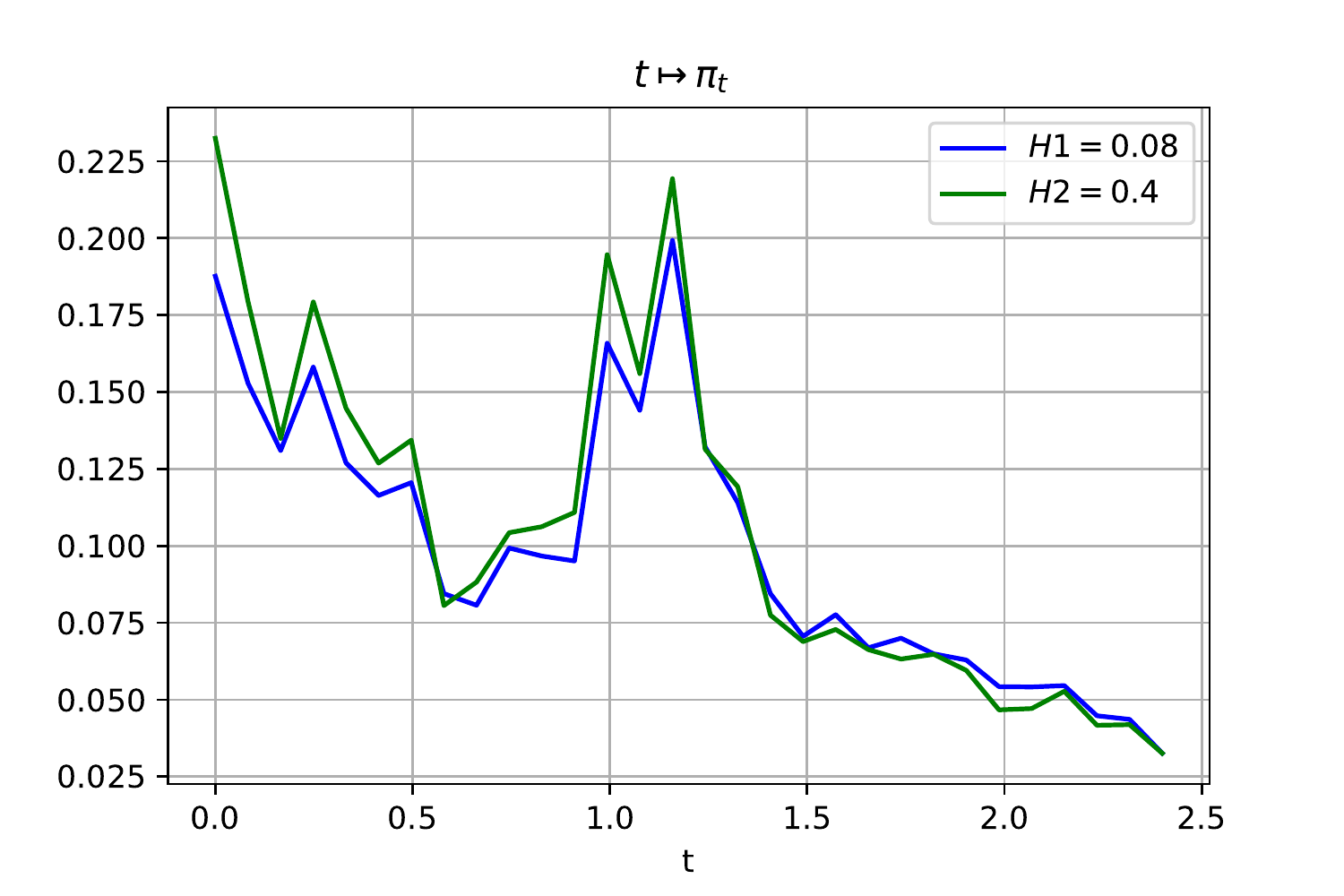}
    \includegraphics[width=70mm]{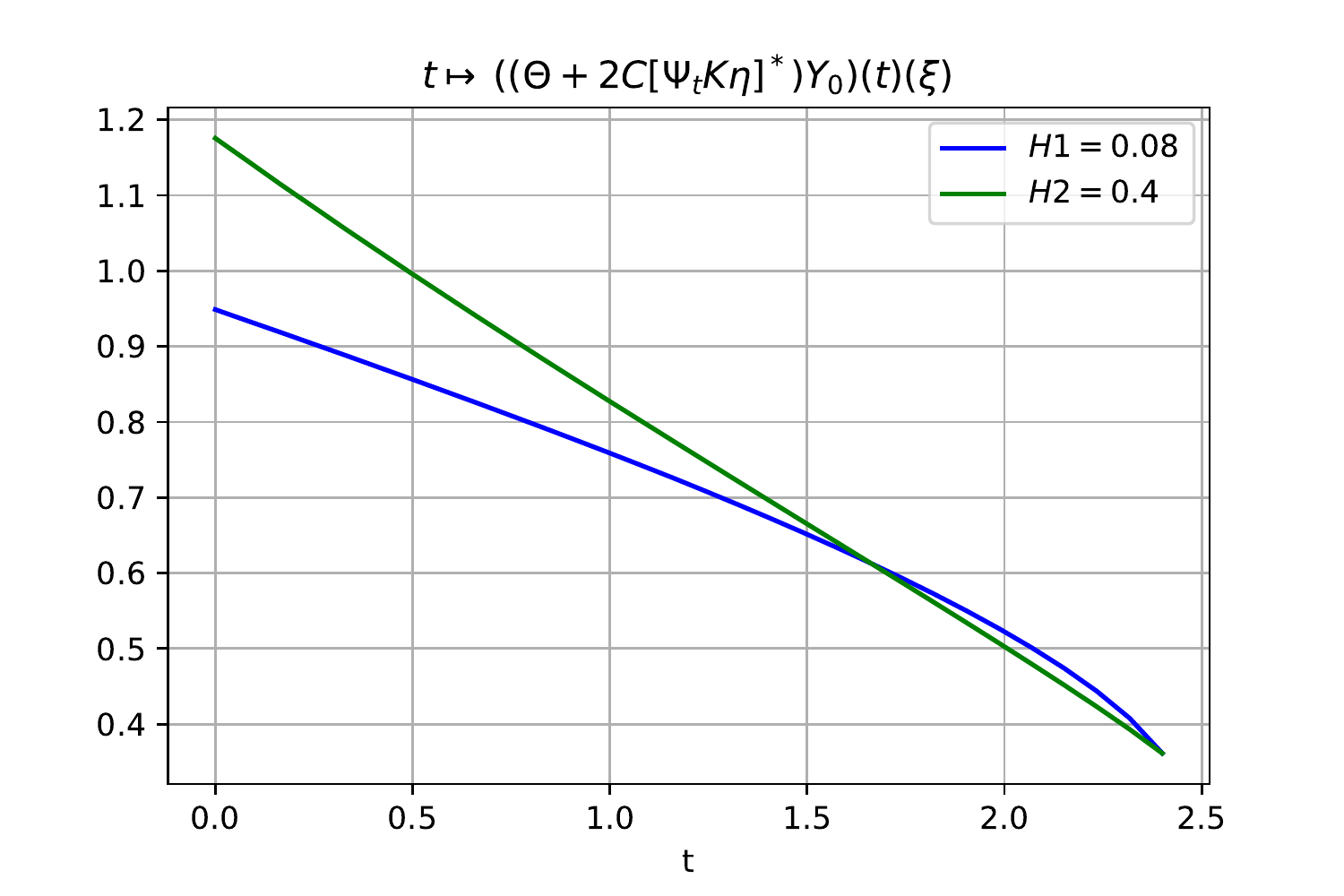}
    }
\caption{Effect of the horizon $T$ on the optimal allocation strategy. When the horizon $T$ approaches, the rough stock in blue is preferred. When $T$ is big enough and the horizon far enough the smooth stock in green is preferred. (The parameters are: $H_1 = 0.08$, $H_2=0.4$, $\rho=0, \eta_1 = \eta_2 =1, c_i = -0.7$.) } 
\end{figure}

    \begin{figure}[h!]
\centering
\subfloat[$T = 0.5$]{
	\label{efficient_frontier_small_T}
    \includegraphics[width=50mm]{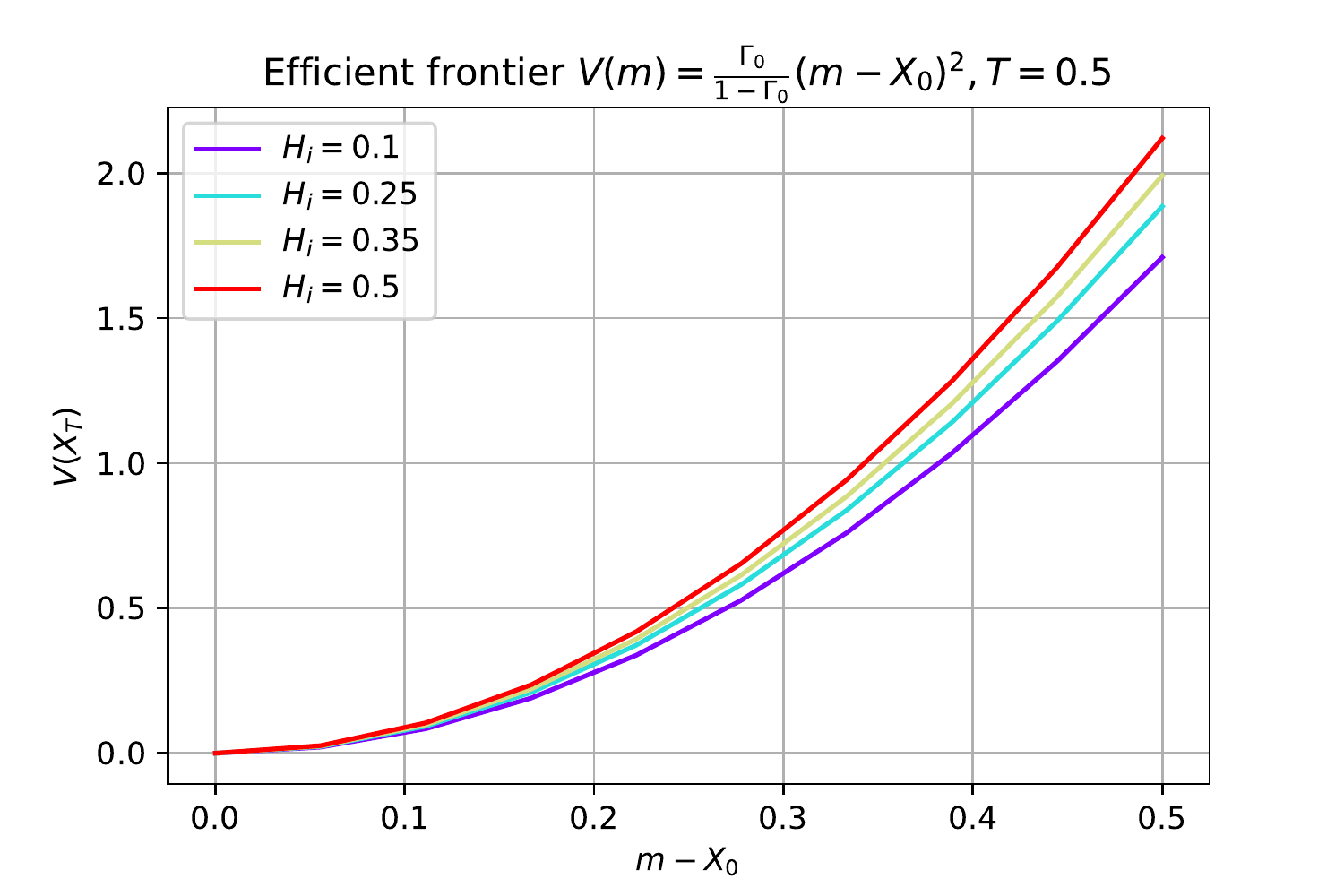}}
\subfloat[$T = 1.9$]{
	\label{efficient_frontier_medium_T}
    \includegraphics[width=50mm]{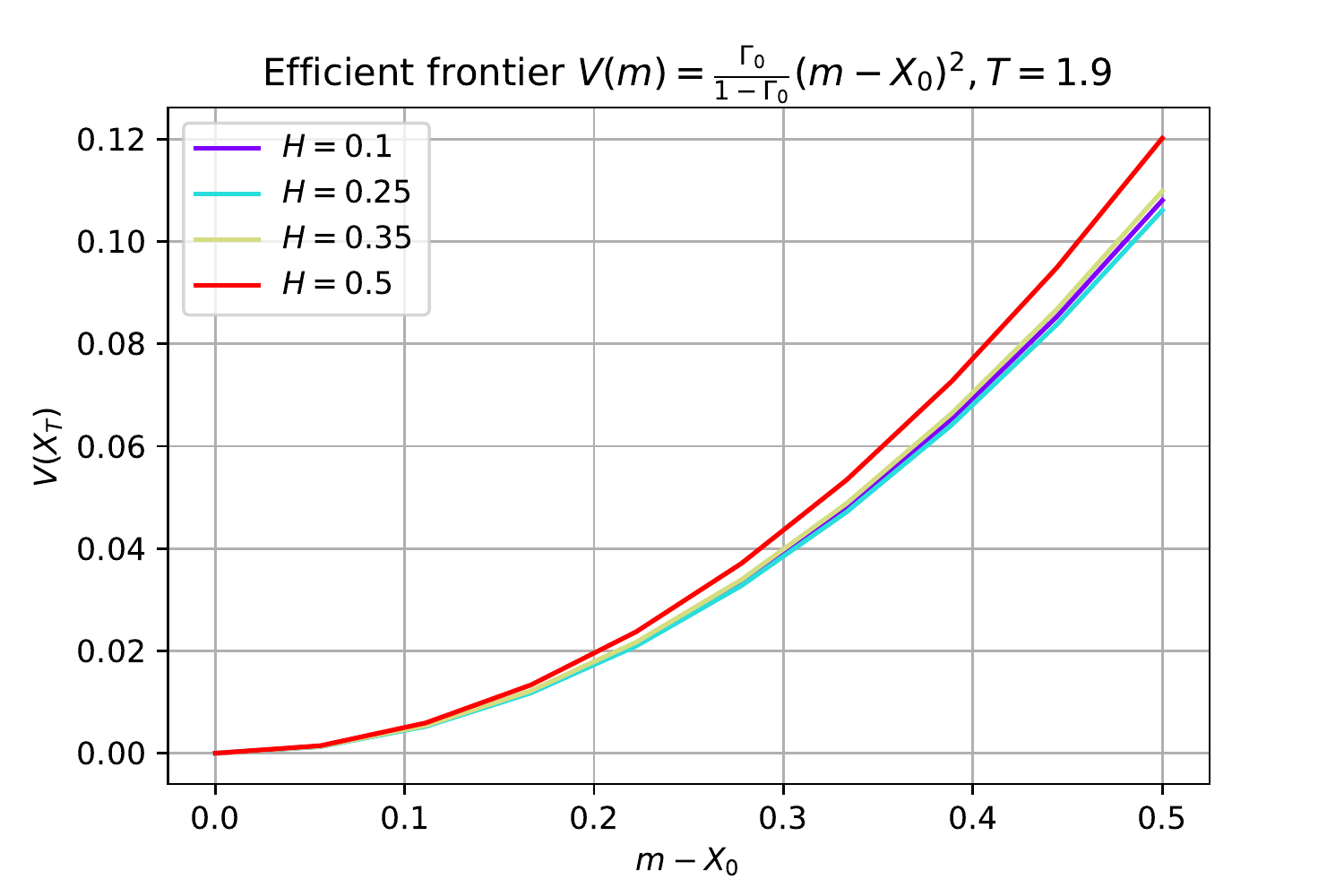}
  }
\subfloat[$T = 2.8$]{
	\label{efficient_frontier_big_T}
    \includegraphics[width=50mm]{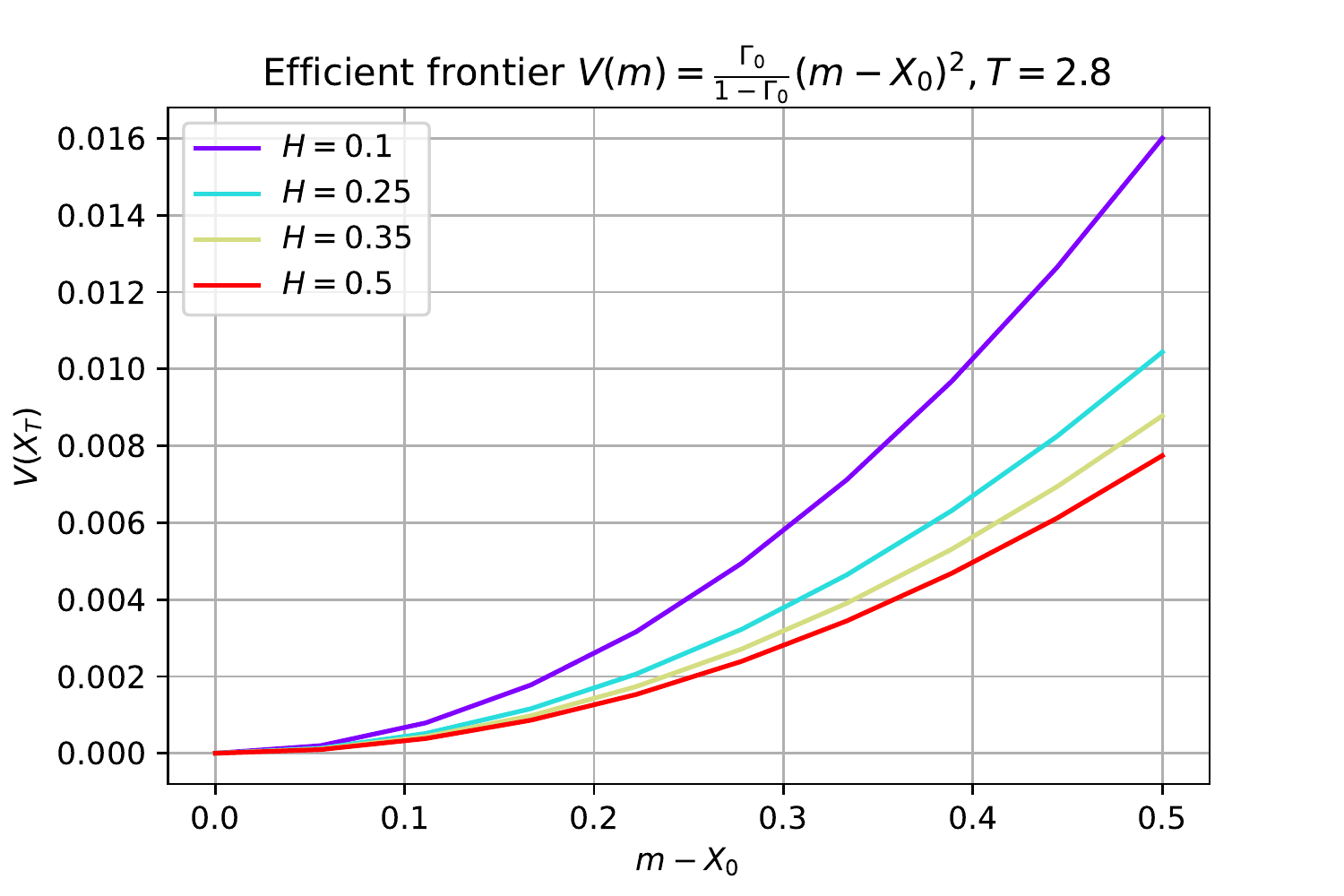}}
 
\caption{The efficient frontier in the case where both assets have the same roughness $H_1 = H_2 = H$. When the horizon $T$ is small, the rough stocks allows for lower variance. When $T$ increases we observe a transition and an inversion of the relation order. Indeed, when $T$ increases, it is the smoothest stocks that allow for a lower variance.}
\end{figure}

  \textbf{2. Vol-of-vol $\eta$:} The volatility of volatility seems to have the opposite effect of the horizon $T$ over the investment strategy as shown on Figures \ref{small_eta}-\ref{big_eta}.
    \begin{itemize}
        \item $\eta \ll 1 $ : The smooth asset and then the rough asset are successively overweighted.
        \item $\eta \gg 1 $ : The rough asset is overweighted.
    \end{itemize}
It is  quite natural to expect the vol-of-vol to have an inverse effect when compared to the horizon $T$, since  increasing the vol-of-vol is similar to accelerating the time scale at a certain rate depending on $H$ 
(think of the self-similarity property of fractional Brownian motion).   \\

\begin{figure}[h!]
\centering
\subfloat[$\eta = 0.01$]{
	\label{small_eta}
    \includegraphics[width=75mm]{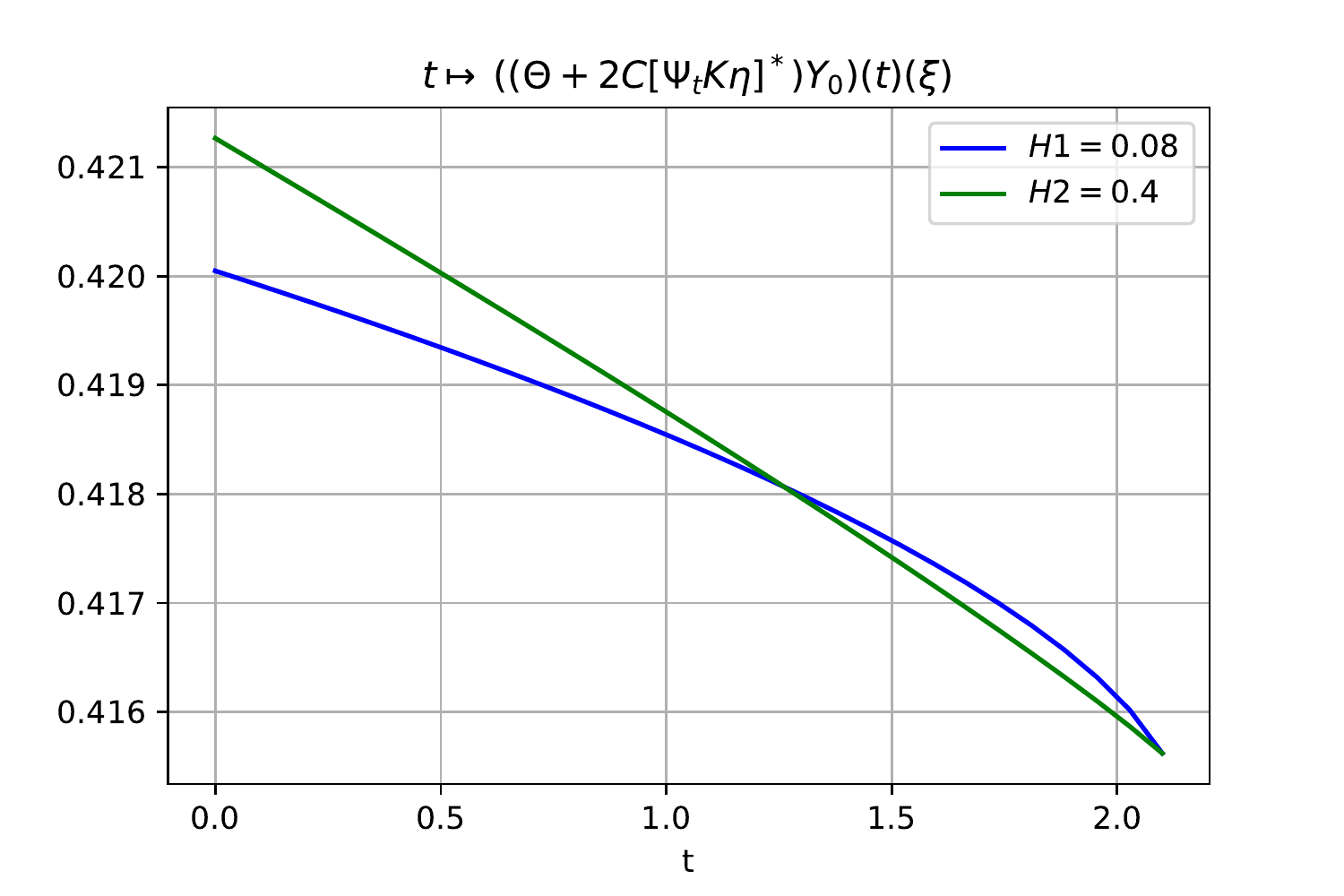}}
\subfloat[$\eta = 1.8$]{
	\label{big_eta}
    \includegraphics[width=75mm]{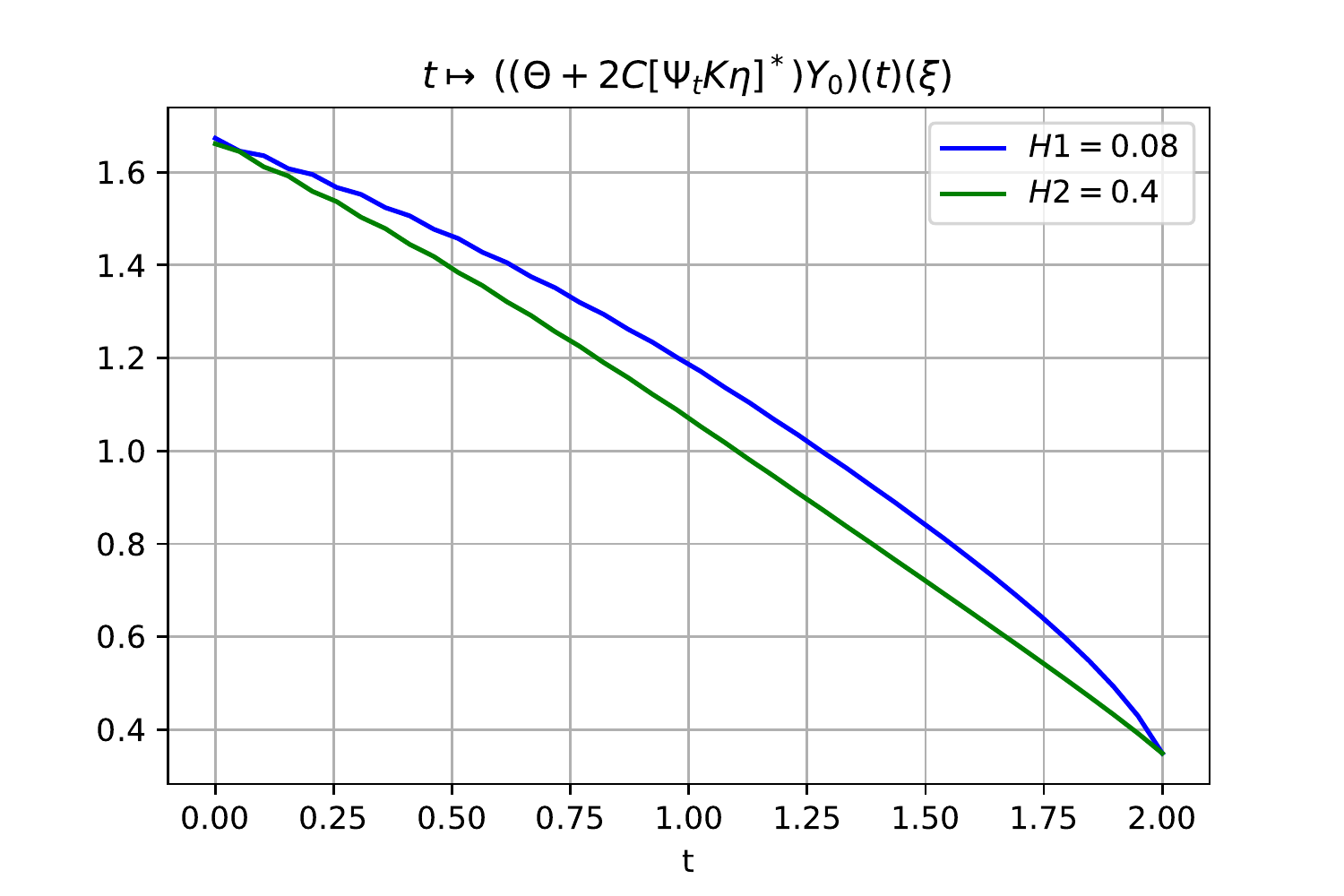}
  }
 
\caption{As the vol-of-vol $\eta$ increases, it is as if the horizon $T$ was decreasing and the rough stock in blue begins to be preferred. $H_1 = 0.08$, $H_2=0.4$, $T=2.1, \rho=0, c_i = -0.7$.}
\end{figure}

 \textbf{3. Correlation $\rho$: } 
    \begin{itemize}
        \item $\rho <0$ : In the case of negatively correlated assets it is natural to expect the following strategy : pick both assets in order to be protected from volatility and benefit from the drift. So we expect the case $\rho<0$ to be similar from $\rho=0$ except that the transition from $T \ll 1$ to $T \gg 1$ should appear at a greater $T$. This is what we observe on Figures \ref{fig:horizon_T_1_rho_neg}-\ref{fig:horizon_T_2_rho_neg}-\ref{fig:horizon_T_3_rho_neg}. We interpret this evolution towards the equally weighted portfolio as the possibility to be protected from volatility by holding both assets.
        \item $\rho >0$ :  when the two stocks are positively correlated with $\rho>0$, there is no  minimization of variance through  diversification by going long in both assets. Thus in the case a positively correlated assets, it is natural to expect the emergence of a starker choice between the assets. In the $\rho > 0$ case, see Figures \ref{pi_buy_rough_sell_smooth}-\ref{op_buy_rough_sell_smooth}, we observe a \textit{buy rough sell smooth} strategy as the one empirically found in \citet{glasserman2020buy}. 
    \end{itemize}

\begin{figure}[h!]
\centering
\subfloat{
	\label{pi_buy_rough_sell_smooth}
    \includegraphics[width=70mm]{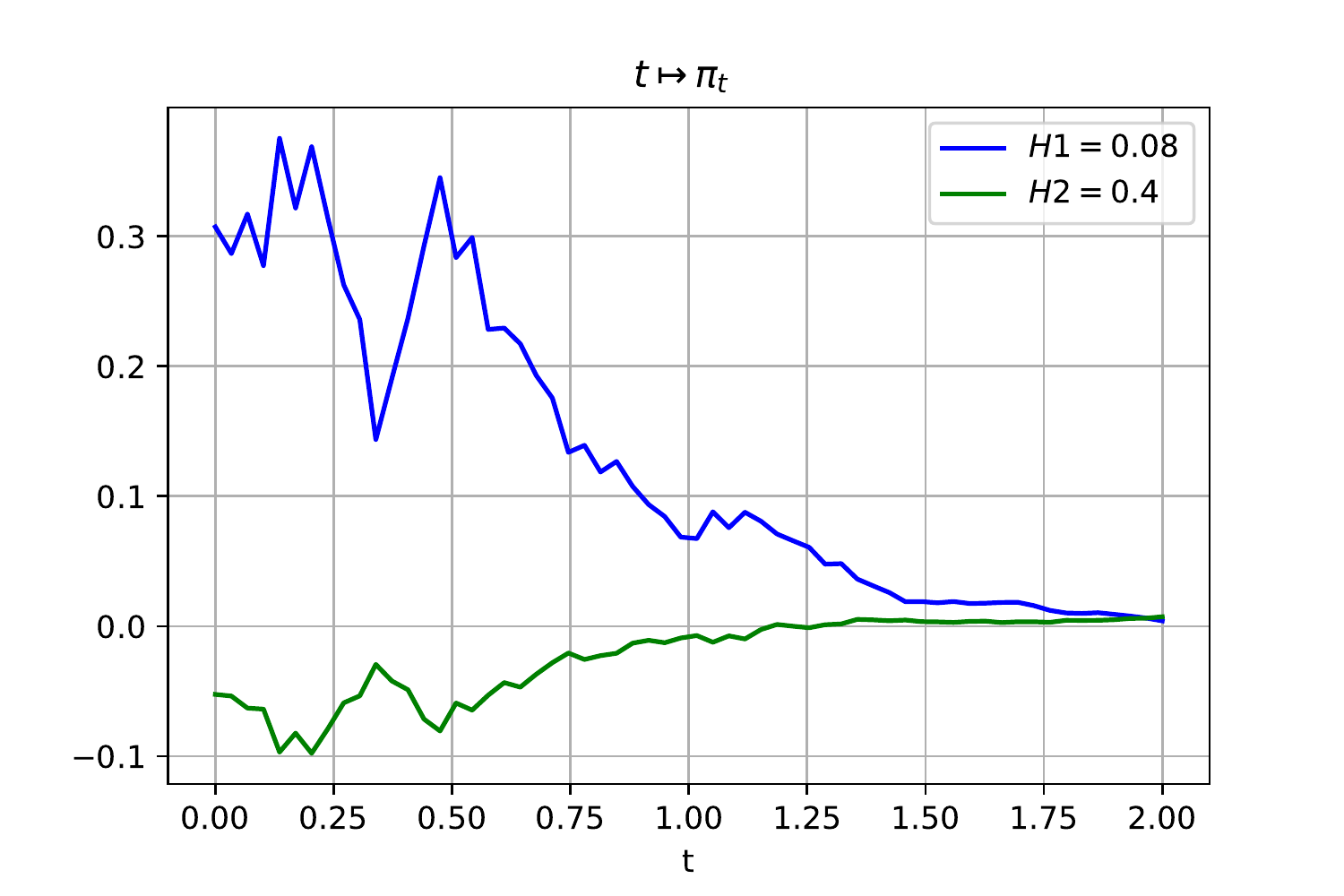}}
\subfloat{
	\label{op_buy_rough_sell_smooth}
    \includegraphics[width=70mm]{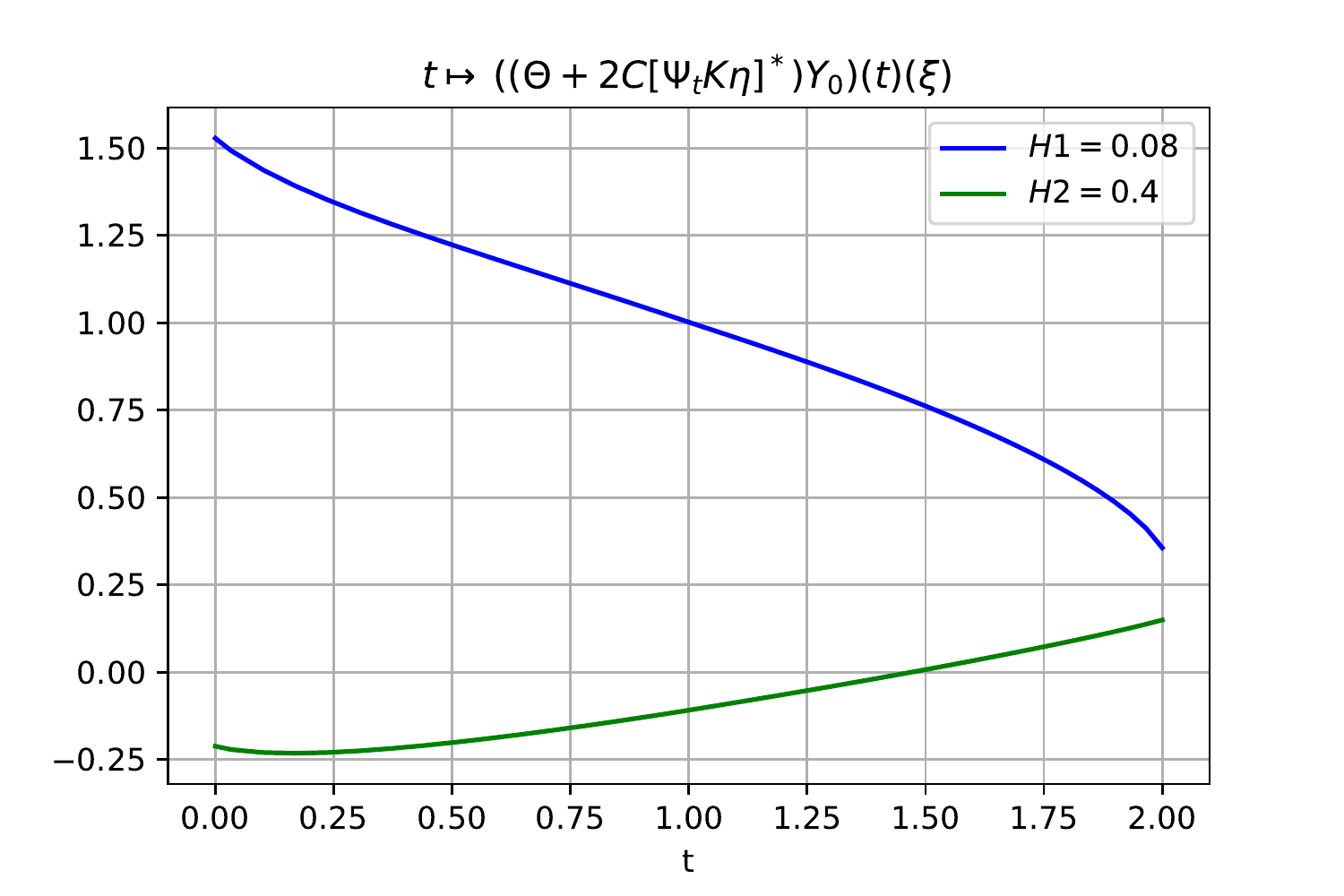}
  }
 
\caption{$\rho = 0.7$, when the two assets are positively correlated we recover \textit{the buy rough sell smooth} strategy as it is described in \cite{glasserman2020buy}. (the parameters are: $H_1 = 0.08$, $H_2=0.4$, $T=2.1$, $\eta_1 = \eta_2 = 1$, $c_i=-0.7$.)}
\end{figure}

\begin{figure}[p]
\centering
\subfloat[$T = 0.5$]{
\label{fig:horizon_T_1_rho_neg}
    \includegraphics[width=70mm]{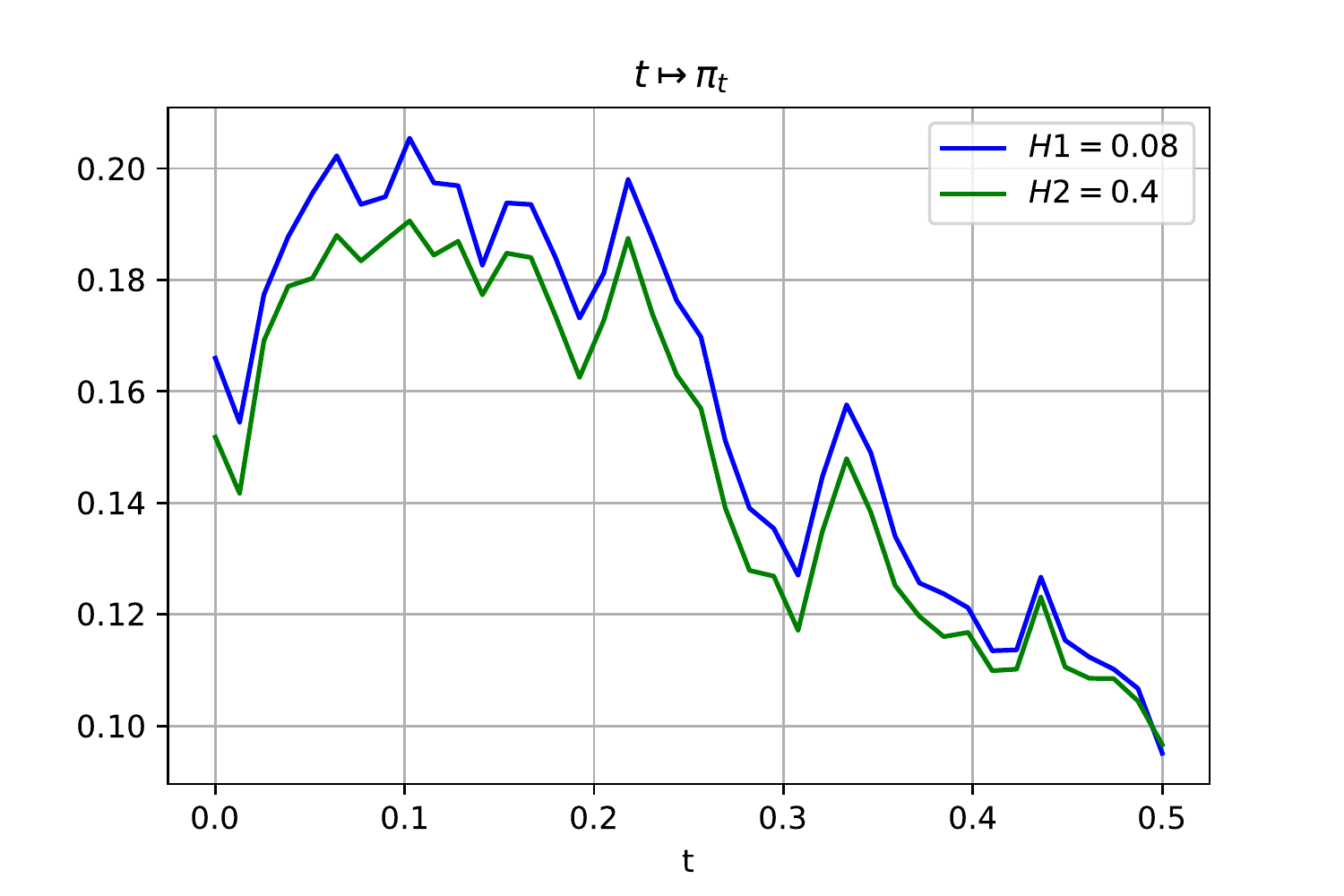}
    \includegraphics[width=70mm]{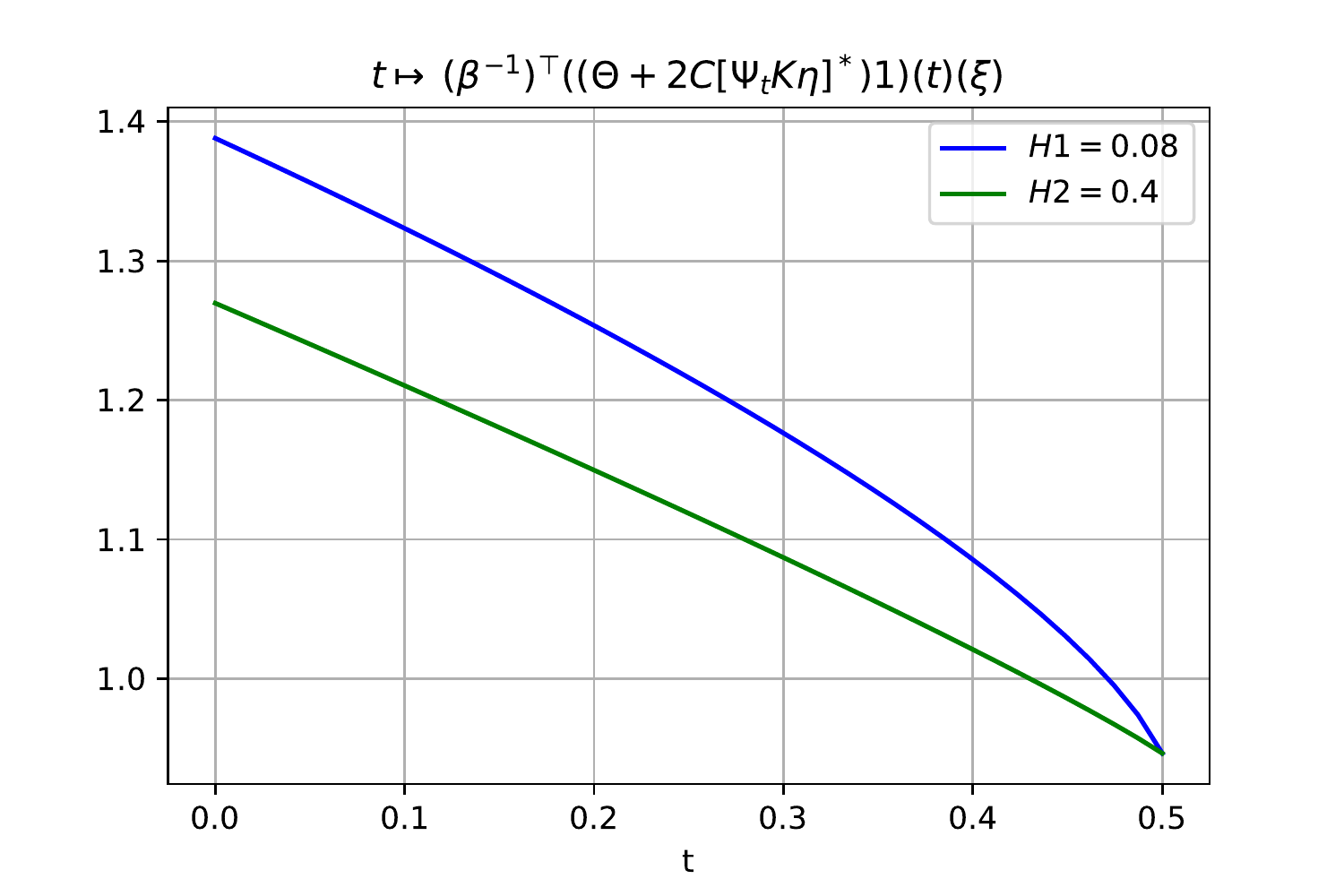}
    }
    
\subfloat[$T = 1.5$]{
\label{fig:horizon_T_2_rho_neg}
    \includegraphics[width=70mm]{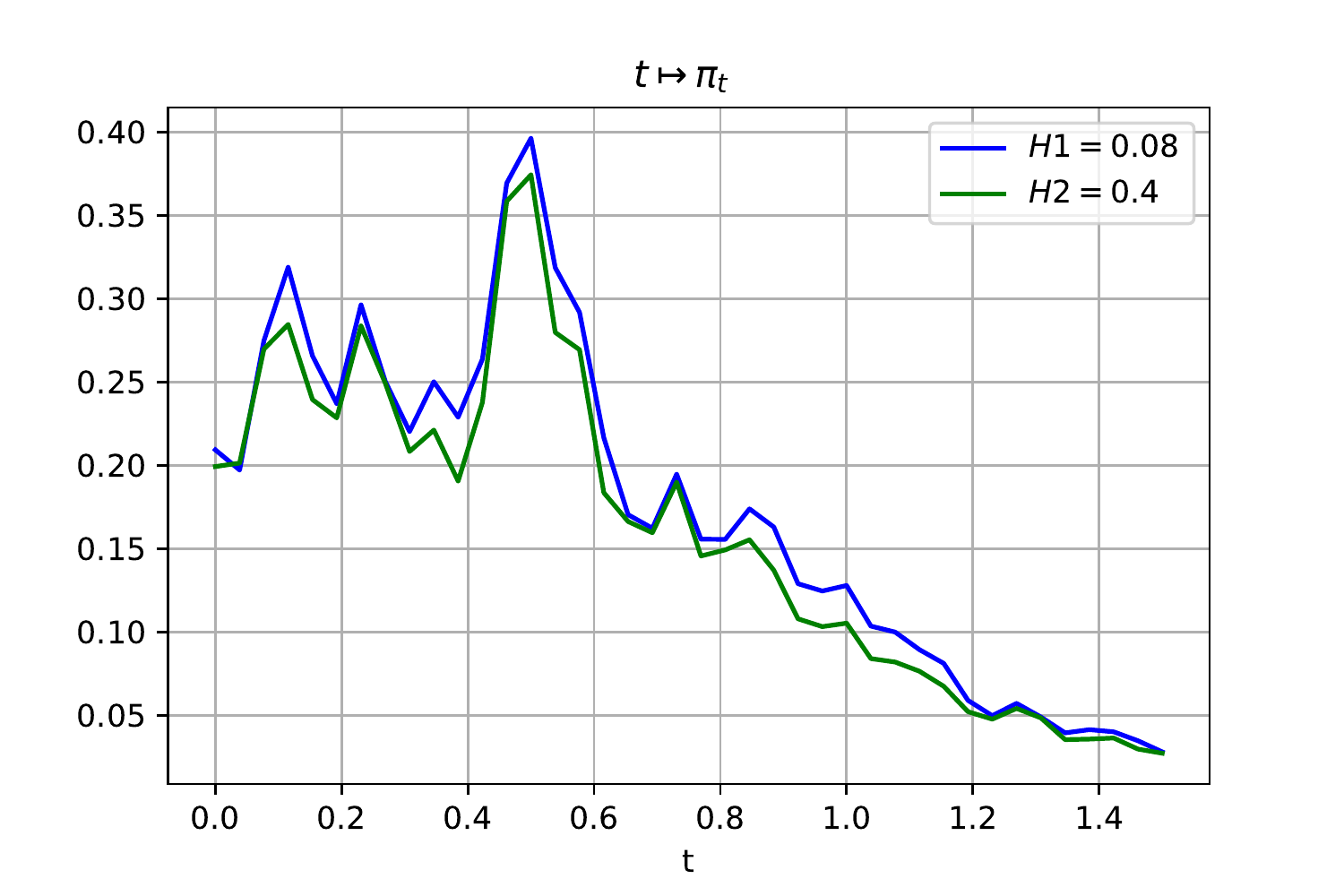}  
    \includegraphics[width=70mm]{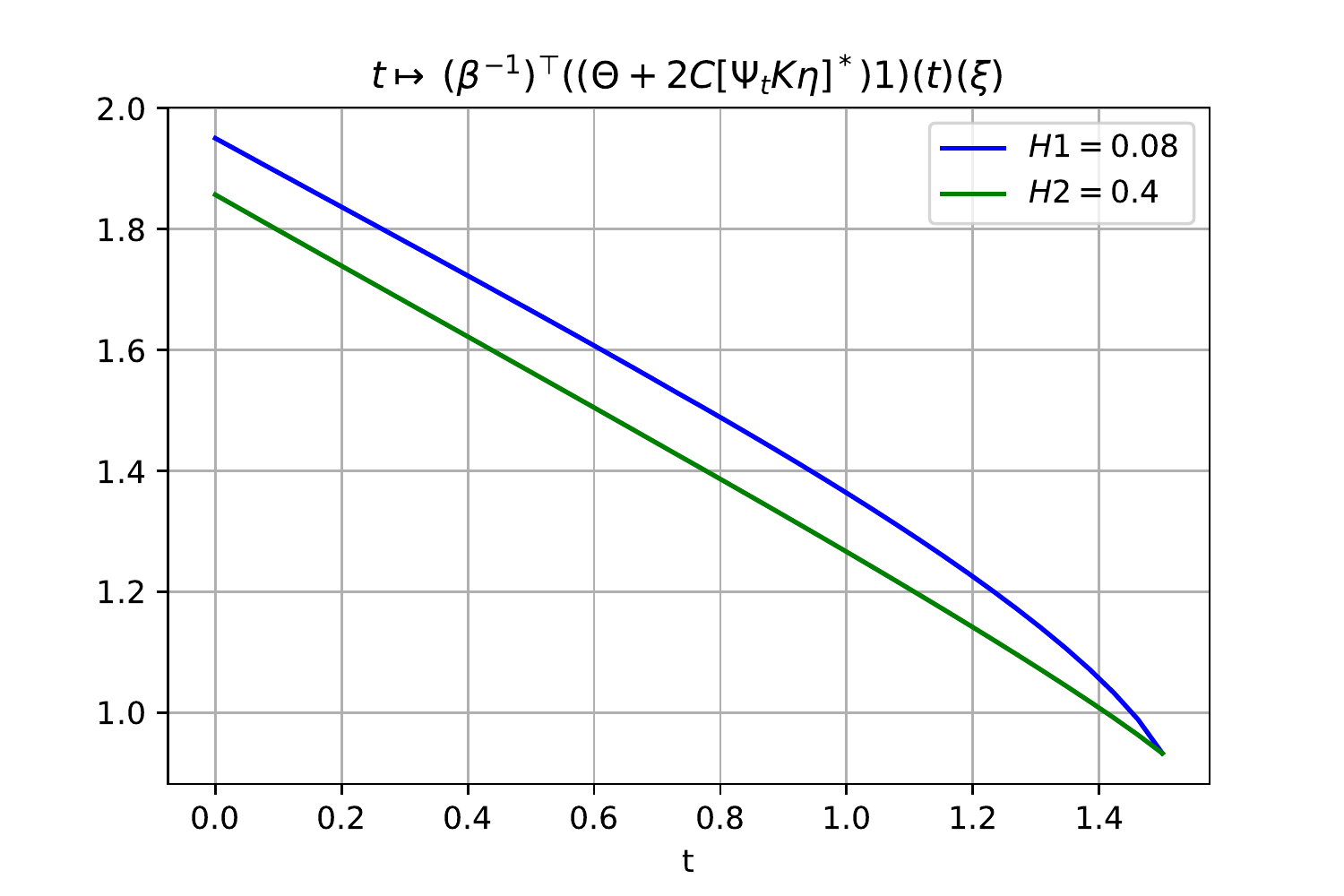}
    }

\subfloat[$T = 2.4$]{
\label{fig:horizon_T_3_rho_neg}
    \includegraphics[width=70mm]{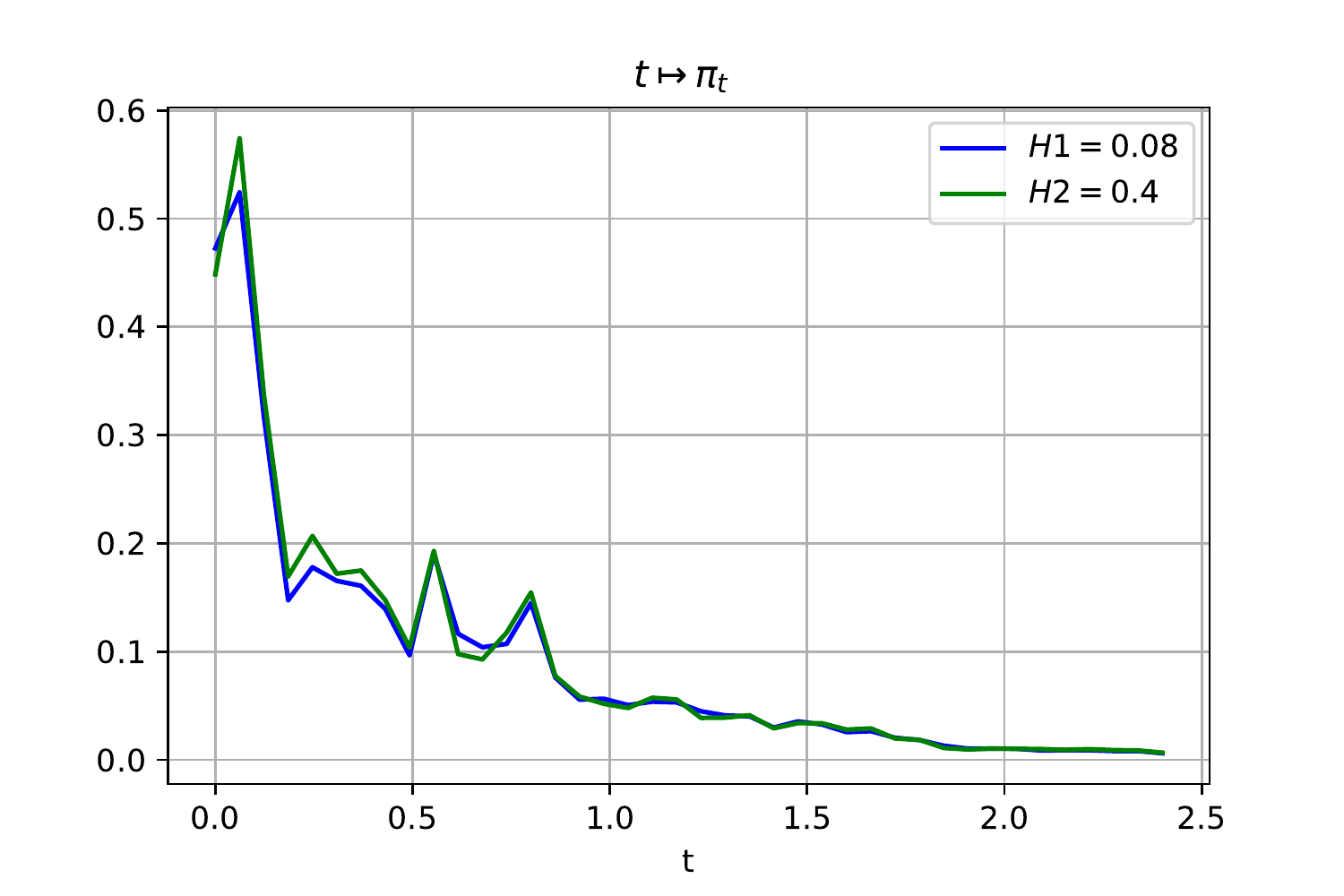}
    \includegraphics[width=70mm]{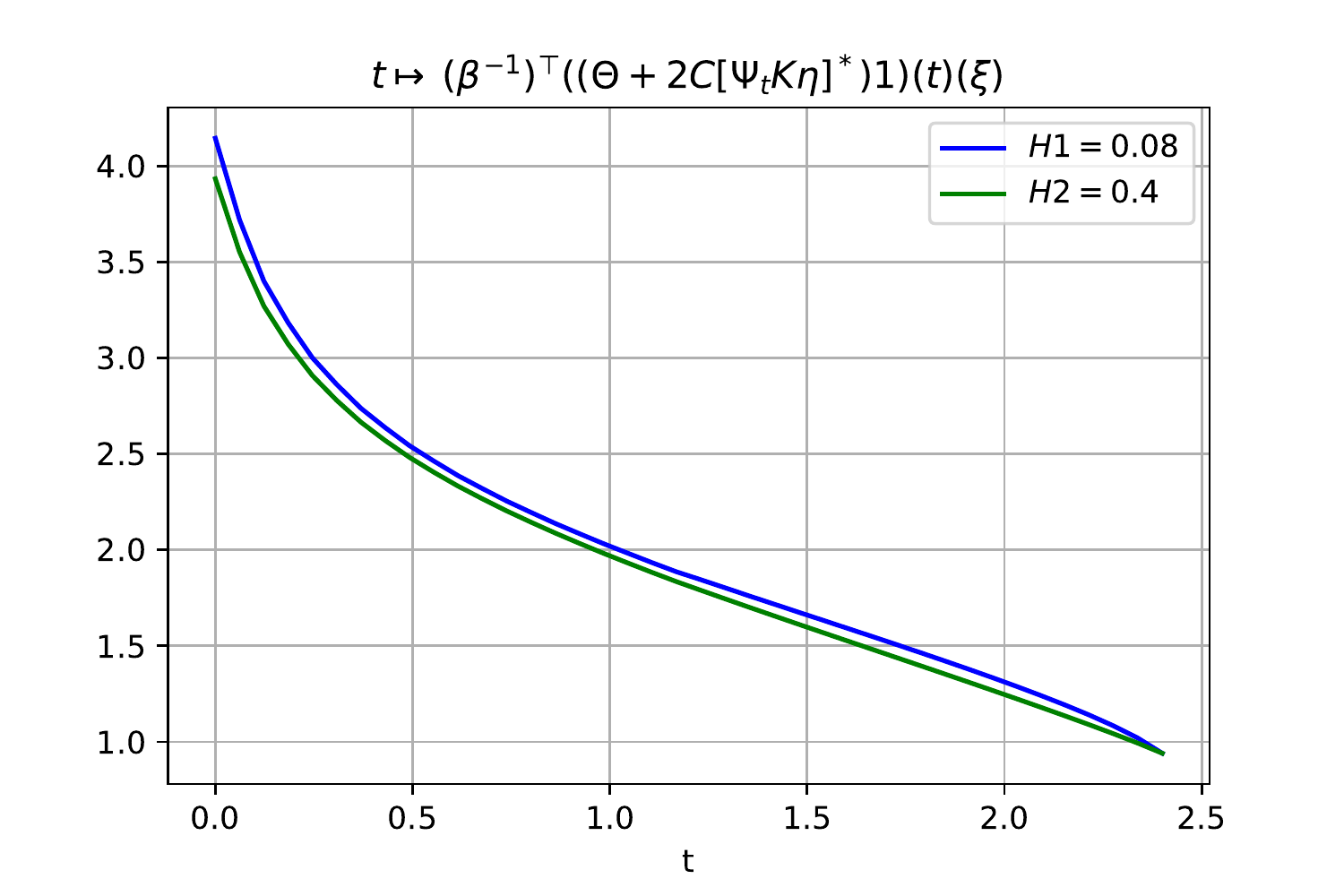}
    }
    
\caption{Effect of the horizon $T$ on the optimal allocation strategy when the two assets are negatively correlated ($\rho = -0.4$),  $H_1 = 0.08$, $H_2=0.4$. As $T$ increases the smooth stock in green is more and more weighted in comparison to the rough one in blue. But the transition takes more time compared to the case $\rho=0$, see Figures \ref{fig:horizon_T_1}-\ref{fig:horizon_T_3}.  $\eta_1= \eta_2 = 1, c_i = -0.7$. Note the beginning of the blow-up when $T$ reaches $T=2.4$, as it could be {foreseen} by the condition of Lemma \ref{L:lemmacondtheta2}.}
\end{figure}

\vspace{1mm}

~~\\
~~\\
~~\\

As a further line of research,  we see two interesting paths :
\begin{itemize}
    \item A theoretical study of influence of the {parameters} onto the investments strategies.
    \item An empirical study testing the different conjectures made about the influence of some parameters such as $T, \eta, \rho, H$, etc. 
\end{itemize}
{Our numerical results extend to larger horizon $T$. For instance, in  Figure \ref{fig:big_T}, we took a maturity of  $T=20$ years, although we noted that a smaller $\eta = 0.1$ had to be chosen to avoid any blow-up, in accordance with Remark \ref{rk:cond_theta}.}

\begin{figure}[h!]
\centering

    \includegraphics[width=70mm]{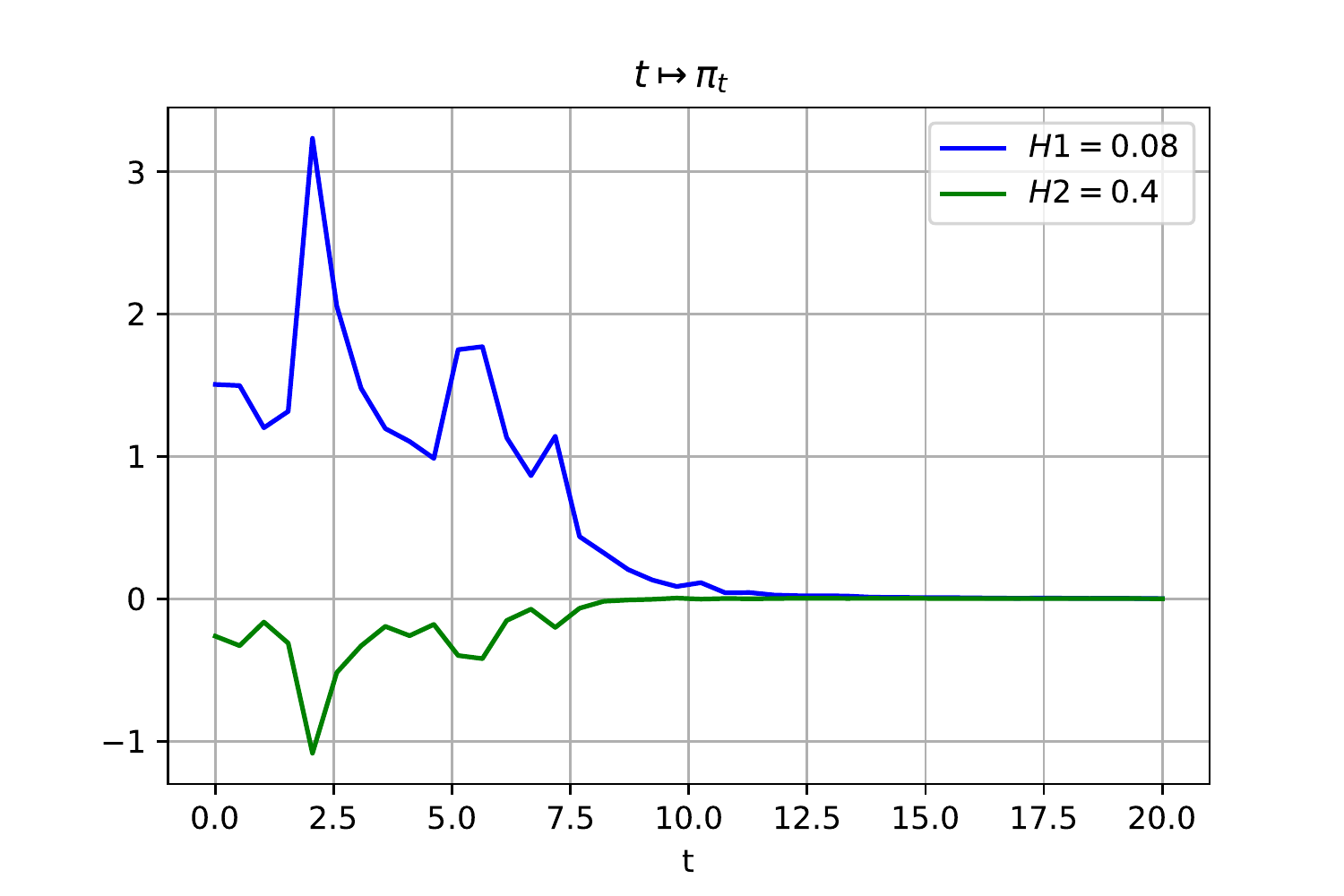}
    \includegraphics[width=70mm]{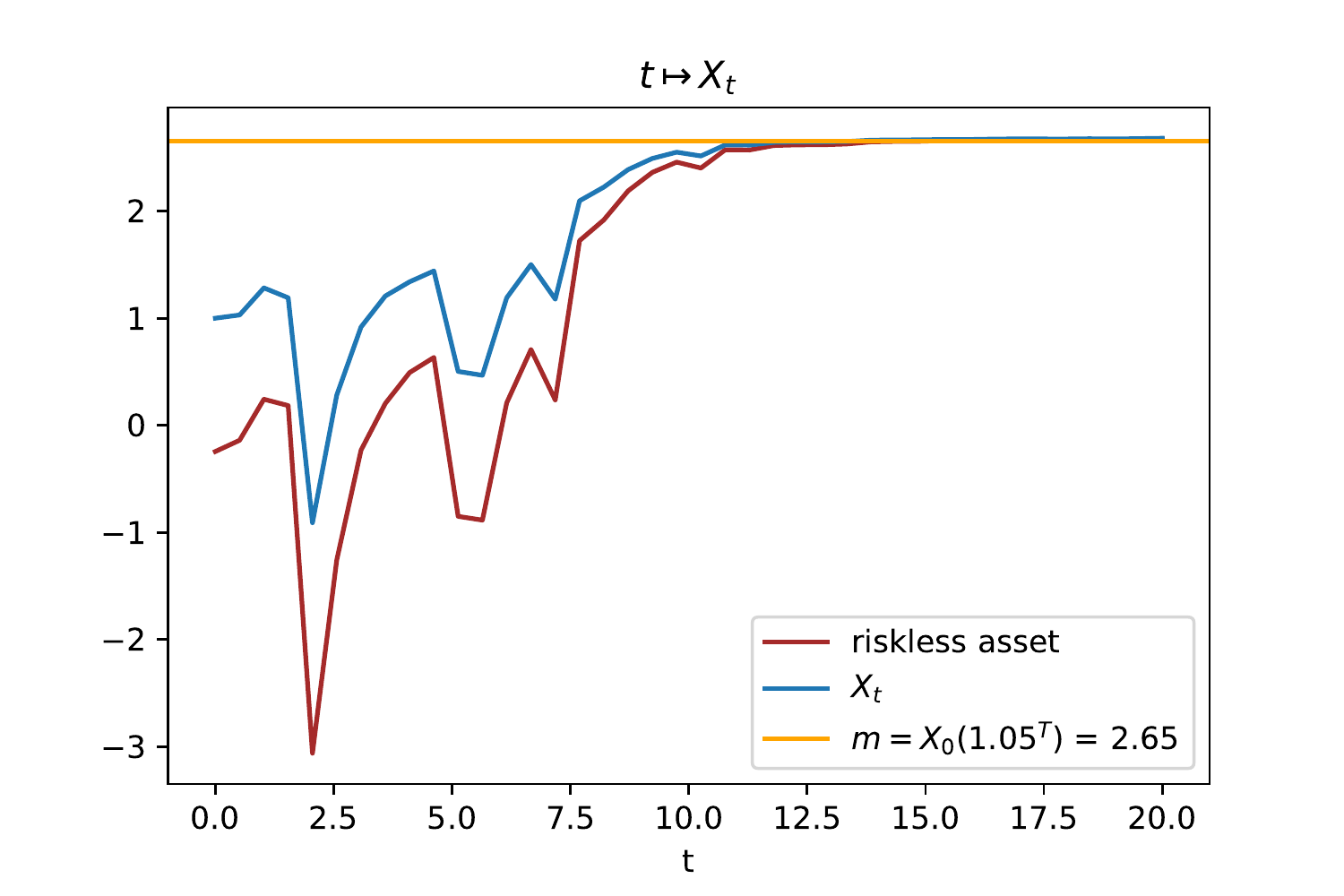}
\caption{Simulation for a larger horizon $T=20$ years. Note that a smaller $\eta$ had to be chosen in accordance with   Remark \ref{rk:cond_theta}.  (The parameters are: $H_1 = 0.08$, $H_2=0.4$, $\rho=0, \eta_1 = \eta_2 =0.1, c_i = -0.7$.) } 
\label{fig:big_T}
\end{figure}

\appendix
\section{Proof of the verification result}\label{A:verifresult}

In this section, we provide a detailed proof of  Theorem~\ref{T:verif}.  It is well-known that Markowitz problem \eqref{optimization_problem} is equivalent to the following max-min problem, see e.g.  \cite[Proposition 6.6.5]{pham2009continuous}:  
\bes{
    \label{outer_inner_optimization_pb}
    V(m) \; = & \; \max_{\eta \in \R} \min_{\substack{\c \in \Acal}} \Big\{  \E\Big[\big| X^{\c}_T - (m-\eta) \big|^2\Big] - \eta^2 \Big\}. 
}
Thus, solving problem \eqref{optimization_problem} involves two steps. First, the internal minimization problem in term of the Lagrange multiplier $\eta$ has to be solved. Second, the optimal value of $\eta$ for the external maximization problem has to be determined. Let us then introduce the inner optimization problem: 
\bes{
    \label{pb:P(c)}
    \tilde V(\xi) \;  := \;  \min_{\alpha \in \mathcal A}\E\Big[ \big| X^{\alpha}_T-\xi\big|^2\Big], \quad \xi \in \R. 
}

First, we provide a verification result for the inner optimization problem \eqref{pb:P(c)} via the standard completion of squares technique, see  for instance \citet[Proposition 3.1]{lim2002mean}, \citet[Proposition 3.3]{lim2004quadratic} and \citet[Theorem 3.1]{chiu2014mean}.

\begin{lemma}
\label{th:verification}
 Assume there exists a solution triplet $ (\Gamma, Z^1, Z^2) \in { \S^{\infty}_{\F}([0,T], \R)}$~~\\ $\times L^{2, loc}_{\F}([0,T], \R^d) \times L^{2, loc}_{\F}([0,T], \R^N) $ to the  Riccati BSDE \eqref{eq:riccati_sto}
such that $\Gamma_t>0$, for all $t\leq T$. 
Fix $\xi$ $\in$ $\R$, and assume that there exists an admissible control $\c^*(\xi)$  satisfying
\bes{
    \label{eq:optimal_control}
    \c^*_t(\xi) = -  \left(\lambda_t + Z^1_t + C Z^2_t\right) \left(X^ {\c^*(\xi)}_t  - \xi e^{-\int_t^T r(s) ds}\right), \quad 0 \leq t \leq T. 
}
Then, the  inner minimization problem \eqref{pb:P(c)} admits $\c^*(\xi)$ as an optimal feedback control and the optimal value is 
\bes{
    \label{eq:optimal_vamue_inner}
 \tilde V(\xi) \; = \;    \Gamma_0 \left|x_0 - \xi e^{-\int_0^T r(s) ds} \right|^2.
}
\end{lemma}

\begin{proof}
Let us first define $\Tilde{X}^{\alpha}_t = X_t^{\alpha} - \xi e^{-\int_t^T r(s) ds}$, for any $\c \in \mathcal{A}$. Then, by It\^o's lemma we have  
\bes{
    d \Tilde{X}^{\alpha}_t &=  \big( r(t) \Tilde{X}^{\alpha}_t  + \alpha_t^\T \lambda_t \big) dt + \alpha_t^\T dB_t, \quad 0 \leq t \leq T, \; 
    \Tilde{X}^{\alpha}_0  \; = \;   x_0 - \xi e^{-\int_0^T r(s) ds}. 
}
As a result, $\Tilde{X}^{\c}$ and $X^{\c}$ have the same dynamics and $\Tilde{X}^{\c}_T=X^{\c}_T-\xi$ so  that problem \eqref{pb:P(c)} can be alternatively written as
\bes{
\min_{\alpha \in \mathcal A}\E\Big[ \big|\Tilde{X}^{\alpha}_T\big|^2\Big].
}
To ease notations, we set $h_t = \lambda_t + Z^1_t + C Z^2_t $. For any $\c \in \mathcal A$, It\^o's lemma combined with \eqref{eq:riccati_sto} and a completion of squares in $\c$ yield
\bes{
    d \Big(\Gamma_t \big| \Tilde{X}_t^{\c} \big|^2\Big) \; =& \;  \big| \Tilde{X}_t^{\c} \big|^2 \Gamma_t \big(-2 r(t) + h_t^\T  h_t \big)dt  
    + \Gamma_t \big| \Tilde{X}_t^{\c} \big|^2 \Big(\big(Z^1_t\big)^\T dB_t + \big(Z^2_t\big)^\T dW_t \Big) \\
    &\; + \Gamma_t \Big(2 \Tilde{X}_t^{\c} \big(r(t) \Tilde{X}_t^{\c} +  \alpha_t^\T \lambda_t\big) + \c_t^\T \c_t  \Big) dt + 2\Gamma_t \Tilde{X}_t^{\c} \c_t^\T dB_t \\
    & \; + 2 \c^\T_t \left( Z^1_t + C Z^2_t\right) \Tilde{X}_t^{\c} dt \\
    =& \; \big( \c_t + h_t \Tilde{X}_t^{\c} \big)^\T \Gamma_t \big( \c_t + h_t\Tilde{X}_t^{\c} \big) dt \\
    & \; + 2\Gamma_t \Tilde{X}_t^{\c} \c_t^\T dB_t + \Gamma_t \big| \Tilde{X}_t^{\c} \big|^2 \Big(\big(Z^1_t\big)^\T dB_t + \big(Z^2_t\big)^\T dW_t \Big).
}
As a consequence,  {using $\Gamma_T=1$}, we get 
\bes{
\big| \Tilde{X}_T^{\c} \big|^2 =& \Gamma_0 \big| \Tilde{X}_0^{\c} \big|^2 + \int_0^T \big( \c_s + h_s \Tilde{X}_s^{\c} \big)^\T  \Gamma_s \big( \c_s + h_s \Tilde{X}_s^{\c} \big) ds \\
    &+ \int_0^T 2\Gamma_s \Tilde{X}_s^{\c} \c_s^\T dB_s +  \int_0^T 2\Gamma_s \big| \Tilde{X}_s^{\c} \big|^2 \Big(\big(Z^1_s\big)^\T dB_s + \big(Z^2_s\big)^\T dW_s \Big).
}
Note that the stochastic integrals
\begin{align}
    \int_0^. 2\Gamma_s \Tilde{X}_s^{\c} \c_s^\T dB_s, &&  \int_0^. \Gamma_s \big| \Tilde{X}_s^{\c} \big|^2 \left(Z^1_s\right)^\T dB_s, &&  \int_0^. \Gamma_s \left( \Tilde{X}_s^{\c} \right)^2  \left(Z^2_s\right)^\T dW_s,
\end{align}
are well-defined since $X^{\alpha}$ is continuous, $(\alpha,Z^1,Z^2)$ are in $L^{2,{loc}}_{\mathbb F}([0,T])$ and $\Gamma$ in $\S^{\infty}_{\F}([0,T], \R)$. Furthermore, they are    local martingales. Let $\{\tau_k\}_{k\geq 1}$ be a common localizing increasing sequence of stopping times  converging to $T$. Then, 
\bes{
\E \Big[ \big| \Tilde{X}_{T \wedge \tau_k}^{\c} \big|^2 \Big] \;=& \;  \Gamma_0 \big| \Tilde{X}_0^{\c} \big|^2 
+ \E \Big[ \int_0^{T \wedge \tau_k} \big( \c_s + h_s \Tilde{X}_s^{\c} \big)^\T  \Gamma_s \big( \c_s + h_s \Tilde{X}_s^{\c} \big) ds \Big]. 
}
Since $\alpha \in \mathcal A$, $X^{\c}$ satisfies \eqref{eq:estimateX}, and so $\E\left[\sup_{t\leq T}  |\Tilde X^{\c}_t|^2\right]<\infty$. An application of the dominated convergence theorem on the left term combined with the monotone convergence theorem on the right term, recall that $\Gamma$ is $\S^d_+$-valued, yields, as $k \to \infty$,
\bes{
\E \Big[ \big| \Tilde{X}_{T}^{\c} \big|^2 \Big] \; =& \; \Gamma_0 \big| \Tilde{X}_0^{\c} \big|^2 
+ \E \Big[ \int_0^{T} \big( \c_s +  h_s \Tilde{X}_s^{\c} \big)^\T  \Gamma_s  \big( \c_s + h_s \Tilde{X}_s^{\c} \big) ds \Big].
}
Since $\Gamma_s$ is positive definite for any $s \leq T$,  we obtain that the optimal strategy  $\c^*(\xi)$ is given by \eqref{eq:optimal_control} and the optimal value of \eqref{pb:P(c)} is equal to \begin{align*}
\tilde V(\xi) \; = \; \Gamma_0 \big| \Tilde{X}_0^{\c^*(\xi)} \big|^2 \;= \;  \Gamma_0 \big| X_0 - \xi e^{-\int_0^T r(s) ds} \big|^2,
\end{align*}
which gives \eqref{eq:optimal_vamue_inner}. 
\end{proof}

\vspace{1mm}

We next address the admissibility of the candidate for the optimal control. 

 \begin{lemma}\label{L:outer} Assume that  there exists a solution triplet $ (\Gamma, Z^1, Z^2) \in \S^{\infty}_{\F}([0,T], \R)  \times L^{2, loc}_{\F}([0,T], \R^d) \times L^{2, loc}_{\F}([0,T], \R^N)$ to the  Riccati BSDE \eqref{eq:riccati_sto} such that  \eqref{eq:assumption_novikov}  holds
for some  $p > 2$ and a constant $a(p)$ given by \eqref{eq:constap}.
    Then, for any $\xi \in \R$, there exists an admissible control process $\c^*(\xi)$ satisfying  \eqref{eq:optimal_control}.
\end{lemma}

\begin{proof}
Fix  $\xi \in \R$. We first prove that there exists a control   $\c^*(\xi)$    satisfying  \eqref{eq:optimal_control}. For this, we prove that  the corresponding wealth equation \eqref{eq:wealth}
admits a solution.  As in the proof of Lemma~\ref{th:verification}, it is enough to consider the modified equation 
\begin{align*}
d\Tilde X_t^* &= \big(r(t) \Tilde X_t^*  + \lambda_t^\top A_t \Tilde X_t^*  \big)dt + \big(A_t   \Tilde X_t^* \big)^\top dB_t, \quad
\Tilde X_0^* \; = \; x_0 -\xi e^{-\int_0^T r(s) ds},
\end{align*}
where $A_t = - \left(\lambda_t + Z^1_t + C Z^2_t\right)$, and then set  $X_t^{*}= \Tilde{X}^{*}_t + \xi e^{-\int_t^T r(s) ds}$. By virtue of  It\^o's lemma the unique continuous solution is given by    
\bes{
    \Tilde{X}^{*}_t =\Tilde{X}^{*}_0 \exp\Big(\int_0^t \big(r(s) + \lambda_s^\T A_s - \frac{A_s^\T A_s}{2} \big)ds + \int_0^t A_s^\T dB_s\Big).
}
Setting $\alpha^*_t(\xi)$ $:=$ $A_t\Tilde X^*_t$, we obtain that $\alpha^*(\xi)$ satisfies \eqref{eq:optimal_control} with the controlled wealth $X^{\alpha(\xi)^*}=X^*$. The crucial step is now to  
obtain the admissibility condition \eqref{eq:estimateX}. For that purpose, observe by virtue of  \eqref{eq:assumption_novikov}, that the Dol\'eans-Dade exponential $\mathcal E\left( \int_0^{\cdot} A_s^\top dB_s\right)$ satisfies Novikov's condition, and is therefore a true martingale. 
Whence, successive applications of the inequality $ab\leq (a^2+b^2)/2 $ and  Doob's maximal inequality yield, for some constant $K>0$ which may vary from line to line,
\bes{
\E \Big[ \sup_{t \in [0,T]} |\tilde{X}^*_t|^p \Big] &\leq  K  \E \Big[ \sup_{t \in [0,T]}  \big|  e^{\int_0^t {\left(r(s) + \lambda_s^\T A_s \right)}ds} \big|^{2p}  \Big] 
+ K \E \Big[ \sup_{t \in [0,T]} \Big| e^{-\int_0^t \frac{A_s^\T  A_s}{2}  ds + \int_0^t A_s^\T dB_s} \Big|^{2p}  \Big] \\
   &\leq K \E \Big[   e^{\int_0^T 2p \big| \lambda_s^\T A_s \big| ds}  \Big] + K \E \Big[   e^{- p \int_0^T A_s^\T  A_s  ds + 2p\int_0^T A_s^\T dB_s}  \Big] \\
    &= K \left( \textbf{1} + \textbf{2} \right),
 }
which is finite since
\bes{
    \textbf{1} &\leq  \E \left[ \exp\left({a(p) \int_0^T \left( |\lambda_s|^2 + |Z^1_s|^2 + |Z^2_s|^2\right)  ds}\right)    \right]< \infty,
}
and, {by virtue of the Cauchy-Schwarz inequality},
\bes{
    \textbf{2} &\leq \left(   \E \left[ e^{(8p^2 - 2p) \int_0^T  A_s^\T  A_s   ds} \right]    \right)^{1/2} \left(\E \left[  e^{-8p^2\int_0^T A_s^\T  A_s ds + 4p \int_0^T A^\T_s dB_s}    \right] \right)^{1/2} \\
  &  \leq  \left(   \E \left[ e^{a(p) \int_0^T \left(|\lambda_s|^2 + |Z^1_s|^2 + |Z^2_s|^2 \right)  ds} \right]    \right)^{1/2} \times 1  < \infty,
}
{where we used Jensen's inequality to bound 
\bes{
 A_s^\top A_s = |\lambda_s + Z_s^1 + CZ_s^2|^2 \leq 3(|\lambda_s|^2 + |Z^1_s|^2 + |CZ^2_s|^2) \leq 3 (1 + |C|^2)(|\lambda_s|^2 + |Z^1_s|^2 + |Z^2_s|^2),}}
together with assumption (H2) and Novikov's condition to the Dol\'eans-Dade exponential $\mathcal{E}(4p \int_0^{{\cdot}} A_s^\T dB_s)$. 
Finally, to get that $\c^*(\xi)$ is admissible,  we are left to prove that $\alpha^*(\xi) \in L^2_{\F}([0,T], \R^d)$.   Let $2/p + 1/\hat{q}=1$, by H\"older's inequality we obtain 
\bes{
    \E \left[ \int_0^{T} |\c_s^*(\xi)|^2 ds \right] & { =} \E \left[ \int_0^{T} | A_s \tilde{X}_s^*|^2 ds \right] \\
    &\leq  \E \left[ \sup_{t \in [0,T]} |\tilde{X}_t^*|^2 \int_0^T |A_s|^2 ds \right] \\
   & \leq \left(  \E \left[ \sup_{t \in [0,T]} |\tilde{X}_t^*|^{p}  \right] \right)^{2/p} \left( \E\left[ \left( \int_0^T |A_s|^2 ds \right)^{\hat{q}}  \right]  \right)^{1/\hat{q}} \\
  &   \leq  C  \left(  \E \left[ \sup_{t \in [0,T]} |\tilde{X}_t^*|^{p}  \right] \right)^{2/p}   \left( \E\left[ \left( \int_0^T \left(|\lambda_s|^2 +|Z^1_s|^2+ |Z^2_s|^2 \right) ds \right)^{\hat{q}}  \right]  \right)^{1/\hat{q}}\\
    &  < \infty,
}
where the last term is finite due to condition  \eqref{eq:assumption_novikov} and the inequality $|z|^q\leq c_{q}e^{|z|}$. The proof is complete.
\end{proof}

\vspace{1mm}

Finally, combining the above,   we deduce the solution for the outer optimization problem \eqref{optimization_problem} under a non-degeneracy condition on the solution  $\Gamma$ to the Riccati BSDE,  yielding Theorem~\ref{T:verif}.

\begin{proof}[Proof of Theorem~\ref{T:verif}]
From Lemmas \ref{th:verification} and \ref{L:outer}, we have that the max-min problem \eqref{outer_inner_optimization_pb} (which is equivalent to the Markowitz problem \eqref{optimization_problem}) is equivalent to 
    \bes{
    \label{outer_inner_optimization_pb_}
    &\max_{\eta \in \R}  J(\eta), \quad \mbox{ with } \; J(\eta) \; = \;  \Gamma_0 \big|X_0 - (m-\eta) e^{-\int_0^T r(s) ds} \big|^2  - \eta^2. \\
}
Furthermore, condition (H1): $\Gamma_0 <  e^{2 \int_0^T r(s) ds}$, ensures that the quadratic function $J$  is strictly concave. This yields that  the maximum is achieved 
from the first-order condition $J'(\eta^*)$ $=$ $0$, which gives 
\begin{align*}
\eta^* &= \;  \frac{\Gamma_0 e^{-\int_0^T r(s) ds} \Big( x_0 - m e^{-\int_0^T r(s) ds } \Big)}{1 - \Gamma_0 e^{-2\int_0^T r(s) ds}},
\end{align*}
and thus $\xi^*$ $=$ $m-\eta^*$ is given by  \eqref{eq:eta_star}.  We conclude that the optimal control is equal to  $\alpha^*=\alpha^*(\xi^*)$ as in \eqref{eq:optimal_control_final}, and by 
\eqref{outer_inner_optimization_pb}, the optimal value of \eqref{optimization_problem}  is equal to $V(m)$ $=$ $\tilde V(\xi^*) - (\eta^*)^2$, given  by \eqref{eq:value_final}. 
\end{proof}

\section{Proofs of some technical lemmas}\label{A:proofs+}

\subsection{Reminder on resolvents of integral operators}\label{A:resolvents}

\begin{lemma}\label{L:staroperation}
	Let $K$ satisfy \eqref{assumption:K_stein} and $L \in L^2([0,T]^2,\R^{N\times N})$. Then, $ K\star L$ satisfies  \eqref{assumption:K_stein}. Furthemore,  if $L$ satisfies  \eqref{assumption:K_stein}, then, $(s,u)\mapsto (K \star L^*)(s,u)$ is continuous.
\end{lemma}
\begin{proof}
	An application of the Cauchy-Schwarz inequality yields the first part. The second part follows along the same lines as in the proof of \citet[Lemma 3.2]{jaber2019laplace}.
\end{proof}

\vspace{1mm}

For a kernel $K \in L^2([0,T]^2,\R^{N\times N})$, we define its resolvent $R_T \in L^2([0,T]^2,\R^{N\times N})$ by the unique solution to 
\begin{align}\label{eq:resolventeqkernel}
R_T = K + K \star R_T, \quad  \quad  K \star R_T =  R_T \star K.
\end{align} 
In terms of integral operators, this translates into 
\begin{align*}
\boldsymbol{R}_T =  \boldsymbol{K} + \boldsymbol{K}\boldsymbol{R}_T, \quad \boldsymbol{K}\boldsymbol{R}_T=\boldsymbol{R}_T\boldsymbol{K}.
\end{align*}
In particular, if $K$ admits a resolvent,  $({\id}-\boldsymbol{K})$ is invertible and
\begin{align}\label{eq:integralreso}
({\id}-\boldsymbol{K})^{-1}=\id+\boldsymbol{R}_T,
\end{align}
where $\id$ denotes the identity operator, i.e. $(\id f)=f$ for all $f \in L^2\left([0,T],\R^N \right)$.

\vspace{1mm}

The following lemma establishes the existence of resolvents for the two classes of kernels introduced above. 

\begin{lemma}\label{L:resolventvolterra}
	Let $K \in L^2\left([0,T]^2,\R^{N\times N}\right)$. $K$ admits a resolvent if either one of the following conditions hold:
	\begin{enumerate}
		\item \label{L:resolventvolterrai} 
		$K$ is a Volterra kernel of continuous and bounded type   in $L^2$ in the sense of Definition~\ref{D:kernelvolterra}. In this case, the resolvent is again a Volterra kernel of continuous and bounded type. 
		\item\label{L:resolventvolterraii}
		$K$ is symmetric  {nonpositive} in the sense of Definition~\ref{D:nonnegative} and $(s,u)\mapsto K(s,u)$ is continuous. 
	\end{enumerate}
\end{lemma}

\begin{proof}
	\ref{L:resolventvolterrai}  follows from  \citet[Lemma 9.3.3, Theorem 9.5.5(i)]{GLS:90}.  \ref{L:resolventvolterraii} follows from an application of Mercer's theorem, see \citet[Section 2.1]{jaber2019laplace}. 
\end{proof}


\subsection{Proof of Lemma~\ref{L:Psi}}\label{AppendixL:Psi}
Fix $t\leq T$. We start by proving that $\boldsymbol{\Psi_t}$ is well defined and is a bounded linear operator from $L^2\left([0,T],\R^N\right)$ to $L^2\left([0,T],\R^N\right)$.  First, since $K$ is a Volterra kernel of continuous and bounded type in $L^2$, so is $\hat K$, and Lemma \ref{L:resolventvolterra}-\ref{L:resolventvolterrai}  yields the existence of its resolvent $\hat{R}$ such that 
\begin{align}
\label{hyp:R_hat}
\sup_{s\leq T} \int_{0}^T|\hat{R}(s,u)| ds <\infty, &&\sup_{u\leq T} \int_{0}^T|\hat{R}(s,u)| du <\infty.
\end{align}
In particular, denoting  by $\boldsymbol{\hat{R}}$ the integral operator induced by  $\hat R$, we obtain that $(\id-\hat {\boldsymbol{K}})$ is invertible with an inverse given by  $(\id-\boldsymbol{\hat{K}})^{-1}=\id+\boldsymbol{\hat{R}}$, recall \eqref{eq:integralreso}. Next, we prove that $\big(\id + 2\Theta \tilde{\boldsymbol{\Sigma}}_t \Theta^\top\big)$ is invertible. It follows from \eqref{def:C_tilde} that
\begin{align}
\boldsymbol{\tilde{\Sigma}}_t = (\id+\boldsymbol{\hat{R}}) \boldsymbol{\Sigma}_t (\id+\boldsymbol{\hat{R}})^{*}= \boldsymbol{\Sigma}_t + \boldsymbol{\Sigma}_t \boldsymbol{\hat{R}}^* + \boldsymbol{\hat{R}} \boldsymbol{\Sigma}_t + \boldsymbol{\hat{R}}\boldsymbol{\Sigma}_t\boldsymbol{\hat{R}}^*.
\end{align}
Whence, $\boldsymbol{\tilde{\Sigma}}_t$ is an integral operator generated by the kernel 
\begin{align}\label{eq:tildekernel}
{\tilde{\Sigma}}_t= {\Sigma}_t + {\Sigma}_t \star\hat{{{R}}}^* + \hat{{{R}}} \star{\Sigma}_t + \hat{{{R}}}\star{\Sigma}_t\star\hat{{{R}}}^*.
\end{align}
Since $K$ satisfies \eqref{assumption:K_stein} and $(U-2C^\top C) \in \S^N_+$, $\Sigma_t$ defined in \eqref{eq:sigmakernel} is clearly a symmetric nonnegative kernel. Combined with \eqref{eq:tildekernel}, we get that ${\tilde{\Sigma}}_t$ is symmetric nonnegative. {Successive} applications of Lemma~\ref{L:staroperation} yield that $(s,u)\mapsto {\tilde{\Sigma}}_t(s,u) $ is continuous. Therefore,   $(-2\Theta\boldsymbol{\tilde{\Sigma}}_t\Theta^\top)$ is symmetric nonpositive and continuous so that  an application of Lemma~\ref{L:resolventvolterra}-\ref{L:resolventvolterraii} yields the existence of its resolvent $R^{\Theta}_t$. In particular, $\big(\id + 2\Theta \tilde{\boldsymbol{\Sigma}}_t \Theta^\top\big)$ is invertible with an inverse given by $(\id +\boldsymbol{R}^{\Theta}_t)$, recall \eqref{eq:integralreso}. Combining the above, we get that $\boldsymbol{\Psi}_t$ is well-defined, and satisfies 
\bes{
	\label{eq:boundary}
	\boldsymbol{\Psi}_t &= -(\id + \boldsymbol{\hat{R}})^* \Theta^\top (\id+\boldsymbol{R}^{\Theta}_t)\Theta(\id+\boldsymbol{\hat{R}})\\
	&= - \Theta^\top \Theta \id  - \boldsymbol{\hat{R}}^* \Theta^\top \Theta  - \Theta^\top \Theta  \boldsymbol{\hat{R}}- \boldsymbol{\hat{R}}^* \Theta^\top \boldsymbol{R}^{\Theta }_t \Theta    - \Theta^\top \boldsymbol{R}^{\Theta }_t  \Theta  \boldsymbol{\hat{R}}\\
	& \quad - \boldsymbol{\hat{R}}^* \Theta^\top  \boldsymbol{R}^{\Theta }_t  \Theta  \boldsymbol{\hat{R}}- \boldsymbol{\hat{R}}^* \Theta^\top \Theta  \boldsymbol{\hat{R}}- \Theta^\top  \boldsymbol{R}^{\Theta }_t \Theta,
}
showing that $\boldsymbol{\Psi}_t$ is a bounded operator.\\

\noindent \ref{L:Psi1}: From \eqref{eq:boundary}, we see that  $(\Theta^\top \Theta \id + \boldsymbol{\Psi}_t)$  is an integral operator whose kernel is of the form
\bes{
	\psi_t&= -{\hat{R}}^* \Theta^\top \Theta  -\Theta^\top \Theta {\hat{R}}- {\hat{R}}^* \Theta^\top \star R^{\Theta }_t  \Theta    - \Theta^\top   R^{\Theta }_t \star  \Theta  {\hat{R}} \\
	&\quad - {\hat{R}}^*  \Theta^\top \star R^{ \Theta }_t  \star \Theta  \hat{R} - {\hat{R}}^*\Theta^\top \star \Theta  {\hat{R}} - \Theta^\top R^{\Theta }_t  \Theta .
}
Then, from  \citet[Lemma C.1]{jaber2019laplace} we get that 
\bes{
	\sup_{t\leq T} \int_{[0,T]^2}|R^{\Theta}_t(s,u)|^2 ds du<\infty,
}
which, combined with \eqref{hyp:R_hat} ensures \eqref{eq:bound_psi_leb}.\\

\noindent \ref{L:Psi2}: Fix $f \in L^2\left([0,T],\R^N\right)$ and $t\leq T$. We first argue that 
\begin{align}\label{eq:bordR}
R^{\Theta}_t(t,.) = 0 \mbox{ and } \hat{R}(s,u)=0, \quad  \mbox{for any }  s<u .
\end{align}
Indeed, since $\hat K$ is a Volterra kernel, its resolvent $\hat R$ is also a Volterra kernel so that $R(s,u)=0$ whenever $s<u$. This, combined with the fact that $\Sigma_t(t,\cdot)=0$ and   \eqref{eq:tildekernel}, yields that $\tilde \Sigma_t(t,\cdot)=0$, so that $R^{\Theta}_t(t,\cdot)=0$  by virtue of the resolvent equation \eqref{eq:resolventeqkernel}. 
Using the relations \eqref{eq:bordR}, we compute
\bes{
	\label{eq:zero_terms}
	\left(\Theta ^\top  \Theta  \boldsymbol{\hat{R}}\right)(f1_t)(t) & = \;  \Theta ^\top  \Theta  \int_0^T \hat{R}(t,s) f(s)1_t(s)ds \; = \;  0, \\
	\left( \Theta^\top \boldsymbol{R}_t^{\Theta} \Theta \right) (f1_t)(t) & = \;  \Theta^\top \int_0^T R^{\Theta}_t(t,s) \Theta f(s)1_t(s)ds \; = \; 0, \\
	\left(\Theta^\top \boldsymbol{R}^{\Theta}_t \Theta \boldsymbol{\hat{R}}\right)(f1_t)(t) & = \;  \Theta^\top \int_0^T \int_0^T R^{\Theta}_t(t,u)\Theta\hat{R}(u,s) f(s)1_t(s) du ds \; = \; 0.
} 
Thus, \eqref{eq:zero_terms} combined with \eqref{eq:boundary} and the resolvent's relations $\boldsymbol{\hat{R}}= \boldsymbol{\hat{K}} + \boldsymbol{\hat{K}}\boldsymbol{\hat{R}}$ and $\boldsymbol{\hat{R}}^* = \boldsymbol{\hat{K}}^* + \boldsymbol{\hat{K}}^*\boldsymbol{\hat{R}}^*$ yield
\bes{
	-({\Theta}^\top {\Theta}\id+\boldsymbol{\Psi}_t) (f 1_t)(t) =& (\boldsymbol{\hat{R}}^* {\Theta}^\top {\Theta}+ \boldsymbol{\hat{R}}^*{\Theta}^\top \boldsymbol{R}^{\Theta} {\Theta}+ \boldsymbol{\hat{R}}^*{\Theta}^\top  \boldsymbol{R}_t^{\Theta} {\Theta} \boldsymbol{\hat{R}}+ \boldsymbol{\hat{R}}^*{\Theta^\top}{\Theta} \boldsymbol{\hat{R}} )(f 1_t)(t) \\
	=& -(\boldsymbol{ \hat{K}}^* \boldsymbol{\Psi}_t)(f1_t)(t) 
}
which proves the second claim \ref{L:Psi2}.\\

\noindent \ref{L:Psi3}: Under \eqref{eq:assumptionkerneldiff1},   \citet[Lemma 3.2]{jaber2019laplace} yields that $t\mapsto \boldsymbol{\Sigma}_t$ is strongly differentiable on $[0,T]$ with a derivative given by $t\mapsto\dot{\boldsymbol{\Sigma}}_t$ induced by the kernel \eqref{eq:diffkernelsigma}. Whence, it follows from  \eqref{def:C_tilde}, that $t\mapsto \boldsymbol{\tilde{\Sigma}}_t$ is also differentiable such that  $\dot{\boldsymbol{\tilde{\Sigma}}}_t=(\id -\boldsymbol{\hat K})^{-1} \dot{\boldsymbol{{\Sigma}}}_t(\id -\boldsymbol{\hat K})^{-*}$. Thus, \eqref{def:riccati_operator} yields that  $t\mapsto \boldsymbol{\Psi}_{t}$ is strongly differentiable with a derivative given by  
\bes{
	\dot{\boldsymbol{\Psi}}_t =& \;  2(\id - \boldsymbol{\hat{K}})^{-*} \Theta^\top  (\id + 2\Theta \tilde{\boldsymbol{\Sigma}}_t\Theta^\top)^{-1}\Theta \dot{\boldsymbol{\tilde{\Sigma}}}_t \Theta^\top (\id + 2\Theta \tilde{\boldsymbol{\Sigma}}_t\Theta^\top)^{-1}   \Theta(\id - \boldsymbol{\hat{K}})^{-1} \\
	=& \; 2\boldsymbol{\Psi}_t \dot{\boldsymbol{\Sigma}}_t  \boldsymbol{\Psi}_t.
}
Finally, evaluating \eqref{eq:sigmakernel} at $t=T$, yields that $\Sigma_T(s,u)=0$ for all $s,u\leq T$, leading to $\boldsymbol{\Sigma}_T=\boldsymbol{0}$ so that $\boldsymbol{\Psi}_T=-\left(\id - \boldsymbol{\hat{K}}\right)^{-*} { \Theta^\top \Theta}\left(\id - \boldsymbol{\hat{K}}\right)^{-1}$. This proves \eqref{eq:riccati_psiBold}.


\subsection{Proof of Lemma~\ref{L:lemmacondtheta1}} \label{seclemma1}
We start with a  lemma to bound the kernel  $\tilde \Sigma$.
\begin{lemma}
\label{l:bound_sig_bar}
Let $f(\Theta) = D-2\eta C^\T \Theta$ and assume that $|f(\Theta)| \times \|K\|_{L^2([0,T]^2)}^2 < 1$. Then there exists a constant $c>0$ such that 
    \bes{
    \label{eq:bound_sigma_tilde}
        \sup_{t\leq T} \| \tilde{\Sigma}_t \|^2_{L^2([0,T]^2)} \leq {c (1 +  \hat \kappa(\Theta))},
    }
    where $ \hat \kappa$ is defined as 
    \bes{
        \label{eq:def_k_hat}
         \hat \kappa(\Theta) =  \left(\frac{|f(\Theta)| \times \|K\|_{L^2([0,T]^2)}^2}{ 1 -|f(\Theta)| \times \|K\|_{L^2([0,T]^2)}^2} \right)^4.
    }
\end{lemma}
\begin{proof}
 Let $\hat{R}$ denote the resolvent kernel of $\hat{K}=K f(\Theta)$ as in the proof of Lemma \ref{L:Psi}. First note that the relation $(\id - \hat{\boldsymbol{K}})^{-1} = \id + \hat {\boldsymbol{R}}$ yields  
    \bes{
        \| \tilde{\Sigma}_t \|^2_{L^2([0,T]^2)} =& \|(\id - \hat{K})^{-1} \star\Sigma_t \star (\id - \hat{K})^{-*} \|^2_{L^2([0,T]^2)}\\
        =& \| \Sigma_t + \hat{R} \star \Sigma_t + \Sigma_t \star  \hat{R} + \hat{R} \star \Sigma_t \star \hat{R} \|^2_{L^2([0,T]^2)} \\
        \leq & 2^{3} \Big(  \|\Sigma_t \|^2_{L^2([0,T]^2)} + \|\hat{R} \star \Sigma_t\|^2_{L^2([0,T]^2)}  \\
        &+ \|\Sigma_t \star  \hat{R}\|^2_{L^2([0,T]^2)} + \|\hat{R} \star \Sigma_t \star \hat{R}\|^2_{L^2([0,T]^2)} \Big).
    }    
      An application of the Cauchy-Schwarz inequality combined with Tonelli's theorem implies  that
    \bes{
        \label{eq:cs}
         \|K \star H \|_{L^2([0,T]^2)}     \leq& \|K \|_{L^2([0,T]^2)} \| H \|_{L^2([0,T]^2)}, \qquad K,H \in L^2([0,T]^2, \R^{N \times N}),
    }
so that 
    \bes{
        \| \tilde{\Sigma}_t \|^2_{L^2([0,T]^2)} \leq & 2^3 \left(  \|\Sigma_t\|^2_{L^2([0,T]^2)} + \|\hat{R} \star \Sigma_t\|^2_{L^2([0,T]^2)} + \|\Sigma_t \star  \hat{R}\|^2_{L^2([0,T]^2)} + \|\hat{R} \star \Sigma_t \star \hat{R}\|^2_{L^2([0,T]^2)} \right) \\
        \leq & 2^3  \|\Sigma_t\|^2_{L^2([0,T]^2)}\left(1 +  \|\hat{R}\|^2_{L^2([0,T]^2)} + \|\hat{R}\|^4_{L^2([0,T]^2)}\right)\\
        \leq & c  \|\Sigma_t\|^2_{L^2([0,T]^2)}\left(1 + \|\hat{R}\|^4_{L^2([0,T]^2)}\right),\\
    }
where $c>0$ is a constant independent of $\Sigma$ and $\hat R$. Thus, to obtain \eqref{eq:bound_sigma_tilde} it is enough to show that $ \|\hat{R}\|^2_{L^2([0,T]^2)} \leq\left(\frac{|f(\Theta)| \times \|K\|^{2 }_{L^2([0,T]^2)}}{ 1 -| f(\Theta)| \times \|K\|^{2 }_{L^2([0,T]^2)} } \right)^2 $. For this, note that applying successive Picard's iteration to $\hat{R} = \hat{K} + \hat{K} \star R $ yields 
    \bes{
    \label{eq:dev_serie}
    \hat{R}(s,u) =\sum_{n=1}^\infty \hat{K}^{\star n}(s,u) = \sum_{n=1}^\infty  (K f(\Theta))^{\star n}(s,u),
    }
    {where $\hat K^{\star n}$ is the $(n)$-fold $\star$-product of $\hat K$ by itself.}
    Combining  \eqref{eq:cs} and \eqref{eq:dev_serie} together with the submultiplicativity of the Frobenius norm  yields
    \bes{
         \|\hat{R}\|^2_{L^2([0,T]^2)} \leq & \sum_{1 \leq n,m \leq \infty} \int_0^T\int_0^T |\left(K(s,u)f(\Theta)\right)^{\star n}| |(K(s,u)f(\Theta))^{\star n}| ds du \\
         \leq & \sum_{1 \leq n,m \leq \infty}| f(\Theta)|^{n+m} \|K^{\star n}\|^2_{L^2([0,T]^2)}  \|K^{\star m}\|^2_{L^2([0,T]^2)} \\
         {=} & \left( \sum_{n=1}^\infty| f(\Theta)|^{n} \|K^{\star n}\|^2_{L^2([0,T]^2)}  \right)^2 \\
         \leq & \left( \sum_{n=1}^\infty| f(\Theta)|^{n} \|K\|^{2 n}_{L^2([0,T]^2)}  \right)^2 \\
         \leq & \left(\frac{|f(\Theta)| \times  \|K\|^{2 }_{L^2([0,T]^2)}}{ 1 -| f(\Theta)| \times \|K\|^{2 }_{L^2([0,T]^2)} } \right)^2.
    }
    This proves the desired inequality on $\hat{R}$ and the claimed inequality \eqref{eq:bound_sigma_tilde} follows.
\end{proof}

We can now complete the proof of Lemma~\ref{L:lemmacondtheta1}.

\vspace{1mm}

\noindent {\textit{Proof of Lemma \ref{L:lemmacondtheta1}.}}
	Fix $s\leq T$ and $\Theta \in \R^{d\times N}$. We first note that 
	\bes{ |\lambda_s|^2 + {\left|Z^1_s\right|^2 + \left|Z^2_s\right|^2} =&  |\Theta g_s(s)|^2 + 4\left|\left(\left( \boldsymbol{\Psi}_{s} \boldsymbol{K}\eta \right)^* g_s\right)(s)\right|^2.    
	}
	Using \ref{L:Psi}-\ref{L:Psi1}, and denoting by $\psi^{\rm op}_s$ the operator induced by the kernel $\psi_s$ there, we write
	\bes{
		\left|\left(\left( \boldsymbol{\Psi}_{s} \boldsymbol{K}\eta \right)^* g_s\right)(s)\right|^2 =& \;  |-((\Theta^\T \Theta \boldsymbol{K} \eta)^*g_s)(s) + ((\psi_s^{\rm op}\boldsymbol{K} \eta)^*g_s)(s)|^2 \\
		= & \; |\textbf{1} + \textbf{2} |^2 \\
		\leq & \;  2 (|\textbf{1}|^2 + |\textbf{2}|^2).
	}
	An application of the Cauchy-Schwarz inequality combined with   \eqref{eq:assumptionkerneldiff1}  leads to 
	\bes{
	    \label{eq:1_bold}
		|\textbf{1}|^2  
		=& \left|- \int_0^T \eta^T K(z,s)^\T \Theta^\T \Theta g_s(z) dz \right|^2 \leq  |\eta|^2 |\Theta \Theta^\T|^2 \sup_{u'\leq T}\int_0^T |K(z,{u'})|^2 du' \int_0^T |g_s(u)|^2 du.    }
	Similarly,
	\bes{
		|\textbf{2}|^2 =& 
		\left(\int_0^T \eta^\T \left(\int_0^T K(r,s)^\T \psi_s(r,z) dr \right) g_s(z) dz \right)^2 \\
		\leq & |\eta|^2 \left( \int_0^T \int_0^T |K(r,s)|^2|\psi_s(r,z)|^2 drdz \right) \left( \int_0^T |g_s(z)|^2 dz \right) \\
		\leq & |\eta|^2  \sup_{u'\leq T} \int_0^T |K(r, {u'})|^2 dr {\left(\int_0^T\int_0^T |\psi_s(r,z)|^2 dr dz\right)}   \left(\int_0^T |g_s(z)|^2 dz \right),
	}
	where we stress that $\psi_s$ is the only term on the right hand side depending on $\Theta$.
	Let us now  show that there exists a constant $c>0$ independant of $\Theta$ such that
	\bes{
		\label{eq:bound_psi_theta}
		\sup_{s \in [0,T]} \int_0^T \int_0^T |\psi_s(r,z)|^2 dr dz \leq  c|\Theta|^2(1 + |\Theta|^4 \hat \kappa(\Theta)),
	}
    where $\hat \kappa$ is defined  as in  \eqref{eq:def_k_hat}.
	Recall from \eqref{eq:boundary} that we have 
	\bes{
		\psi_t&= -{\hat{R}}^* \Theta^\top \Theta  -\Theta^\top \Theta {\hat{R}}- {\hat{R}}^* \Theta^\top \star R^{\Theta }_t  \Theta    - \Theta^\top   R^{\Theta }_t \star  \Theta  {\hat{R}} \\
		&\quad - {\hat{R}}^*  \Theta^\top \star R^{ \Theta }_t  \star \Theta  \hat{R} - {\hat{R}}^*\Theta^\top \star \Theta  {\hat{R}} - \Theta^\top R^{\Theta }_t  \Theta .
	}
	Thus,  recalling \eqref{hyp:R_hat},   there exists a constant $c>0$ independent of $\Theta$ such that
		\bes{
		\label{eq:bound_psi_theta_2}
		\sup_{s \in [0,T]} \int_0^T \int_0^T |\psi_s(r,z)|^2 dr dz \leq c |\Theta|^2 \left(1 + 	\sup_{t \in [0,T]} \int_0^T \int_0^T |R_t^\Theta(s,u)|^2 ds du \right).
	}
		To obtain \eqref{eq:bound_psi_theta},  
		it is enough to show that 
		\bes{
		\label{eq:bound_R}
			\sup_{t \in [0,T]} \int_0^T \int_0^T |R_t^\Theta(s,u)|^2 ds du \leq c  |\Theta|^4  (1 +  \hat \kappa(\Theta)), 
		}
 for some constant $c>0$ not depending on $\Theta$ and $\hat \kappa$ defined in \eqref{eq:def_k_hat}. For this recall that $\boldsymbol{R}_t^\Theta$ is the resolvent of $ -2 \Theta \boldsymbol{\tilde{\Sigma}}_t \Theta^\T$ which implies that $\boldsymbol{R}_t^{\Theta}=(\id+2\Theta \boldsymbol{\tilde \Sigma}_t \Theta^T)^{-1} - \id$.  Since, for each $t\leq T$, $\Theta \boldsymbol{\tilde \Sigma}_t \Theta
	^\T$ is a positive symmetric operator on $L^2([0,T], \R^d )$ induced by a continuous kernel, an application of Mercer's theorem, see \citet[Theorem 1, p.208]{shorack2009empirical}, yields the existence of  a countable orthonormal basis $(e_{t, \Theta}^n)_{n \geq 1}$ of $L^2([0,T], \R^d )$ such   that 
	\bes{
		2\Theta \tilde \Sigma_t(s,u) \Theta = \sum_{n\geq 1 } \lambda^{n}_{t,\Theta} e_{t, \Theta}^n(s)e_{t, \Theta}^n(u)^\T,
	}
	where $\lambda^n_{t,\Theta}\geq 0$, for all $n\geq 1$.
	Consequently
	\bes{
		R_t^\Theta(s,u) = \sum_{n \geq 1} \frac{-\lambda^{n}_{t,\Theta} }{ 1 + \lambda^{n}_{t,\Theta} }e_{t, \Theta}^n(s)e_{t, \Theta}^n(u)^\T, 
	}
	which yields 
	\begin{align}
	\int_0^T  \int_0^T |R^{\Theta}_t(s,u)|^2 ds du &= \displaystyle\sum_{n \geq 1} \frac{(\lambda^{n}_{t,\Theta})^2}{(1+\lambda^{n}_{t,\Theta})^2} \\
	&\leq \displaystyle\sum_{n \geq 1} (\lambda^{n}_{t,\Theta})^2  \; = \;  \int_0^T\int_0^T |2\Theta \Tilde \Sigma_t(s,u) \Theta^\T  |^2 ds du  \\
	&\leq 4 |\Theta|^4  \sup_{t\leq T}\int_0^T\int_0^T |\Tilde \Sigma_t(s,u)|^2 ds du \\
	& \leq {c  |\Theta|^4  (1 +  \hat \kappa(\Theta))},
	\end{align}
	{where the last inequality comes from Lemma \ref{l:bound_sig_bar}.  Consequently, inequality \eqref{eq:bound_R} combined with  \eqref{eq:bound_psi_theta_2} yield inequality \eqref{eq:bound_psi_theta}. Finally, the claimed  bound \eqref{eq:bound_lambda_plus_Z} follows by recollecting inequalities \eqref{eq:bound_psi_theta} and \eqref{eq:1_bold}. }

\subsection{Proof of Lemma~\ref{L:lemmacondtheta2}} \label{seclemma2}
\begin{proof}
{Recalling the decomposition \eqref{eq:decompbarsigma}, 
the process $Z$ admits the following Karhunen-Loeve representation 
    \bes{
    \label{eq:Z_decomp}
        Z(s,u) = \sum_{n \geq 1} \xi_n  e^n(s,u), \qquad s,u \in [0,T]^2,
    }
    where $(\xi_n)_{n\geq 1}$ is a sequence of independent  Gaussian random variables with mean $\mu_n = \langle \mu,e^n\rangle_{L^2([0,T]^2, \R^{2N})}$ and variance $\lambda^n$, for each  $n\in \mathbb N$.
    Now observe that the representation \eqref{eq:Z_decomp} combined with the orthogonality of $(e_n)_{n \geq 1}$ in $L^2([0,T]^2,\R^{2N})$ yields
    \bes{
    \label{eq:Z_norm}
     a\int_0^T\left( |g_s(s)|^2 +  \int_0^T |g_s(u)|^2 du \right)ds &= a\|  Z\|^2_{L^{2}([0,T]^2, \R^{2N})}= \sum_{n\geq 1} a \xi^2_n,
    }
   so that the independence of $(\xi_n)_{n \geq 1}$ leads to 
    \begin{align}
        \label{eq:tr_laplace}
        \E\left[ \exp \left(a \int_0^T \left(|g_s(s)|^2 +  \int_0^T |g_s(u)|^2 du\right)ds \right) \right] = & \E\left[ \exp \left(  \sum_{n \geq 1} a\xi_n^2 \right) \right]  \\
       = & \prod_{n \geq 1}  \E \left[\exp\left(a \xi_n^2 \right)\right]  \\
        = & \prod_{n \geq 1} \frac{e^{\frac{a\mu_n^2}{1 - 2 a\lambda^n}   }}{\sqrt{1 - 2 a\lambda^n}}
    \end{align}
    where the last equality follows from the fact that $\xi_n^2$ is chi-squared distributed and $0 < 1 - 2 a \lambda^1 < 1-2a \lambda^n$ by hypothesis. We now argue that the right hand side  of \eqref{eq:tr_laplace} is finite.  For the denominator, due to  $\sum_{n\geq 1} \lambda^n < \infty$, we obtain that $0<\prod_{n\geq 1}(1-2a\lambda^n)<\infty$. For the numerator, since $\lambda^n\to 0$, as $n\to \infty$, $\left(\frac{1}{1-2a\lambda^n}\right)_{n\geq 1}$ is uniformly bounded by a constant $c>0$  so that an application of Parseval's identity yields
	\begin{align}
	\prod_{n\geq 1}\exp\left(  \frac{a\mu_n^2}{1-2a\lambda^n}\right) & \leq \exp\left( c a \|\mu\|^2_{L^2([0,T]^2, \R^{2N})}  \right)\\
	&= \exp\left( c a  \left( \int_0^T \int_0^T \left(\frac{1}{T^2}|g_0(s)|^2 + |g_0(u)|^2 \right) dsdu \right)  \right)\\
	&= \exp\left( c a \left(T+\frac 1 T\right) \| g_0\|_{L^2([0,T],\R^N)}^2   \right)< \infty.
	\end{align}
	The proof is complete.
}
\end{proof}

\section{Additional proof for the martingale {property}}
For completeness, we adapt  \citet[Lemma 7.3]{AJLP17} to the multi-dimensional setting  to prove that the  local martingale
\begin{align} 
M_t &= \; M_0\, \mathcal  E\Big( -\int_0^t \sum_{i=1}^d \psi^i(T-s)\nu_i \sqrt{V^i_s}dW^i_s \Big). 
\end{align}
is a true martingale.  For this we set $U=\int_0^{\cdot} V_s ds$ and we observe that, thanks to stochastic Fubini's theorem, integrating \eqref{VolSqrt} yields
$$ U^i_t = \int_0^t g^i_0(s)ds + \int_0^t  K_i(t-s) Z^i_s ds    $$ 
with $$ Z^i_t = \int_0^t (DV_s)_i ds + \int_0^t \nu_i \sqrt{V_s^i}dW_s^i.$$
\begin{proof}
 Since $M$ is a nonnegative local martingale, it is a supermartingale by Fatou's lemma. Whence to obtain the true martingality it suffices to show that $\E[M_T]= 1$ for any $T\in\R_+$. To this end, fix $T>0$ and  define the stopping times $\tau_n=\inf\{t\ge0\colon \int_0^{t}V^i_s ds> n \mbox{ for some i}\leq d\}\wedge T$. Novikov's condition, recall that $\psi$ is bounded on $[0,T]$ being continuous, yields that   $M^{\tau_n}=M_{\tau_n \wedge \cdot }$ is a uniformly integrable martingale for each $n$. Whence,
\begin{align}
1=M^{\tau_n}_0 = \E_{\P}\left[ M^{\tau_n}_T \right]=
\E_{\P}\left[ M_T \bm 1_{\tau_n\geq T} \right]+ \E_{\P}\left[ M_{\tau_n} \bm 1_{\tau_n< T} \right],
\end{align}
where we made the dependence of the expectation on $\P$ explicit. 
Since $ \E_{\P}\left[ M_T \bm 1_{\tau_n\geq T} \right] \to \E_{\P}\left[ M_T  \right]$ as $n\to \infty$, by dominated convergence, in order to get that  $\E_{\P}[M_T]=1$, it suffices to  prove that 
\begin{align}\label{eq:girsanovproof1}
\E_{\P}\left[ M_{\tau_n} \bm 1_{\tau_n< T} \right]\to 0, \quad \mbox{as } n\to \infty. 
\end{align}
To this end, since $M^{\tau_n}$ is a martingale, we may define probability measures $\Q^n$ by
\[
\frac{d\Q^n}{d\P} = M^{\tau_n}_{\tau_n}.
\]
By Girsanov's theorem, the process $W^n=(W^{n,1},\ldots,W^{n,d})$ defined by
$$ W^{n,i}= W^i + \int_0^{\cdot} \bm 1_{s\leq \tau_n} \psi^i(T-s)\nu_i \sqrt{V^i_s}ds, \quad i=1,\ldots,d, $$ is a  Brownian motion under $\Q^n$. Furthermore, under $\mathbb Q^n$, we  have
\begin{equation} 
\begin{aligned}
U^i_t &= \int_0^t g^i_0(s)ds  + \int_0^t K_i(t-s)  Z^{n,i}_s ds  \\
Z^{n,i}_t&= \int_0^t ((D V_s)_i -  \bm 1_{s\leq \tau_n} \psi^i(T-s)\nu_i^2 {V^i_s}) ds  + \int_0^t \nu_i \sqrt{V^i_s}dW^{n,i}_s.
\end{aligned}
\end{equation}
and we observe that, due to the boundedness of $\psi$, the drift of $Z^n$ under $\Q^n$ satisfy a linear growth condition in $U$  for some constant $\kappa_{L}$ independent of $n$. An application of the generalized Gr\"onwall inequality for convolution equations would yield the moment bound
\[
\E_{\Q^n}[|U_T|^2 ] \le \eta(\kappa_L,T,K,g_0),
\] 
where $\eta(\kappa_L,T,K,g_0)$ does not depend on $n$, see for instance \citet[Lemma 3.1]{ajL1}.
We then get by an application of Chebyshev's inequality
\begin{align*}
\E_{\P}\left[ M_{\tau_n} \bm 1_{\tau_n\leq T} \right] &= 	 \Q^n(\tau_n< T) \\
&\leq \sum_{i=1}^d \Q^n\left( U^i_T>n \right)\\
&\le  \sum_{i=1}^d  \frac{1}{n^2} \E_{\Q^n}\left[|U^i_T|^2 \right] \\
&=   \frac{1}{n^2} \E_{\Q^n}\left[|U_T|^2 \right]\\ 
&\le \frac{1}{n^2} \eta(\kappa_L,T,K,g_0).
\end{align*}
Sending $n\to \infty$, we obtain \eqref{eq:girsanovproof1}, proving that $M$ is  martingale.	
\end{proof}

\vspace{5mm}

\small

\bibliographystyle{plainnat}
\bibliography{bibl1}
\end{document}